\documentclass[aop,preprint]{imsart}
\usepackage{exscale,amsmath,amssymb,amsthm,color,graphicx,accents,mathrsfs,mathtools,hyperref}
\usepackage{enumitem}
\usepackage{multicol,multirow,hhline,comment,url}
\usepackage{cellspace}		
\startlocaldefs

\numberwithin{equation}{section}
\numberwithin{figure}{section}
\numberwithin{table}{section}

\setattribute{keyword}{KWD}{Keywords:}

\allowdisplaybreaks[3]

\setlist[enumerate,1]{leftmargin=1cm}
\newcommand{\cev}[1]{\accentset{\leftharpoonup}{#1}}
\DeclareRobustCommand{\cev}[1]{\accentset{\leftharpoonup}{#1}}

\theoremstyle{plain}

\newtheorem{theorem}{Theorem}[section]

\newtheorem{proposition}[theorem]{Proposition}
\newtheorem{corollary}[theorem]{Corollary}
\newtheorem{lemma}[theorem]{Lemma}
\newtheorem{conjecture}{Conjecture}

\theoremstyle{definition}
\newtheorem{definition}[theorem]{Definition}

\theoremstyle{remark}
\newtheorem{remark}[theorem]{Remark}

\DeclareMathSymbol{\minus}{\mathbin}{AMSa}{"39}	

\newcommand{\ff}{\mathbf{f}}

\makeatletter
 \DeclareRobustCommand{\checkarg}{\@ifnextchar[{\@witharg}{}}
 \DeclareRobustCommand{\@witharg}[1][]{\ensuremath{\left(#1\right)}}
 
 \DeclareRobustCommand{\scaleGen}[1]{\@ifnextchar[{\@scalewithargs{#1}}{\odot^{}_{#1}}}
 \def\@scalewithargs#1[#2][#3]{#2 \odot^{}_{#1} #3}
 
 \DeclareRobustCommand{\blankBinOp}{\@ifnextchar[{\@blankbinopwithargs}{}}
 \def\@blankbinopwithargs[#1][#2]{#1 #2}
\makeatother

\def\IPspace{\mathcal{I}}
\def\dI{d_{\IPspace}}
\def\HIPspace{\IPspace_H}

\def\Exc{\mathcal{E}}

\def\mBxc{\nu_{\BESQ}}	
\def\mBxcA{\nu_{\BESQ}^{(-2\alpha)}}	

\def\cC{\mathcal{C}}	

\def\cCRI{\cC([0,\infty),\IPspace)}
\def\cCRIi{\cC([0,\infty),\IPspace_1)}

\def\cD{\mathcal{D}}

\def\mSxc{\nu_{\textnormal{stb}}}	
\def\mSxcA{\nu_{\textnormal{stb}}^{(\alpha)}}	

\def\mClade{\nu_{\textnormal{cld}}}	
\def\mCladeA{\nu_{\textnormal{cld}}^{(\alpha)}}	

\def\cT{\mathcal{T}}

\def\len{\textnormal{len}}			
\def\life{\zeta}					
\def\dis{\textnormal{dis}}			
\def\IPmag#1{\left\|\vphantom{I}#1\right\|}		

\def\skewer{\textsc{skewer}}		
\def\skewerP{\widebar{\skewer}}		

\def\cutoffL#1#2{\textsc{cutoff}^{\leq #1}_{#2}}
\def\cutoffG#1#2{\textsc{cutoff}^{\geq #1}_{#2}}
\def\cutoffLB#1#2{\textsc{cutoff}^{\leq #1}_{#2}}
\def\cutoffGB#1#2{\textsc{cutoff}^{\geq #1}_{#2}}

\def\Dirac#1{\delta\left( #1 \right)}
\def\DiracBig#1{\delta\big( #1 \big)}	

\def\reverse{\mathcal{R}}							
\def\reverseexc{\reverse_{\textnormal{spdl}}}		
\def\reverseH{\reverse_{\textnormal{cld}}}			

\def\scaleB{\scaleGen{\textnormal{spdl}}}		
\def\scaleH{\scaleGen{\textnormal{cld}}}		

\def\scaleI{\blankBinOp}	

\def\ShiftRestrict#1#2{#1\big|^{\from}_{#2}} 
\def\shiftrestrict#1#2{#1|^{\from}_{#2}}

\def\Restrict#1#2{#1\big|_{#2}}
\def\restrict#1#2{#1|_{#2}}

\def\Concat{ \mathop{ \raisebox{-2pt}{\Huge$\star$} } }
\def\ConcatIL{ \mbox{\huge $\star$} }
\def\concat{\star}

\def\bN{\mathbf{N}}			
\def\bF{\mathbf{F}}			
\def\bX{\mathbf{X}}			
\newcommand{\IPLT}{\mathscr{D}}
\def\bM{\mathbf{M}}			
\def\bff{\mathbf{f}}		

\newcommand{\td}[1]{\widetilde{#1}}

\newcommand{\wh}[1]{\widehat{#1}}
\def\whN{\widehat{\mathbf{N}}}
\def\whF{\widehat{\mathbf{F}}}
\def\whX{\widehat{\bX}}

\newcommand{\ol}[1]{\widebar{#1}}
\def\olN{\widebar{\mathbf{N}}}
\def\olF{\widebar{\mathbf{F}}}
\def\olX{\widebar{\bX}}

\def\oll{\widebar{\ell}}

\def\BR{\mathbb{R}}				
\def\BN{\mathbb{N}}				
\def\Leb{\textnormal{Leb}}		

\def\to{\rightarrow}
\def\downto{\downarrow}
\def\upto{\uparrow}
\def\from{\leftarrow}

\def\cf{\mathbf{1}}				
\def\Pr{\mathbf{P}}				
\def\BPr{\mathbb{P}}			
\def\bQ{\mathbf{Q}}				
\def\EV{\mathbf{E}}				
\def\cF{\mathcal{F}}			

\def\cA{\mathcal{A}}	

\def\distribfont#1{\texttt{\upshape #1}}

\def\ExpDist{\distribfont{Exponential}\checkarg}
\def\GammaDist{\distribfont{Gamma}\checkarg}
\def\InvGammaDist{\distribfont{InverseGamma}\checkarg}
\def\GeomDist{\distribfont{Geometric}\checkarg}
\def\BetaDist{\distribfont{Beta}\checkarg}

\def\PoiDir{\distribfont{PD}\checkarg}
\def\PoiDirAT{\PoiDir[\alpha,\theta]}
\def\PRM{\distribfont{PRM}\checkarg}

\def\Stable{\distribfont{Stable}\checkarg}
\def\BESQ{\distribfont{BESQ}\checkarg}
\newcommand{\StableA}{\distribfont{Stable}\ensuremath{(1\!+\!\alpha)}}

\newcommand{\BESQA}{\distribfont{BESQ}\ensuremath{(-2\alpha)}}
\def\PDIP{\distribfont{PDIP}\checkarg}

\newcommand{\vecc}[1]{\accentset{\rightharpoonup}{#1}}

\makeatletter
\let\save@mathaccent\mathaccent
\newcommand*\if@single[3]{%
  \setbox0\hbox{${\mathaccent"0362{#1}}^H$}%
  \setbox2\hbox{${\mathaccent"0362{\kern0pt#1}}^H$}%
  \ifdim\ht0=\ht2 #3\else #2\fi
  }
\newcommand*\rel@kern[1]{\kern#1\dimexpr\macc@kerna}
\newcommand{\widebar}{}
\DeclareRobustCommand*\widebar[1]{\@ifnextchar^{\wide@bar{#1}{0}}{\wide@bar{#1}{1}}}
\newcommand*\wide@bar[2]{\if@single{#1}{\wide@bar@{#1}{#2}{1}}{\wide@bar@{#1}{#2}{2}}}
\newcommand*\wide@bar@[3]{%
  \begingroup
  \def\mathaccent##1##2{%
    \let\mathaccent\save@mathaccent
    \if#32 \let\macc@nucleus\first@char \fi
    \setbox\z@\hbox{$\macc@style{\macc@nucleus}_{}$}%
    \setbox\tw@\hbox{$\macc@style{\macc@nucleus}{}_{}$}%
    \dimen@\wd\tw@
    \advance\dimen@-\wd\z@
    \divide\dimen@ 3
    \@tempdima\wd\tw@
    \advance\@tempdima-\scriptspace
    \divide\@tempdima 10
    \advance\dimen@-\@tempdima
    \ifdim\dimen@>\z@ \dimen@0pt\fi
    \rel@kern{0.6}\kern-\dimen@
    \if#31
      \overline{\rel@kern{-0.6}\kern\dimen@\macc@nucleus\rel@kern{0.4}\kern\dimen@}%
      \advance\dimen@0.4\dimexpr\macc@kerna
      \let\final@kern#2%
      \ifdim\dimen@<\z@ \let\final@kern1\fi
      \if\final@kern1 \kern-\dimen@\fi
    \else
      \overline{\rel@kern{-0.6}\kern\dimen@#1}%
    \fi
  }%
  \macc@depth\@ne
  \let\math@bgroup\@empty \let\math@egroup\macc@set@skewchar
  \mathsurround\z@ \frozen@everymath{\mathgroup\macc@group\relax}%
  \macc@set@skewchar\relax
  \let\mathaccentV\macc@nested@a
  \if#31
    \macc@nested@a\relax111{#1}%
  \else
    \def\gobble@till@marker##1\endmarker{}%
    \futurelet\first@char\gobble@till@marker#1\endmarker
    \ifcat\noexpand\first@char A\else
      \def\first@char{}%
    \fi
    \macc@nested@a\relax111{\first@char}%
  \fi
  \endgroup
}
\makeatother

\newcommand{\bP}{\mathbb{P}}

\newcommand{\cE}{\mathcal{E}}
\newcommand{\cI}{\mathcal{I}}

\newcommand{\fN}{\mathbf{N}}

\newcommand{\fX}{\mathbf{X}}

\newcommand{\ed}{\mbox{$ \ \stackrel{d}{=}$ }}

\newcommand{\nbeta}{\beta}

\endlocaldefs

\begin{document}

\begin{frontmatter}

\title{Diffusions on a space of interval partitions:\\ Poisson--Dirichlet stationary distributions\thanksref{T0}}

\runtitle{Poisson--Dirichlet interval partition diffusions}
\runauthor{Forman, Pal, Rizzolo and Winkel}

\thankstext{T0}{This research is partially supported by NSF grants {DMS-1204840, DMS-1308340, DMS-1612483, DMS-1855568}, UW-RRF grant A112251, EPSRC grant EP/K029797/1.}

\begin{aug}
  \author{\fnms{Noah} \snm{Forman}\thanksref{m4}\ead[label=e1]{noah.forman@gmail.com}},
  \author{\fnms{Soumik} \snm{Pal}\thanksref{m2}\ead[label=e2]{soumikpal@gmail.com}},
  \author{\fnms{Douglas} \snm{Rizzolo}\thanksref{m3}\ead[label=e3]{drizzolo@udel.edu}},\\ and 
  \author{\fnms{Matthias} \snm{Winkel}\thanksref{m1}\ead[label=e4]{winkel@stats.ox.ac.uk}}
  
  \affiliation{McMaster University\thanksmark{m4}, University of Washington\thanksmark{m2},\\ University of Delaware\thanksmark{m3}, University of Oxford\,\thanksmark{m1}}
  
  \address{Department of Mathematics \& Statistics\\ McMaster University\\ 1280 Main Street West\\ Hamilton, Ontario L8S 4K1\\ Canada\\
   \printead{e1}
  }
  
  \address{Department of Mathematics \\ University of Washington\\ Seattle WA 98195\\ USA\\
   \printead{e2}
  }
  
  \address{Department of Mathematical Sciences\\ University of Delaware\\ Newark DE 19716\\ USA\\
   \printead{e3}
  }
  
  \address{Department of Statistics\\ University of Oxford\\ 24--29 St Giles'\\ Oxford OX1 3LB\\ UK\\
   \printead{e4}
  }
\end{aug} 
  
 \begin{abstract}
  We introduce diffusions on a space of interval partitions of the unit interval that are stationary with the Poisson--Dirichlet laws with parameters $\left(\alpha,0\right)$ and $\left(\alpha,\alpha \right)$. The construction has two 
steps. The first is a general construction of interval partition processes obtained previously, by decorating the jumps of a L\'evy process with independent excursions. Here, we focus on the second step, which requires explicit
transition kernels and what we call pseudo-stationarity. This allows us to study processes obtained from the 
original construction via scaling and time-change. 
 In a sequel paper, we establish connections to diffusions on decreasing sequences introduced by Ethier and Kurtz (1981) and Petrov (2009). The latter diffusions are continuum limits of up-down Markov chains on Chinese restaurant processes. 
 Our construction is also a step towards resolving longstanding conjectures by Feng and Sun on measure-valued Poisson--Dirichlet diffusions, and by Aldous on a continuum-tree-valued diffusion.
 \end{abstract}



\end{frontmatter}

\section{Introduction}
\label{sec:intro}

The two-parameter Poisson--Dirichlet distributions $(\PoiDirAT,\,\alpha\in [0,1),\,\theta>-\alpha)$ \cite{PitmYorPDAT} are a family of laws on the Kingman simplex: the set of non-increasing sequences of reals that sum to 1. This family rose to prominence in applications following the work of Ishwaran and James \cite{IshwJame01}, subsequently becoming a standard distribution used in non-parametric Bayesian clustering models \cite{BrodJordPitm13}.  The two-parameter family extends the one-parameter family of Kingman \cite{Kingman75}, $\PoiDir[\theta] := \PoiDir[0,\theta]$, which was originally studied as a model for allele frequencies.  

Diffusive models for the fluctuation of allele frequencies over time were considered by Ethier and Kurtz \cite{EthiKurt81}, who devised ``infinitely-many-neutral-alleles'' diffusions on the Kingman simplex with $\PoiDir[\theta]$ stationary distributions, for $\theta>0$.  Petrov \cite{Petrov09} extended these diffusions to the two-parameter setting.  The purpose of this paper is to initiate the development of an analogous family of diffusions for $\alpha\in (0,1), \theta\geq 0$ whose stationary distributions are the interval partitions obtained by ordering the components of a $\PoiDir[\alpha, \theta]$-distributed random variable in their unique random regenerative order \cite{GnedPitm05}.  We will refer to the left-to-right reversal of these interval partitions as $(\alpha,\theta)$-Poisson--Dirichlet Interval Partitions, whose distribution will be denoted $\PDIP[\alpha,\theta]$. 

\begin{definition}\label{def:IP_1}
 An \emph{interval partition} is a set $\beta$ of disjoint, open subintervals of some finite real interval $[0,M]$, that cover $[0,M]$ up to a Lebesgue-null set. We write $\IPmag{\beta}$ to denote $M$. We refer to the elements of an interval partition as its \emph{blocks}. The Lebesgue measure $\Leb(U)$ of a block $U\in\beta$ is called its \emph{mass}. We denote the empty interval partition by $\emptyset$.\vspace{-0.2cm}
\end{definition}

Regenerative ordering arose early in the study of the $\PoiDir[\alpha, \theta]$ family due to the following observation \cite{Pitman97}: if $\mathcal{Z}_\alpha$ is the zero set of a $(2\!-\!2\alpha)$-dimensional Bessel bridge (or Brownian bridge, if $\alpha=\frac12$) then $\mathcal{Z}^c_\alpha$ can be written as a countable union of disjoint open intervals that comprise a regenerative interval partition in the sense of \cite{GnedPitm05,Taksar80}, with the ranked sequence of their lengths having $\PoiDir[\alpha, \alpha]$ distribution.  The law of this collection of open intervals is the $\PDIP[\alpha,\alpha]$ distribution.  A similar relation holds between the complement of the zero set of a $(2\!-\!2\alpha)$-dimensional Bessel process run for time $1$ and the $\PoiDir[\alpha, 0]$-distribution, although for our purposes we define the \PDIP[\alpha,0] as the left-to-right reversal of the resulting partition. Our interest in the regenerative order comes from its connection with continuum random trees \cite{PitmWink09} which shows, for example, how Aldous's Brownian Continuum Random Tree can be constructed from an i.i.d.\ sequence of $\PDIP[\frac12,\frac12]$.

Here is another well-known construction of a \PDIP[\alpha,\alpha]: begin with a \Stable[\alpha] subordinator $Y\! =\! (Y_t,t\!\ge\!0)$, with Laplace exponenet $\Phi(\lambda)\! =\! \lambda^\alpha$. Let $S\sim\ExpDist[r]$ independent of $Y$, and consider the complement of the range $B = \{Y_t,t\!\ge\!0\} \cap [0,S]$ as an interval partition of $[0,\sup(B)]$. I.e.\ we are looking at the jump intervals of the subordinator prior to exceeding level $Y$. We normalize this to a partition of $[0,1]$ by dividing the endpoints of each block by $\sup(B)$; the resulting partition is a \PDIP[\alpha,\alpha] and is independent of $\sup(B)$, which has law \GammaDist[\alpha,r]. See Proposition \ref{prop:PDIP} for details.

In Section \ref{sec:mainresults} we specify transition kernels for the diffusions that we will study and we state our main results. 
In fact, these diffusions arise from a construction in \cite{IPPA}, explained in Section \ref{IPPAsection}; however, once we have shown that said constructions yield the claimed transition kernels, we can study the diffusions directly via these kernels.
%
Further motivation via connections to long-standing conjectures is given in Section \ref{sec:intro:AD}.

\subsection{Main results}\label{sec:mainresults}

Fix $\alpha\in(0,1)$. We first specify a transition semigroup $(\kappa_y^{(\alpha)},y\!\ge\! 0)$ on interval partitions that satisfies a \em branching property\em: given any interval partition $\beta$, the blocks $U\in\beta$ will give rise to \em independent \em interval partitions $\gamma_U$ (possibly empty) at 
time $y$, and $\kappa_y^{(\alpha)}(\beta,\,\cdot\,)$ will be the distribution of their concatenation. See Figure \ref{fig:semigroup}. 

Let us formalize this concatenation. We call the family $(\gamma_U)_{U\in\beta}$ \emph{summable} if $\sum_{U\in\beta}\IPmag{\gamma_U} < \infty$. We 
then define
 $S(U) := \sum_{U^\prime=(u^\prime,v^\prime)\in\beta\colon u^\prime<u}\IPmag{\gamma_{U^\prime}}$ for $U\!=\!(u,v)\!\in\!\beta$,
 and the \emph{concatenation}\vspace{-0.1cm}
\begin{figure}[t]
 \centering
 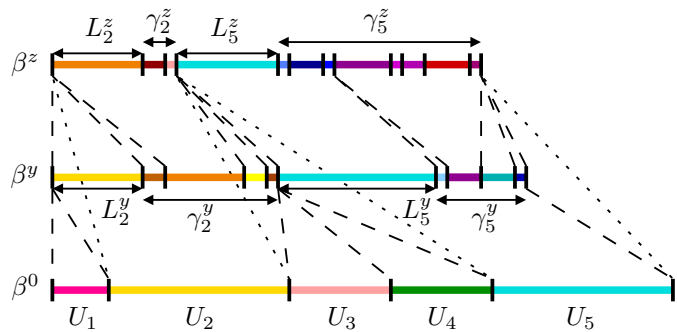
 \caption{Illustration of the transition kernel $\kappa_y$: $\beta^0$ has five blocks $U_1,\ldots,U_5$; some blocks contribute $\emptyset$ for time $y$, here
   $U_1$, $U_3$ and $U_4$, others non-trivial partitions, here $U_2$ and $U_5$, hence $\beta^y=(0,L_2^y)\concat\gamma_2^y\concat(0,L_5^y)\concat\gamma_5^y$. Semigroup property requires consistency of dotted transition from 0 to $z$ 
   and composition of dashed transitions from $0$ to $y$ and from $y$ to $z$.\label{fig:semigroup}\vspace{-0.1cm}}
\end{figure}
 \begin{equation}\label{eq:IP:concat_def}
  \Concat_{U\in\beta}\gamma_U := \{(S(U)+x,S(U)+y)\colon\ U\in\beta,\ (x,y)\in \gamma_U\}.\vspace{-0.2cm}
 \end{equation} 
We also write $\gamma\concat\gamma^\prime$ to concatenate two interval partitions. 
For $c\ge0$ let $c\gamma$ denote the interval partition obtained by multiplying each block in $\beta$ by $c$.

Now, fix $b,r>0$. Let $B_{r}^{(\alpha)}\sim \GammaDist[\alpha,r]$, $\bar\gamma^{(\alpha)}\sim \PDIP[\alpha,\alpha]$, and let $L_{b,r}^{(\alpha)}$ be an $(0,\infty)$-valued random variable with Laplace transform\vspace{-0.1cm}
\begin{equation}
  \EV\left[e^{-\lambda L_{b,r}^{(\alpha)}}\right] = \left(\frac{r+\lambda}{r}\right)^{\alpha}\frac{e^{br^2/(r+\lambda)}-1}{e^{br}-1},
  \label{LMBintro}\vspace{-0.1cm}
\end{equation}
all assumed to be independent. Let
\begin{equation}\label{eq:mu}
 \mu_{b,r}^{(\alpha)}(\,\cdot\,) = e^{-br}\delta_\emptyset(\,\cdot\,) + (1-e^{-br})\Pr\big\{\big(0,L_{b,r}^{(\alpha)}\big) \concat B_{r}^{(\alpha)}\bar\gamma^{(\alpha)}\in\cdot\,\big\}.
\end{equation}
For $y\ge0$ and $\beta$ any interval partition, let $\kappa_y^{(\alpha)}(\beta,\cdot)$ denote the law of\vspace{-0.1cm}
\begin{equation}
 \label{eq:intro:transn}
 \Concat_{U\in\beta}\gamma_U \quad \text{where} \quad \gamma_U\sim \mu_{\Leb(U),1/2y}^{(\alpha)}\text{ independently for each }U\in\beta.
\end{equation}
In fact, it is easily checked that a.s.\ only finitely many of the $\gamma_U$ are non-empty; cf.\ \cite[Lemma 6.1]{IPPA}. Thus, \eqref{eq:intro:transn} describes a concatenation of just finitely many partitions of the form $(0,L)\concat (B\bar\gamma)$, each comprising a leftmost block of mass $L$ followed by a randomly scaled \PDIP[\alpha,\alpha].

Adapting Lamperti \cite{Lamperti72}, we say that a Markov process $(\beta^y,y\!\ge\! 0)$ is 1-self-similar if $(c\beta^{y/c},y\ge 0)$ has the same semigroup as $(\beta^y,y\!\ge\! 0)$.

\begin{theorem}\label{thmtype1Hunt} Let $\alpha\!\in\!(0,1)$. Then the maps $\beta\mapsto \kappa_y^{(\alpha)}(\beta,\cdot\,)$, $y\ge 0$, are weakly continuous and form the transition semigroup of a 1-self-similar path-continuous Hunt process $(\beta^y,\,y\ge 0)$ on a space $(\IPspace,\dI)$ of interval partitions (defined in Definition \ref{def:IP:metric}).  
\end{theorem}

We will refer to such processes as \emph{type-1 evolutions}. The reader may wonder why we chose \eqref{LMBintro}-\eqref{eq:intro:transn}. 
On the one hand, the choice of \eqref{LMBintro} is constrained by the semigroup property, including the branching property; and on the other, it will fit neatly into place in our stationarity computations. 
To prove Theorem \ref{thmtype1Hunt}, we will show that this semigroup belongs to a class of interval partition evolutions (\em IP-evolutions\em) introduced in 
\cite{IPPA} by a Poissonian construction that reveals the branching property but otherwise leaves the semigroup rather implicit. 

As each \PDIP[\alpha,\alpha] has infinitely many blocks, so too do the type-1 evolutions $(\beta^y,y\ge 0)$ of Theorem \ref{thmtype1Hunt} on the event $\{\beta^y\neq\emptyset\}$. There is no rightmost block but rather $\beta^y$ has infinitely many blocks to the right of $\IPmag{\beta^y}-\epsilon$, for every $\epsilon\!>\!0$. 
 However, as only finitely many of the $\gamma_U$ in \eqref{eq:intro:transn} are non-empty, $\beta^y$ comprises a finite alternating sequence of leftmost blocks and (reversible) rescaled \PDIP[\alpha,\alpha] with no right or leftmost blocks. In particular, $\beta^y$ has a leftmost block when $\beta^y\neq\emptyset$. It is natural to then consider a related kernel that begins with an additional \PDIP\ component: let $\widetilde{\kappa}_y^{(\alpha)}(\beta,\cdot)$ denote the distribution of
\begin{equation}\label{eq:intro:transn_0}
 B\bar\gamma\concat\Concat_{U\in\beta}\gamma_U,
\end{equation}
where $B \sim \GammaDist[\alpha,1/2y]$, $\bar\gamma\sim \PDIP[\alpha,\alpha]$, and the $\gamma_U$, $U\in\beta$, are as in \eqref{eq:intro:transn}, all jointly independent.
\begin{theorem}\label{thmtype0Hunt} The conclusions of Theorem \ref{thmtype1Hunt} also hold for the family $\big(\widetilde{\kappa}_y^{(\alpha)},\,y\ge 0\big)$, for each $\alpha\in(0,1)$.
\end{theorem}

We will refer to such processes as \em type-0 evolutions\em. We will present all further developments in parallel for type-1 and type-0 evolutions
. Let us start with total mass processes.

Squared Bessel processes are 1-self-similar $[0,\infty)$-valued diffusions with
$$dZ(y)=\delta dy+2\sqrt{Z(y)}dB(y),\quad Z(0)=b\ge 0,\quad 0<y<\zeta,$$
where $\delta\in\BR$ is a real parameter, $(B(y),y\ge 0)$ is Brownian motion and $\zeta=\infty$ if $\delta>0$ while $\zeta=\inf\{y\ge 0\colon Z(y)=0\}$ for $\delta\le 0$. We set $Z(y)=0$ for $y\ge\zeta$. See \cite[Chapter XI]{RevuzYor}. We say $Z$ is $\BESQ[\delta]$ starting from $b$.
The \BESQ[0] diffusion, also called the Feller diffusion, is well-known to satisfy the additivity property that the sum of independent \BESQ[0] diffusions from any initial conditions is still \BESQ[0]. 
A \BESQ[\delta] diffusion for $\delta\!>\!0$ ($\delta\!=\!0$) can be viewed as branching processes with (without) immigration \cite{KawaWata71}.

\begin{theorem}\label{thm:BESQ_total_mass}
  Consider a type-1 evolution, respectively a type-0 evolution, $\left( \beta^y,\; y\ge 0  \right)$ with $\beta^0 \in \IPspace$. Then the total mass 
  process $\left(\IPmag{\beta^y},\,y\ge0\right)$ is a \BESQ[0] diffusion, respectively a \BESQ[2\alpha] diffusion, starting from $\IPmag{\beta^0}$.
\end{theorem}

Neither \BESQ[0] nor \BESQ[2\alpha] admit stationary distributions, so neither do the type-1 or type-0 evolutions themselves. However, we do have the
following ``pseudo-stationarity'' result. 

\begin{theorem}\label{thm:pseudostat}
 Let $\ol\beta\sim\PDIP[\alpha,0]$, respectively \PDIP[\alpha,\alpha], and, independently, $Z\sim\BESQ[0]$, respectively \BESQ[2\alpha], with an arbitrary  initial distribution. Let $(\nbeta^y,\,y\geq 0)$ be a type-1 evolution, respectively \vspace{-0.1cm} type-0 evolution, with $\nbeta^0\stackrel{d}{=}Z(0)\ol\beta$. Then for each fixed $y\ge 0$ we have $\nbeta^y \stackrel{d}{=} Z(y)\ol\beta$.
\end{theorem}

To obtain stationary diffusions on partitions of the unit interval, 
our third and fourth families of IP-evolutions, we employ a procedure that we call \textit{de-Poissonization}. 
Consider the total mass process $\left(\IPmag{\beta^y},\; y\ge 0 \right)$ from Theorem \ref{thm:BESQ_total_mass} and the time-change\vspace{-0.2cm}
\begin{equation}\label{eq:de_Pois_tc_intro}
 \rho(u) := \inf\left\{  y\ge 0:\;  \int_0^y \IPmag{\beta^z}^{-1} dz > u   \right\}, \quad u \ge 0.\vspace{-0.1cm}
\end{equation}

\begin{theorem}\label{thm:stationary}
 Let $(\beta^y,\,y\ge0)$ denote a type-1 evolution, respectively a type-0 evolution, with initial state $\beta^0 \in \IPspace\setminus\{\emptyset\}$. Then the process\vspace{-0.2cm}
$$
  (\ol\beta^u,\,u\ge0) := \left(  \IPmag{\beta^{\rho(u)}}^{-1} \beta^{\rho(u)},\; u \ge 0  \right)\vspace{-0.2cm}
$$
 is a path-continuous Hunt process on $(\IPspace_1,\dI)$ where $\IPspace_1 := \{\beta\!\in\!\IPspace\colon \IPmag{\beta}\!=\!1\}$, with a stationary distribution given by $\PDIP[\alpha,0]$ respectively $\PDIP[\alpha, \alpha]$.
\end{theorem}

In light of this, we will refer to the de-Poissonized evolutions $(\ol\beta^u,u\!\ge\!0)$ as $\left(\alpha,0\right)$- and  \em $\left(\alpha,\alpha\right)$-IP-evolutions\em.  
%
%
In a sequel paper \cite{IPPPetrov}, we show that the process of ranked block sizes of an $\left(\alpha,0\right)$- or 
$\left(\alpha,\alpha\right)$-IP evolution is a diffusion introduced by Petrov \cite{Petrov09}, extending Ethier and Kurtz \cite{EthiKurt81}, and 
further studied and ramified in \cite{RuggWalk09,RuggWalkFava13,Ruggiero14,FengSun10,FengSunWangXu11}. Indeed, the viewpoint of the richer 
IP-evolution in $(\IPspace,\dI)$ gives insight into the creation of blocks and the evolution of associated quantities such as the 
$\alpha$-diversity process in their diffusions on decreasing sequences. This is of interest in allele-frequency models with
infinitely many types \cite{RuggWalkFava13}. 

Lamperti \cite{Lamperti67} used the time-change \eqref{eq:de_Pois_tc_intro} to construct continuous-state branching processes from 
L\'evy processes. When applied to Brownian motion, the diffusion coefficient becomes state-dependent, linear in population size. When applied to
\BESQ, the diffusion coefficient becomes quadratic, and while this is wiped out when scaling by population size, the effect on 
the scaled interval partition is that blocks behave like coupled Jacobi diffusions. Indeed, in \cite{WarrYor98,Pal11}, Jacobi and Wright--Fisher 
diffusions are obtained as a de-Poissonization of a vector of independent \BESQ\ processes via the same time-change. In that paper, the sum of the 
\BESQ\ processes turns out to be independent of the de-Poissonized process. See also \cite{Pal13}. 

\begin{conjecture}\label{conj:mass_struct}
 Consider a type-1 or type-0 evolution $(\beta^y,\,y\ge0)$ starting from $\beta\in\IPspace$. Its total mass process $(\IPmag{\beta^y},\,y\ge 0)$,  
 as in Theorem \ref{thm:BESQ_total_mass}, is independent of its de-Poissonization $(\ol\beta^u,\,u\ge 0)$, as in Theorem \ref{thm:stationary}.
\end{conjecture}

See Theorem \ref{thm:pseudostat_strong} for a weaker result.

\subsection{Construction of interval partition evolutions following \cite{IPPA}}\label{IPPAsection}

In \cite{IPPA}, we gave a general construction of processes in a space of interval partitions based on spectrally positive L\'evy processes
(\em scaffolding\em) whose point process of jump heights (interpreted as lifetimes of individuals) is marked by excursions (\em spindles\em, giving ``sizes'' varying during the lifetime). Informally, the IP-evolution, indexed by level, considers for each level $y\ge 0$
the jumps crossing that level and records for each such jump an interval whose length is the ``size'' of the individual (width of the spindle) when crossing that level, ordered from left to right without leaving gaps. This construction and terminology is illustrated in Figure \ref{fig:skewer_1}; we inscribe the corresponding spindle vertically into each jump, depicted as a laterally symmetric shaded blob.

\begin{figure}[t]
 \centering
 \input{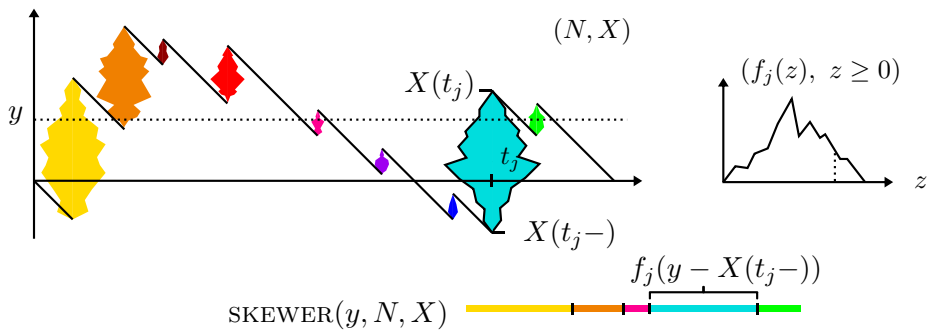}
 \caption{Left: The slanted black lines comprise the graph of the scaffolding $X$. Shaded blobs decorating jumps describe the corresponding spindles: points $(t_j,f_j)$ of $N$. Right: Graph of one spindle. Bottom: A skewer, with blocks shaded to correspond to spindles; not to scale.\label{fig:skewer_1}\vspace{-0.2cm}}
\end{figure}

Specifically, if $N=\sum_{i\in I}\delta(t_i,f_i)$ is a point process of excursions $f_i$ at times $t_i\in[0,T]$ of excursion lengths $\zeta_i$ (spindle heights), 
and $X$ is a real-valued process with jumps $\Delta X(t_i):=X(t_i)-X(t_i-)=\zeta_i$ at times $t_i$, $i\in I$, we define the interval partition $\skewer(y,N,X)$ at level $y$, as follows.

\begin{definition}\label{def:skewer}
 For $y\in\BR$, $t\in [0,T]$, the \emph{aggregate mass} in $(N,X)$ at level $y$, up to time $t$ is\vspace{-0.3cm}
 \begin{equation}
  M_{N,X}^y(t) := \sum_{i\in I\colon t_i\le t}f_i(y - X(t_i-)).\label{eq:agg_mass_from_spindles}\vspace{-0.3cm}
 \end{equation}
 The \emph{skewer} of $(N,X)$ at level $y$, denoted by $\skewer(y,N,X)$, is defined as\vspace{-0.3cm}
 \begin{equation}
   \left\{\left(M^y_{N,X}(t-),M^y_{N,X}(t)\right)\!\colon t\in [0,T],\ M^y_{N,X}(t-) < M^y_{N,X}(t)\right\}\label{eq:skewer_def}\vspace{-0.3cm}
 \end{equation}
 and the \emph{skewer process} as $\skewerP(N,X) := \big( \skewer(y,N,X),y\!\geq\! 0\big)$. 
\end{definition}

Let $\fN$ denote a Poisson random measure (\PRM) with intensity measure $\Leb\otimes\nu$ (denoted by \PRM[\Leb\otimes\nu]), where $\nu$ is the Pitman--Yor excursion measure \cite{PitmYor82} of a $[0,\infty)$-valued diffusion. Let $\fX$ be an associated L\'evy process, with its \PRM\ of jumps equaling the image of $\fN$ under the map from excursions to lifetimes, stopped at a stopping time $T$. In \cite{IPPA}, we established criteria on $\nu$ and $T$ under which $\skewerP(\fN,\fX)$ is a diffusion, i.e.\ a path-continuous strong Markov process. In this IP-evolution, each interval length (block) evolves independently according to the $[0,\infty)$-valued diffusion, which we call the \em block diffusion\em, while in between the (infinitely many) blocks, new blocks appear at times equal to the pre-jump levels of $\fX$. 


This IP-evolution has a particular initial distribution. In Definition \ref{constr:type-1} we formalize a construction to start this diffusion from any $\beta\in\IPspace$. Informally, we define 
$(\bN_{\beta},\bX_{\beta})$ as follows. For each interval $V\in\beta$ let $\bff_V$ denote a block diffusion starting from $\Leb(V)$ and killed 
upon hitting $0$. We denote its lifetime by $\life(\bff_V)$. Let $\bX_V$ be an independent L\'evy process starting from $\life(\bff_V)$, killed 
upon hitting $0$. We form $\bN_V$ by marking jumps of $\bX_V$ with independent excursions drawn from $\nu$, conditioned to have lifetimes equal to corresponding jump heights, and we mark with $\ff_V$ the jump of height $\life(\bff_V)$ at time 0. We do this independently for each $V\in\beta$. Scaffolding $\bX_\beta$ and point measure $\bN_\beta$ are formed by concatenating $\bX_V$ or $\bN_V$, $V\in\beta$. A continuous version of 
$\skewerP(\bN_{\beta},\bX_{\beta})$ is an IP-evolution starting from $\beta$.

\begin{theorem}\label{thm:diffusion_0}
 When the block diffusion is \BESQA\ and the scaffolding is built from spectrally positive \StableA\ L\'evy processes, the IP-evolution constructed as above is the type-1 evolution of Theorem \ref{thmtype1Hunt}.
\end{theorem}

Theorems \ref{thm:diffusion_0} and \ref{thm:BESQ_total_mass} together can be viewed as a Ray--Knight theorem for a discontinuous L\'evy process. 
The local time of the stopped 
L\'evy process is not Markov in level \cite{EiseKasp93}, but our marking of jumps and skewer map fill in the missing information about jumps to construct a larger Markov process.  Moreover, the local time of the L\'evy process can be measurably recovered from the skewer process; see \cite[Theorem 37]{Paper0}. The appearance of \BESQ[0] total mass is an additional connection to the second Brownian Ray--Knight theorem \cite[Theorem XI.(2.3)]{RevuzYor}, in which local time evolves as \BESQ[0].

It is well-known (see \cite{PitmYor82,ShigaWata73}) that $\BESQ$ processes of nonnegative dimensions satisfy an additivity property. This does not extend to negative dimensions (see \cite{PitmWink18}, however). Theorem \ref{thm:BESQ_total_mass} states that the sum of countably many \BESQA\ excursions anchored at suitably random positions on the time axis gives a \BESQ[0] process. This can be interpreted as an extension of the additivity of $\BESQ$ processes to negative dimensions.

\subsection{Bigger picture: conjectures by Aldous and by Feng and Sun}
\label{sec:intro:AD}

\newcommand{\BCRT}{\texttt{BCRT}}

\begin{figure}\centering
 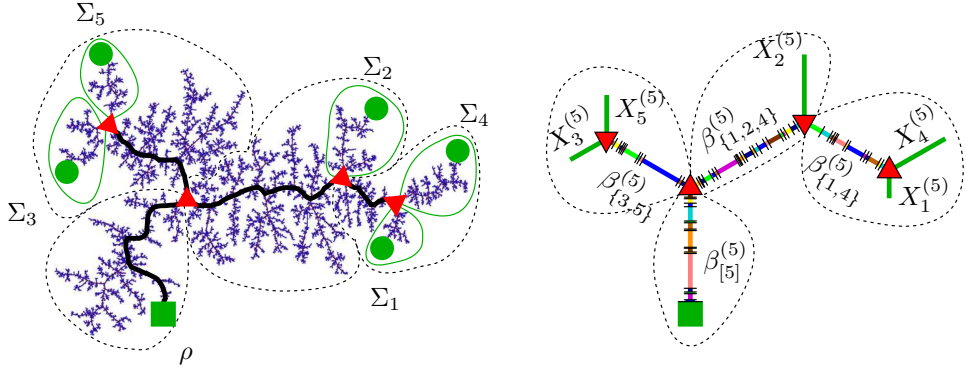
 \caption{Left: a \BCRT, from simulations by Kortchemski \cite{Kortchemski}, and the subtree (heavy black lines) spanned by the branch points (red triangles) separating five leaves (green circles) and the root (green square). Right: projection of mass $p$ onto the subtree spanned by these branch points, represented as interval partitions. The five green lines on the outside of the tree represent masses of the subtrees containing the five leaves.\label{fig:BCRT}}
\end{figure}

The Brownian continuum random tree (\BCRT) is a random rooted, weighted continuum tree $(\mathcal{T},d,\rho,\mu)$, where $(\mathcal{T},d)$ is a tree-like metric space called a continuum tree, $\rho\in\mathcal{T}$ is the root, and $\mu$ is a non-atomic probability measure supported on the leaves of $\cT$. See Figure \ref{fig:BCRT}. The \BCRT\ was introduced by Aldous \cite{AldousCRT1,AldousCRT3} and has subsequently become a major topic of study, touching fields including the Brownian map \cite{LeGallMier12} 
and Liouville quantum gravity \cite{Sheffield16,MillShef19}. The \emph{spinal decomposition} of a \BCRT\ \cite{PitmWink09} is formed by sampling a random leaf $\Sigma\sim \mu$ and decomposing the tree into a path from $\rho$ to $\Sigma$, called the ``spine,'' and a totally ordered collection of ``spinal'' subtrees branching off of this path. The spinal projection of the \BCRT\ is an interval partition with block masses equal to the $\mu$-masses of the spinal subtrees ordered by increasing distance from the root. This forms a \PDIP[\frac12,\frac12] \cite{PitmWink09}.

In 1999, Aldous conjectured \cite{ADFields,AldousDiffusionProblem} that a certain Markov chain on binary trees should have a continuum analogue that would be a diffusion on continuum trees, stationary with the law of the \BCRT. The present work belongs to a series of papers by the same authors  \cite{Paper0,IPPA,Paper2,Paper3,Paper4,Paper5} that resolve this conjecture. In particular, we introduce a ``multi-spinal projection'' of the CRT, depicted in Figure \ref{fig:BCRT}, that finds \PDIP[\frac12,\frac12] and \PDIP[\frac12,0] components in the \BCRT. In \cite{Paper4}, we construct the conjectured diffusion via a projectively consistent system of multi-spinal projections in which the interval partitions evolve according to the dynamics introduced here.

L\"ohr, Mytnik, and Winter \cite{LohrMytnWint18} have also solved a version of Aldous's conjecture, but their process is on a space of ``algebraic measure trees,'' which capture the algebraic but not the metric structure of the trees.

The present work also belongs to a project to resolve a conjecture of Feng and Sun \cite{FengPD,FengSun10}. They conjecture that Petrov's two-parameter Poisson--Dirichlet diffusions should arise naturally as the projection onto the Kingman simplex of a Fleming--Viot diffusion. There have been several attempts to solve this, most recently \cite{CdBERS17,FengSun19}. In \cite{IPPPetrov} we show that the diffusions constructed here project down to Petrov's diffusions with $\alpha\in(0,1)$ and $\theta=0$ or $\theta=\alpha$. In forthcoming work \cite{IPPAT} we use our method to construct measure-valued diffusions based upon the interval partition evolutions presented here, and we generalize to a full two-parameter family of diffusions.

\subsection{Structure}

The structure of this paper is as follows. In Section \ref{sec:prelim}, we recall more details of the set-up of \cite{IPPA}, as required in
this paper. In Section \ref{sec:type-0+transitions}, we prove Theorems \ref{thm:diffusion_0}, \ref{thmtype1Hunt}, \ref{thmtype0Hunt} and 
\ref{thm:BESQ_total_mass} about transition kernels, type-0 evolutions and total mass processes. In Section \ref{sec:dePoissonization}, we turn to 
pseudo-stationarity and de-Poissonization, and we prove Theorems \ref{thm:pseudostat} and \ref{thm:stationary}.


\section{Preliminaries}\label{sec:prelim}

In this section we recall the set-up in which \cite{IPPA} constructed IP-evolutions, including any of the technical notation and results that are 
relevant for the present paper, from \cite{IPPA} and other literature. We fix $\alpha\in(0,1)$ throughout.

\subsection{The state space $(\IPspace,\dI)$: interval partitions with diversity}
\label{sec:prelim_IP}

The definitions in this section are recalled from \cite{IPspace}.

\begin{definition}\label{def:diversity_property} Let $\HIPspace$ denote the set of all interval partitions in the sense of Definition \ref{def:IP_1}. 
  We say that an interval partition $\beta\in\HIPspace$ of a finite interval $[0,M]$ has the \emph{$\alpha$-diversity property}, or that $\beta$ is an 
  \emph{interval partition with diversity}, if the following limit exists for every $t\in [0,M]$:
  \begin{equation}
   \IPLT_{\beta}(t) := \Gamma(1-\alpha)\lim_{h\downto 0}h^\alpha\#\{(a,b)\in \beta\colon\ |b-a|>h,\ b\leq t\}.\label{eq:IPLT}
  \end{equation}
  We denote by $\IPspace\subset\HIPspace$ the set of interval partitions $\beta$ that possess the $\alpha$-diversity property. We call $\IPLT_{\beta}(t)$ the 
  \emph{diversity} of the interval partition up to $t\in[0,M]$. For $U\in\beta$, $t\in U$, we write $\IPLT_{\beta}(U)=\IPLT_{\beta}(t)$, and we write 
  $\IPLT_{\beta}(\infty) := \IPLT_{\beta}(M)$ to denote the \emph{total ($\alpha$-)diversity} of $\beta$.
\end{definition}

Note that $\IPLT_\beta(U)$ is well-defined, since $\IPLT_{\beta}$ is constant on each interval $U \in \beta$, as the intervals are disjoint. We define a reversal involution and a scaling operation on interval partitions $\beta\in\HIPspace$,
 \begin{equation}\label{eq:IP:xforms}
  \reverse_{\textnormal{IP}}(\beta) :=\big\{(\IPmag{\beta}\!-\!b,\IPmag{\beta}\!-\!a)\colon (a,b)\!\in\!\beta\big\}, \quad \scaleI[c][\beta] := \big\{(ca,cb)\colon (a,b)\!\in\! \beta\big\}
 \end{equation}
 for $c>0$. Recall from \eqref{eq:IP:concat_def} the notion of concatenation $\Concat_{U\in\beta}\gamma_U$ of a summable family of interval partitions, $(\gamma_U)_{U\in\beta}$.

Let us discuss some examples of interval partitions relevant for this paper. 

\begin{proposition}\label{prop:PDIP}
\begin{enumerate}[label=(\roman*), ref=(\roman*)]
 \item Consider the zero-set $\mathcal{Z}=\{t\!\in\![0,1]\colon B^{\rm br}_t\!=\!0\}$ of standard Brownian bridge $B^{\rm br}$. Then $(0,1)\setminus\mathcal{Z}$ is a union of
   disjoint open intervals that form an interval partition $\gamma$ with $\frac12$-diversity a.s.. The ranked interval lengths have \PoiDir[\frac12,\frac12] distribution. We call 
   $\gamma$ a Poisson--Dirichlet interval partition \PDIP[\frac12,\frac12].
 \item For Brownian motion $B$, the interval partition $\gamma^\prime$ of $[0,1]$ associated with its zero-set has $\frac12$-diversity a.s.. 
   The ranked
   interval lengths have $\PoiDir[\frac12,0]$ distribution. We call the reversal $\reverse_{\textnormal{IP}}(\gamma^\prime)$ a \PDIP[\frac12,0].
 \item With $B^{\rm br}$ and $B$ in (i)-(ii) as \BESQ[2-2\alpha] bridge and \BESQ[2-2\alpha] process, we define \PDIP[\alpha,\alpha] and
   \PDIP[\alpha,0], which have $\alpha$-diversity.
 \item Let $Y$ be a \Stable[\alpha] subordinator with Laplace exponent $\Phi(\lambda) = \lambda^{\alpha}$. 
   Let $Z\!\sim\!\ExpDist[r]$ be independent of $Y$ and
   $S\!:=\! \inf\{s\!>\! 0\!\!:$ $Y(s)\! >\! Z\}$. Then $Y(S\!-)\!\sim\!\GammaDist[\alpha,r]$ and 
   $Z\!-\!Y(S\!-)\!\sim\!\GammaDist[1\!-\!\alpha,r]$ are independent. For $\beta := \{(Y(s-),Y(s))\colon s\!\in\!(0,S),Y(s-)\!<\!Y(s)\}$ and $\beta^\prime=\{(0,Z\!-\!Y(S\!-))\}\concat\beta$, \label{item:PDIP:Stable}\vspace{-0.2cm}
   \begin{align*}\reverse_{\textnormal{IP}}(\ol{\beta})\stackrel{d}{=}\ol{\beta}&:= \scaleI[\IPmag{\beta}^{-1}][\beta]\sim\PDIP[\textstyle\alpha,\alpha]\\
    \ol{\beta}^\prime&:=\IPmag{\beta^\prime}^{-1}\scaleI\beta^\prime\sim\PDIP[\textstyle\alpha,0].
    \end{align*}
 \item  \label{item:PDIP:tilted}
  For any $r>0$, let $Y_r  = (Y_{r}(s),\, s\ge0)$ denote a subordinator with Laplace exponent
  \begin{equation}
    \Phi_r(\lambda) = (r+\lambda)^\alpha - r^\alpha = \int_0^\infty(1-e^{-\lambda c})\frac{\alpha}{\Gamma(1-\alpha)}c^{-1-\alpha}e^{-cr}dc
    \label{eq:tilted_Laplace}
  \end{equation}
  and let $S_r \sim \ExpDist[r^\alpha]$ independent. Let $\beta$ be as in (iv). Then
  \begin{equation}\label{eq:PDIP_from_tilted}
   \{(Y_{r}(s-),Y_{r}(s))\colon s\!\in\!(0,S_r),Y_{r}(s-)<Y_{r}(s)\}\stackrel{d}{=}\beta.
  \end{equation}
  \end{enumerate}
\end{proposition}
\begin{proof}
  That $\beta$ is an interval partition with $\alpha$-diversity follows from the Strong Law of Large Numbers for the Poisson process of 
  jumps and the monotonicity of $\IPLT_{\beta}(t)$ in $t$. By the definition of the scaling map, the same holds for $\ol{\beta}$.   
  This entails the same for $\beta^\prime$ and $\ol{\beta}^\prime$. Recall that the inverse local time of $B$ is a \Stable[\alpha] 
  subordinator. Hence, the remainder is well-known; see \cite[Lemma 3.7]{PermPitmYor92}, which states that the last zero $G$ of $B$ is 
  a \BetaDist[\alpha,1-\alpha] variable independent of a Bessel bridge $(B(uG)/\sqrt{G},0\!\le\! u\!\le\! 1)$. Finally, the \PoiDir[\alpha,\alpha]\ and 
  \PoiDir[\alpha,0]\ distributions can be read from \cite[Corollary 4.9]{CSP}, and the last claim follows from \cite[Proposition 21]{PitmYorPDAT} and elementary properties of (killed) subordinators.
\end{proof}

\begin{definition} \label{def:IP:metric}
 We adopt the standard discrete mathematics notation $[n] := \{1,2,\ldots,n\}$.
 For $\beta,\gamma\in \IPspace$, a \emph{correspondence} from $\beta$ to $\gamma$ is a finite sequence of ordered pairs of intervals $(U_1,V_1),\ldots,(U_n,V_n) \in \beta\times\gamma$, $n\geq 0$, where the sequences $(U_j)_{j\in [n]}$ and $(V_j)_{j\in [n]}$ are each strictly increasing in the left-to-right ordering of the interval partitions.

 The \emph{distortion} $\dis(\beta,\!\gamma,(U_j,V_j)_{j\in [n]})$ of a correspondence $(U_j,V_j)_{j\in [n]}$ from $\beta$ to $\gamma$ is defined to 
 be the maximum of the following four quantities:\vspace{-0.1cm}
 \begin{enumerate}[label=(\roman*), ref=(\roman*)]
  \item[(i-ii)] $\sup_{j\in [n]}|\IPLT_{\beta}(U_j) - \IPLT_{\gamma}(V_j)|$ and $|\IPLT_{\beta}(\infty) - \IPLT_{\gamma}(\infty)|$,
  \item[(iii)] $\sum_{j\in [n]}|\Leb(U_j)-\Leb(V_j)| + \IPmag{\beta} - \sum_{j\in [n]}\Leb(U_j)$, \label{item:IP_m:mass_1}
  \item[(iv)] $\sum_{j\in [n]}|\Leb(U_j)-\Leb(V_j)| + \IPmag{\gamma} - \sum_{j\in [n]}\Leb(V_j)$. \label{item:IP_m:mass_2}
 \end{enumerate}
 
 For $\beta,\gamma\in\IPspace$ we define $\displaystyle\dI(\beta,\gamma) := \inf_{n\ge 0,\,(U_j,V_j)_{j\in [n]}}\dis\big(\beta,\gamma,(U_j,V_j)_{j\in [n]}\big)$,
 where the infimum is over all correspondences from $\beta$ to $\gamma$.
\end{definition}

\begin{theorem}[Theorem 2.3(d) of \cite{IPspace}]\label{thm:Lusin}
 $(\IPspace,\dI)$ is Lusin, i.e.\ homeomorphic to a Borel subset of a compact metric space. 
\end{theorem}

\subsection{Spindles: excursions to describe block size evolutions}
\label{sec:BESQ}

Let $\cD$ denote the space of c\`adl\`ag functions from $\BR$ to $[0,\infty)$. Let
 \begin{equation}
  \Exc := \left\{f\in\cD\ \middle| \begin{array}{c}
    \displaystyle \exists\ z\in(0,\infty)\textrm{ s.t.\ }\restrict{f}{(-\infty,0)\cup [z,\infty)} = 0,\\[0.2cm]
    \displaystyle f\text{ positive and continuous on }(0,z)
   \end{array}\right\}.\label{eq:cts_exc_space_def}
 \end{equation}
denote the space of positive c\`adl\`ag excursions whose only jumps may be at birth and death. We define the \emph{lifetime}
 $
  \life(f) = \sup\{s\geq 0\colon f(s) > 0\}
 $.

\begin{lemma}[Equation (13) in \cite{GoinYor03}]\label{lem:BESQ:length}
  Let $Z=(Z_s,s\ge 0)$ be a \BESQA\ process starting from $z\!>\!0$. Then the absorption time $\zeta(Z)\!=\!\inf\{s\!\ge\! 0\!\!:$ $Z_s\!=\!0\}$ has distribution 
  \InvGammaDist[1+\alpha,z/2], i.e.\ $z/(2\zeta(Z))$ has density $(\Gamma(1+\alpha))^{-1}x^{\alpha}e^{-x}$, $x\in(0,\infty)$. 
\end{lemma}

For the purpose of the following, we define \emph{first passage times} $H^a\colon\Exc\to[0,\infty]$ via $H^a(f)=\inf\{s\ge 0\colon f(s)=a\}$, $a>0$.

\begin{lemma}[Section 3 of \cite{PitmYor82}]\label{lem:BESQ:existence}
  There is a measure $\Lambda$ on $\Exc$ such that $\Lambda\{f\!\in\!\Exc\colon f(0)\!\neq\! 0\}\!=\!0$, 
  $\Lambda\{H^a\!<\!\infty\}\!=\!a^{-1-\alpha}$, $a\!>\!0$, and under $\Lambda(\,\cdot\,|\,H^a\!<\!\infty)$,\linebreak the restricted canonical process 
  $f|_{[0,H^a]}$ is a \BESQ[4+2\alpha] process starting from 0 and stopped at the first passage time of $a$, independent of $f(H^a+\cdot\,)$, which is a \BESQA\ process 
  starting from $a$. 
\end{lemma}  

We call continuous excursions, such as $\Lambda$-a.e.\ $f\in\Exc$, \emph{spindles}. If $f\in\Exc$ is discontinuous at birth and/or death, we call it a \emph{broken spindle}.  We make the following choice of scaling, which yields $\Phi(\lambda) = \lambda^\alpha$ in Proposition \ref{prop:agg_mass_subord}:
\begin{equation}\label{eq:BESQ:exc_meas_defn}
 \mBxc = \mBxcA = \frac{2\alpha(1+\alpha)}{\Gamma(1-\alpha)}\Lambda,
\end{equation}
where $\Lambda$ is as in Lemma \ref{lem:BESQ:existence}. We will use this as the intensity measure for \BESQA\ spindles in our Poissonian scaffolding and spindles construction. 

We define a \emph{reversal involution} $\reverseexc$ and a \emph{scaling operator} $\scaleB$ by saying, for $a>0$ and $f\in\Exc$,
 \begin{equation}
  \reverseexc (f)\!:=\!\big(f\big((\life(f)\!-\!y)\!-\!\big),y\!\in\!\BR\big)  \text{ and }  \scaleB[a][f]\!:=\!\left(af(y/a),y\!\in\!\BR\right)\!.\label{eq:BESQ:scaling_def}
 \end{equation}

Lemmas 2.8 and 2.9 of \cite{IPPA} state that for $a>0$,
\begin{gather}
 \mBxc(\reverseexc \in\cdot\;) = \mBxc,\qquad 
 \mBxc(\scaleB[a][f] \in\cdot\;) = a^{1+\alpha}\mBxc,\label{eq:BESQ:scaling_inv}\\
 \text{and}\ \ 
 \mBxc\{f \in \Exc\colon \life(f) \geq y\} = \frac{\alpha}{2^\alpha\Gamma(1 - \alpha)\Gamma(1 + \alpha)}y^{-1-\alpha}.\label{eq:BESQ:lifetime}
\end{gather}

\subsection{Scaffolding: jumps describe births and deaths of blocks}
\label{sec:prelim:JCCP}

Let $\bN$ denote a \PRM[\Leb\otimes\mBxc] on $[0,\infty)\times \Exc$. By \eqref{eq:BESQ:lifetime} and standard mapping of \PRM s, $\int\delta(s,\zeta(f))d\bN(s,f)$ is a \PRM[\Leb\otimes\mBxc(\zeta\in\cdot\,)], where, as $z\downarrow 0$, \vspace{-.1cm}
\begin{align*}
 \int_{(z,\infty]}x\,\mBxc(\zeta\in dx)&=\int_\Exc\cf\{\zeta(f)>z\}\zeta(f)d\mBxc(f)\vspace{-.1cm}\\
  & = \frac{1+\alpha}{2^\alpha\Gamma(1-\alpha)\Gamma(1+\alpha)}z^{-\alpha}\longrightarrow\infty.\vspace{-.1cm}
\end{align*}
As a consequence, for a \PRM[\Leb\otimes\mBxc(\zeta\in\cdot\,)], the sum of $\zeta(f)$ over points $(t,f)$ in any finite time interval $(a,b)$ is a.s.\ infinite. To define a L\'evy process $\bX$ incorporating all $\zeta(f)$ as jump heights, we require a limit with compensation, namely $\bX=\xi_\bN$, where for all $t\ge 0$,\vspace{-.1cm}
\begin{equation}\label{eq:scaffolding}
  \xi_\bN(t):=\lim_{z\downto 0}\!\left(\!\int_{[0,t]\times\{g\in\Exc\colon\zeta(g) > z\}}\!\!\!\life(f)d\bN(s,\!f) - t\frac{(1\!+\!\alpha)z^{-\alpha}}{2^\alpha\Gamma(1\!-\!\alpha)\Gamma(1\!+\!\alpha)}\!\right)\!.\vspace{-.1cm}
\end{equation}

Then the \emph{scaffolding} $\fX$ is a spectrally positive stable L\'evy process of index $1+\alpha$, with L\'evy measure and Laplace exponent given by \vspace{-.1cm}
 \begin{equation}
  \mBxc(\zeta\!\in\! dx)=\frac{\alpha(1\!+\!\alpha)x^{-2-\alpha}}{2^\alpha\Gamma(1\!-\!\alpha)\Gamma(1\!+\!\alpha)}dx\quad\mbox{and}\quad\psi(\lambda) = \frac{\lambda^{1+\alpha}}{2^\alpha\Gamma(1\!+\!\alpha)}.\label{eq:JCCP:Laplace}\vspace{-.1cm}
 \end{equation}
 We write ``\StableA '' to refer exclusively to L\'evy processes with this Laplace exponent. In particular, such processes are spectrally positive. 

Boylan \cite{Boylan64} proved that $\fX$ has an a.s.\ unique jointly continuous \emph{local time} process $(\ell_\bX^y(t);\ y\in\BR,t\geq 0)$; i.e.\ for all bounded measurable $f\colon\BR\rightarrow[0,\infty)$ and $t\ge 0$,\vspace{-.1cm}
\begin{equation*}
 \int_{-\infty}^\infty f(y)\ell_\bX^y(t)dy=\int_0^t f(\bX(s))ds.
\end{equation*}

\subsection{Interval partition evolutions from scaffolding and spindles}


We now formalize the construction stated before Theorem \ref{thm:diffusion_0}. We write $\bX|_{[0,T]}$, respectively 
$\bN|_{[0,T]}:=\bN|_{[0,T]\times\Exc}$, to restrict to times $[0,T]$, resp.\ $[0,T]\times\Exc$, by setting the process, resp.\ measure, equal to 0 outside this set.

\begin{definition}[$\bP^1_{\beta}$, $\mBxc$-IP-evolution, Lemma 5.1 of \cite{IPPA}]\label{constr:type-1}
 Let $\beta\in\IPspace$. If $\beta = \emptyset$ then $\bP^1_{\beta}=\delta_\emptyset$ is the Dirac mass on the constant function
 $\emptyset\in\cC([0,\infty),\IPspace)$. Otherwise, for each $U\!\in\!\beta$ we carry out the following construction independently; see Figure \ref{fig:clade_constr}. Let $\bN$ denote a \PRM[\Leb\otimes\mBxc] with scaffolding $\fX$ as in \eqref{eq:scaffolding} and $\bff$ an independent \BESQA\ started from $\Leb(U)$ and absorbed at 0. Consider the hitting time $T := \inf\{t\!>\!0\colon\fX(t)\! =\! -\life(\bff)\}$. Let
 $$\bN_U := \Dirac{0,\bff}+\restrict{\bN}{[0,T]}, \quad \len(\bN_U):=T, \quad \fX_U:=\life(\bff)+\fX|_{[0,T]}.$$ 
 
 \begin{figure}
  \centering
  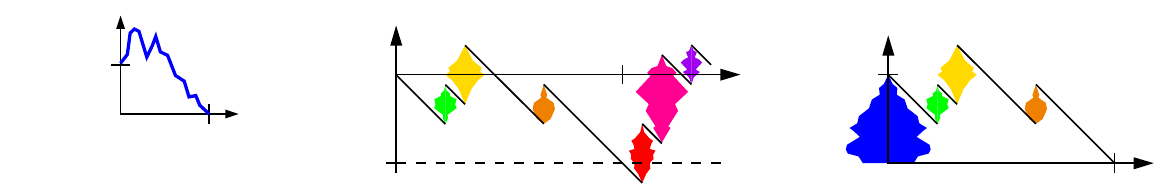
  \caption{Illustration of the construction of $(\bN_U,\fX_U)$ in Definition \ref{constr:type-1}.\label{fig:clade_constr}}
 \end{figure}
 
 Recalling Definition \ref{def:skewer}, 
 we call a $\dI$-path-continuous version $(\nbeta^y,\,y\ge0)$ of $\Concat_{U\in\beta}\skewer(y,\bN_{U},\fX_{U})$, $y\ge 0$, a \emph{$\mBxc$-IP-evolution starting from $\beta$}.
 We denote its distribution on $\cC([0,\infty),\IPspace)$ by $\bP^1_\beta$. For probability measures $\mu$ on $\IPspace$, we write $\bP^1_{\mu} := \int \bP^1_{\beta}\mu(d\beta)$ to denote the $\mu$-mixture of the laws $\bP^1_{\beta}$.
\end{definition}

In \cite[Proposition 5.11]{IPPA}, we proved the existence of $\mBxc$- and other IP-evolutions. Moreover, we showed that we can define 
concatenations $\bN_{\beta} := \ConcatIL_{U\in\beta}\bN_U$ (and $\fX_\beta=\ConcatIL_{U\in\beta}\fX_U$) of point measures (and scaffoldings) such that $\skewerP(\bN_\beta,\bX_\beta)$ is well-defined, $\dI$-path-continuous, and the diversities of the resulting interval partitions coincide with scaffolding local times at all times and levels:
\begin{equation}
    \IPLT_{\skewer(y,\bN_\beta,\bX_\beta)}\big( M^y_{\bN_\beta,\bX_\beta}(t) \big) = \ell^y_{\bX_\beta}(t) \quad \text{for }t,y\ge0.\label{eq:div_LT_condition}%
\end{equation}


Recall from \cite{IPPA} the following useful property of $\mBxc$-IP-evolutions.

\begin{lemma}[\cite{IPPA}, Lemma 5.1]\label{lem:finite_survivors} $\!\!$
 Let $(\bN_U,\fX_U)_{U\in\beta}$ as in Definition \ref{constr:type-1}, $y\!>\!0$. 
 Then a.s.\ $\skewerP(y,\bN_U,\bX_U)\neq\emptyset$ for at most finitely many $U\!\in\!\beta$.
\end{lemma}

\subsection{Decompositions of scaffolding and spindles at a fixed level}\label{sec:decomp}

This section summarizes \cite[Section 4]{IPPA}. Let $(\fN,\fX)$ denote a \PRM[\Leb\otimes\mBxcA] and the associated \StableA\ scaffolding, as above. For $y\in\BR$ we denote by
\begin{equation}
 \tau^y(s) := \tau_\fX^y(s) := \inf\{t\ge 0\colon \ell^y_\fX(t)>s\},\quad  s\geq 0,
\end{equation}
the level-$y$ inverse local time process of $\bX$. In order to prove that the interval partition process of Definition \ref{constr:type-1} has the transition kernel $\kappa^{(\alpha)}$ of \eqref{eq:intro:transn}, it is useful to consider the decomposition of a \StableA\ scaffolding process $\bX$ into excursions about a level. For fixed $y\in\BR$, we can decompose the path of $\bX$ after the first hitting time of level $y$ into a collection of excursions about level $y$. It\^o's excursion theory \cite{Ito72} states that excursions of $\fX$ about level $y$ form a \PRM , $\sum_{s>0\colon \tau^y(s)>\tau^y(s-)}\delta(s,g_s)$, where each $g_s$ is an excursion of $\fX$ about level $y$ and the corresponding $s$ is the level-$y$ local time at which $g_s$ occurs. The $\sigma$-finite intensity measure $\mSxc = \mSxcA$ of this \PRM\ is called the It\^o measure of these excursions.

\begin{figure}
 \centering
 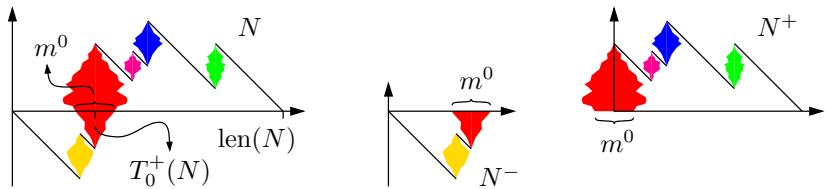
 \caption{Decomposition of a bi-clade $N$ into an anti-clade $N^-$ and clade $N^+$.\label{fig:bi-clade_decomp}}
\end{figure}

Extending the restriction notation of the previous section, we write $\shiftrestrict{\fN}{[a,b]}$ to denote the \emph{shifted restriction}, formed by first taking the restriction of $\fN$ to $[a,b]\times\cE$, then translating to obtain a point process supported on $[0,b-a]\times\Exc$. If $\fX = \xi(\fN)$ has an excursion $g$ about level $y$ during the interval $[a,b]$, then $N = \shiftrestrict{\fN}{[a,b]}$ satisfies $g = \xi(N)$. We call $N$ a \emph{bi-clade}; see Figure \ref{fig:bi-clade_decomp}.

Each excursion of $\fX$ about level $y$ comprises an initial escape down from $y$; a single jump up across $y$, which we call the middle jump; and a final return back down to $y$. For the purpose of the following, let $T$ denote the time of this middle jump in a bi-clade $N$. We split the spindle $f_T$ marking the middle jump into two broken spindles: $\check f_T$ corresponding to the part of the jump that occurs below level $0$, and $\hat f_T$ corresponding to the continuation of that jump. Then, we decompose our bi-clade $N$ into two point processes, depicted in Figure \ref{fig:bi-clade_decomp}:
\begin{equation*}
 N^- = \Restrict{N}{[0,T)}+\delta\big(T,\check f_T\big) \quad \text{and} \quad N^+ = \delta\big(0,\hat f_T\big) + \ShiftRestrict{N}{(T,\infty)}.
\end{equation*}
We call $N^-$ and $N^+$ the \emph{anti-clade} and \emph{clade} parts of $N$.

Consider the stochastic kernel $\lambda$ that takes in a scaffolding process $g$ and maps it to the law of a point process of spindles $N$ with $\xi(N) = g$ by independently marking each jump of $g$ with a \BESQA\ excursion conditioned to have lifetime equal to the height of the jump. We can obtain a $\sigma$-finite It\^o measure on bi-clades by mixing this kernel over the It\^o measure of \StableA\ excursions: 
$\mClade := \mCladeA := \int\lambda(g,\cdot\,)d\mSxc(g)$. 
Let $\mClade^+$ and $\mClade^-$ denote the pushforward of this measure onto clades and anti-clades, respectively. 
We denote by $\bF^y = \sum_{s>0\colon \tau^y(s)>\tau^y(s-)}\delta(s,N_s)$ the point process of bi-clades corresponding to the excursions of $\fX$ about level $y$, and we denote by $\bF^{\ge y} = \sum\delta(s,N_s^+)$ and $\bF^{\le y} = \sum\delta(s,N_s^-)$ the corresponding point processes of (anti-)clades.

\begin{proposition}[Proposition 4.9 and Corollary 4.10 of \cite{IPPA}]\label{prop:bi-clade_PRM}
 $\bF^y$ is a \PRM[\Leb\otimes\mClade], $\bF^{\ge y}$ is a \PRM[\Leb\otimes\mClade^+], and $\bF^{\le y}$ is a \PRM[\Leb\otimes\mClade^-].
\end{proposition}

We define a time-reversal involution and a scaling operator on point processes of spindles:
\begin{equation}\label{eq:clade:xform_def}
 \begin{split}
  \reverseH(N) &:= \int \Dirac{\len(N)-t,\reverseexc(f)}dN(t,f),\\
  b\scaleH N &:= \int \Dirac{b^{1+\alpha}t,\scaleB[b][f]}dN(t,f) \qquad \text{for }b>0,
\end{split}
\end{equation}
where $\scaleB$ and $\reverseexc$ are as in \eqref{eq:BESQ:scaling_def}. The map $\reverseH$ reverses the order of spindles and reverses time within each spindle. Lemma 4.11 of \cite{IPPA} notes that, for $b>0$ and $A$ a measurable set of bi-clades,
\begin{equation}\label{eq:clade:invariance}
 \mClade(\reverseH(A)) = \mClade(A) \quad \text{and} \quad \mClade(b\scaleH A) = b^{-\alpha}\mClade(A).
\end{equation}

Note that each bi-clade of $\fN$ corresponding to an excursion of $\fX$ about level $y$ gives rise to a single block in $\skewer(y,\fN,\fX)$, with block mass equal to a cross-section of the spindle marking the middle jump of the excursion. We denote this mass by
\begin{equation}
 m^0(N) := \!\int\! \max\Big\{f\big((-X(s-))\!-\!\big),\,f\!\big(\!-\!X(s-)\big)\Big\}dN(s,f).\label{eq:clade:mass_def}
\end{equation} 
The two quantities in the $\max\{\cdot\,,\cdot\}$ are equal for typical bi-clades but will differ when $m^0$ is applied to an anti-clade or clade; cf.\ Figure \ref{fig:bi-clade_decomp}.

\begin{proposition}[Aggregate mass process, Proposition 8(i) of \cite{Paper0} with $q=c=1$]\label{prop:agg_mass_subord}
  Consider $\fN\sim\PRM(\Leb\otimes\mBxc)$ and $\fX=\xi_\fN$ as in \eqref{eq:scaffolding}. Then for any fixed $y\in\BR$, the process $\bM^y(s):=M^y_{\bN,\fX}\circ\tau^y_{\fX}(s) - M^y_{\bN,\fX}\circ\tau^y_{\fX}(0)$, $s\geq 0$, is a \Stable[\alpha] subordinator with Laplace exponent $\Phi(\lambda) = \lambda^\alpha$.
\end{proposition}

Note that the L\'evy measure for this subordinator equals $\mClade(m^0\in\cdot\,)$, so the formula for this is implied by the proposition; we state it explicitly in Proposition \ref{prop:clade:stats}\ref{item:CS:mass}. 


By Lemma 4.14 of \cite{IPPA}, $\mClade$ has a unique $m^0$-disintegration $\mClade(\,\cdot\,|\,m^0)$ with the scaling property that for all $b>0$,
\begin{equation}
  \mClade(\,\cdot\;|\;m^0 =b) = \mClade\left(\scaleH[b^{-1}][N]\in\,\cdot\ \middle|\;m^0 = 1\right).\label{eq:clade_mass_ker}
\end{equation}
 Proposition 4.15 of \cite{IPPA} notes that under $\mClade(\,\cdot\,|\,m^0=b)$ the point measure is distributed as the concatenation 
\begin{equation}\label{eq:clade_split:indep}
  \ol{N}_b^-\concat\ol{N}_b^+\text{\ \ where\ \ }\ol{N}_b^+\text{ is independent of }\ol{N}_b^-,
\end{equation}
with the convention that the two broken spindles at the concatenation time $\len(\ol{N}_b^-)$, which are resp.\ rightmost for
$\ol{N}_b^-$ and leftmost for $\ol{N}_b^+$ and which resp.\ end and start at mass $b$, are concatenated to form a single spindle. This concatenation reverses the decomposition depicted in Figure \ref{fig:bi-clade_decomp}. The same proposition states that
\begin{equation}\label{eq:clade_split:distrib}
 \reverseH(\ol{N}_b^-) \stackrel{d}{=} \ol{N}_b^+\stackrel{d}{=}\fN_{U} \sim \mClade^+\{\,\cdot\;|\;m^0 = b\},
\end{equation}
where $\fN_{U}$ is as in Definition \ref{constr:type-1} with $U = (0,b)$.

\section{Transition kernels, total mass and type-0 evolutions\vspace{-0.1cm}}
\label{sec:type-0+transitions}

\subsection{The transition kernel of $\mBxc$-IP-evolutions}
\label{sec:type-1:clade}

\begin{definition}\label{def:clade_stats}
 We define statistics of bi-clades $N$ with $X=\xi_N$:
  \begin{itemize}\item \emph{overshoot} and \emph{undershoot} at $T_0^+(N)=\inf\{t>0\colon X(t)\ge 0\}$\vspace{-0.1cm}  
      $$J^+(N) := X(T_0^+(N)),\quad\mbox{and}\quad J^-(N) := -X(T_0^+(N)-),\vspace{-0.1cm}$$
    \item \emph{clade lifetime} and \emph{anti-clade lifetime}\vspace{-0.1cm}
	  $$\life^+(N) := \sup_{t\in[0,\len(N)]} X(t)\quad\mbox{and}\quad\life^-(N) := -\inf_{t\in[0,\len(N)]}X(t).$$
  \end{itemize}
\end{definition}

\begin{figure}
 \centering
 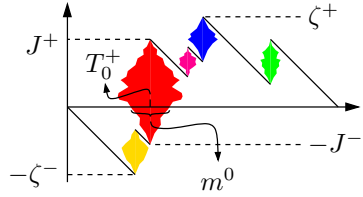\vspace{-0.2cm}
 \caption{A bi-clade, with the statistics of \eqref{eq:clade:mass_def} and Definition \ref{def:clade_stats} labeled.\label{fig:clade_stats}}
\end{figure}

\begin{table}[b]
 \centering
 \begin{tabular}{|c|Sc|}\hline
  $J^+(\scaleH[c][N]) = cJ^+(N)$ & $J^-(\scaleH[c][N]) = cJ^-(N)$ \\\hline
  $\life^+(\scaleH[c][N]) = c\life^+(N)$ & $\life^-(\scaleH[c][N]) = c\life^-(N)$ \\\hline
  $T_0^+(c\scaleH N) = c^{1+\alpha}T_0^+(N)$ & $\len(c \scaleH N) = c^{1+\alpha}\len(N)$ \\\hline
  $m^0(c\scaleH N) = cm^0(N)$& $\ell^y_{\xi(c \scaleH N)}(t) = c^{\alpha}\ell^{y/c}_{\xi(N)}(tc^{-1-\alpha})$\\\hline
 \end{tabular}\vspace{4pt}
 \caption{How statistics of \eqref{eq:clade:mass_def} and Definition \ref{def:clade_stats} scale as $N$ is scaled, as in Section \ref{sec:decomp}.\label{tbl:clade_scaling}}
\end{table}

The above quantities appear labeled in Figure \ref{fig:clade_stats}. The rates at which they scale under $\scaleH$ are listed in Table \ref{tbl:clade_scaling}. 
We call $\life^+$ ``lifetime'' rather than ``maximum'' since values in the scaffolding function play the role of times in the evolving interval 
partitions $(\skewer(y,N,X),\ y\geq 0)$ that we ultimately wish to study. As with conditioning on $m^0$ in \eqref{eq:clade_mass_ker}, there are unique kernels with scaling properties that allow us to condition $\mClade$ on the exact value of any one of the other quantities of Definition \ref{def:clade_stats} and get a resulting probability distribution.

\begin{proposition}\label{prop:clade:stats}
 \begin{enumerate}[label=(\roman*), ref=(\roman*)]
  \item $\displaystyle \mClade\big\{ m^0 > b \big\} = \frac{1}{\Gamma(1-\alpha)}b^{-\alpha}.$\label{item:CS:mass}\vspace{-0.1cm}
  \item $\displaystyle \mClade\big\{ \life^+ > z \big\} = \frac{1}{2^\alpha}z^{-\alpha}.$\label{item:CS:max}
  \item $\displaystyle \mClade\big\{ J^+ \in dy\;\big|\;m^0 = b\big\} = \frac{b^{1+\alpha}}{2^{1+\alpha}\Gamma(1+\alpha)}y^{-2-\alpha}e^{-b/2y}dy.$\label{item:CS:over:mass}
  \item $\displaystyle \mClade\big\{ \life^+ \leq z\;\big|\;m^0 = b\big\} = e^{-b/2z}.$\label{item:CS:max:mass}
  \item $\displaystyle \mClade\big\{ m^0 \leq b\;\big|\;J^+ = y\big\} = 1 - e^{-b/2y}.$\label{item:CS:mass:over}
  \item $\displaystyle \mClade\big\{ \life^+ \leq z\;\big|\;J^+ = y\big\} = \cf\{z\geq y\}\left(\frac{z-y}{z}\right)^\alpha.$\label{item:CS:max:over}
  \item $\displaystyle \mClade\big\{ m^0 \in db\;\big|\;\life^+ \geq z\big\} = \frac{\alpha 2^\alpha z^\alpha}{\Gamma(1-\alpha)}b^{-1-\alpha}(1-e^{-b/2z})db.$\label{item:CS:mass:max}
 \end{enumerate}\vspace{-2pt}\ 
 
 \noindent All of these equations remain true if we replace all superscripts `$+$' with `$-$'.
\end{proposition}

The proof of this is given in Appendix \ref{sec:clade_stats}; it is based on (\ref{eq:clade_split:distrib}) and well-known properties of our spectrally one-sided \StableA\ process.

\begin{corollary}\label{cor:clade_law_given_over}
 Take $y>0$. Let $A\sim \ExpDist[1/2y]$. Conditionally given $A$, let $\bff$ denote a \BESQA\ first-passage bridge from $\bff(0) = A$ to $\bff(y) = 0$, in the sense of \cite{BertChauPitm03}. Let $\bN$ be a \PRM[\Leb\otimes\mBxc], with $T^{-y}$ the hitting time of $-y$ by $\bX$. Then $\Dirac{0,\bff}+\restrict{\bN}{[0,T^{-y}]}$ has law $\mClade^+(\,\cdot\;|\;J^+ = y)$.
\end{corollary}

\begin{proof}
 By \eqref{eq:clade_split:distrib}, under $\mClade^+(\,\cdot\,|\,m^0\!=\!b)$ a clade $N^+$ has the form $\DiracBig{0,\bff} + \bN'$. Here, $\bff$ is a \BESQA\ starting from $b$ and, conditionally given $\bff$, the point process $\bN'$ is distributed like $\bN$ stopped at time $T^{-\life(\bff)}$. Then $J^+(N^+) = \zeta(\bff)$. Thus, we may further condition $\mClade^+(\,\cdot\;|\;m^0 = b,\ J^+ = y)$. Under this new law, $N^+$ has the same form, and $\bff$ is now distributed like a \BESQA\ first-passage bridge from $b$ to $0$ in time $y$. So, since its lifetime is fixed, in this setting $\bff$ is independent of $\bN'$. Now,
 $$\mClade^+(\,\cdot\;|\;J^+ = y) = \int \mClade^+(\,\cdot\;|\;m^0 = b,\ J^+ = y)\mClade^+(m^0\in db\;|\;J^+ = y).$$
 The conditional law of $m^0$ above appears in Proposition \ref{prop:clade:stats} \ref{item:CS:mass:over}. In particular, under this law, $m^0\sim \ExpDist[1/2y]$.
\end{proof}

On a suitable probability space $(\Omega,\cA,\Pr)$ let $\bN$ be a \PRM[\Leb\otimes\mBxc] on $[0,\infty)\times \Exc$. We write $\bX=\xi_\bN$ for the scaffolding associated as in \eqref{eq:scaffolding}. 
Fix $b>0$ and let $\bff$ be a \BESQA\ starting from $b$ and absorbed upon hitting zero, independent of $\bN$. 
Let $\olN := \Dirac{0,\bff} + \bN$. We use barred versions of our earlier notation to refer to the corresponding objects associated with $\olN$,
e.g.\ $\olX = \bX + \zeta(\bff)$. Let $\ol T^0 = T^{-\life(\bff)}$ denote the first hitting time of 0 by $\olX$, or of $-\life(\bff)$ by $\fX$, and set $\widehat\bN := \restrict{\olN}{[0,\ol T^0)}$. By (\ref{eq:clade_split:distrib}), $\widehat\bN$ has distribution $\mClade^+(\,\cdot\;|\;m^0 = b)$. We use hatted versions of our earlier notation to refer to the corresponding objects associated with $\widehat\bN$. Set $(\widehat\nbeta^y,\,y\geq 0) := \skewerP(\widehat\bN,\widehat\bX)$.

\begin{proposition}[Entrance law for $\mBxc$-IP-evolution from $\{(0,b)\}$]\label{prop:type-1:transn}
 The lifetime of $(\widehat\nbeta^y,y\ge 0)$ has \InvGammaDist[1,b/2] distribution, i.e.
 \begin{equation}
  \Pr\big\{\life^+\big(\whN\big)> y\big\} = \Pr(\widehat\nbeta^y \not= \emptyset) = 1 - e^{-b/2y} \qquad \text{for }y>0.\label{eq:transn:lifetime}
 \end{equation}
 The conditional law of $\widehat\nbeta^y$ given the event $\{\widehat\nbeta^y\neq \emptyset\}$ equals $\mu_{b,1/2y}^{(\alpha)}$ as defined in \eqref{eq:mu}.
\end{proposition}

We restate the claim about the leftmost block $L^y\!:=\!L_{b,1/2y}^{(\alpha)}$ under $\mu_{b,1/2y}^{(\alpha)}$ in terms of the leftmost spindle mass at 
level $y$ in a bi-clade $N$ with $X=\xi_N$, 
\begin{equation}\label{eq:LMB_def}
  m^y(N) := M^y_{N,X}\big(\inf\{ t\geq 0\colon M^y_{N,X}(t) >0\}\big).
 \end{equation}

\begin{lemma}\label{lem:LMB}
 For $b,c,y>0$, $\mClade\{ m^y\in dc\,|\,m^0 = b,\,\life^+ > y\}$ equals 
 \begin{equation}
 \Pr(L^y\!\in\! dc)=\frac{2^\alpha y^\alpha}{e^{b/2y}-1}c^{-1-\alpha}e^{-c/2y}\sum_{n=1}^\infty\frac{1}{n!\Gamma(n-\alpha)}\left(\frac{bc}{4y^2}\right)^ndc.
 \end{equation}
 with Laplace transform as specified in \eqref{LMBintro} for $r=1/2y$. 
\end{lemma}

That this distribution has Laplace transform \eqref{LMBintro} is elementary. We prove the remainder of this lemma at the end of Appendix \ref{sec:clade_stats}.

\begin{remark}\label{rmk:transn:PDIP}
 $\!\!$By Proposition \ref{prop:PDIP} \ref{item:PDIP:Stable}--\ref{item:PDIP:tilted}, taking $B^y\sim\GammaDist[\alpha,1/(2y)]$ independent of $\ol\gamma\sim\PDIP[\alpha,\alpha]$, and $S^y\sim\ExpDist[(2y)^{-\alpha}]$ independent of a subordinator $(R^y,y\ge 0)$ with Laplace exponent $\Phi_{1/2y}(\lambda)$, then
 \begin{equation}\label{eq:transn:PDIP} 
   B^y\scaleI\ol\gamma\stackrel{d}{=}\{(R^y(t-),R^y(t))\colon t\in [0,S^y],\ R^y(t-)<R^y(t)\},
 \end{equation}
 and we denote its distribution by $\mu_{0,1/2y}^{(\alpha)}$.
\end{remark}

\begin{proof}[Proof of Proposition \ref{prop:type-1:transn}]
 By construction, $\life(\bff)$ is independent of $\bN$. 
 By Proposition \ref{prop:bi-clade_PRM} and the aforementioned independence, the point process $\olF^y = \bF^{y-\life(\bff)}$ is a \PRM[\Leb\otimes\mClade]. Let $\widehat S^y := \oll^y(\ol T^0)$. If $\widehat\bN$ survives past level $y$ then $\widehat S^y$ is the level $y$ local time at which some excursion of $\olX$ about level $y$ first reaches down to level zero:
 \begin{equation*}
  \widehat S^y = \cf\big\{\life^+\big(\whN\big)\!>\!y\big\}\inf\left\{s\!>\!0\colon \olF^y\big([0,s]\times\{N\colon \life^-(N) \geq y\}\big) > 0 \right\}.
 \end{equation*}
 Conditionally given the event $\{\life^+(\widehat\bN)\!>\!y\}$ of survival beyond level $y$, it follows from the Poisson property of $\olF^y$ and the description of $\mClade\{\life^-\!\in\!\cdot\,\}$\linebreak in Proposition \ref{prop:clade:stats} \ref{item:CS:max} that $\widehat S^y \sim \ExpDist[(2y)^{-\alpha}]$, which is the distribution of $S^y$ as specified in Remark \ref{rmk:transn:PDIP}. In light of this, up to null events,
 \begin{equation}\label{eq:type-1:no_revival}
  \big\{\life^+\big(\widehat\bN\big) \leq y\big\} = \big\{\wh S^y = 0\big\} = \big\{\widehat\bF^{\geq y} = 0\big\} = \big\{\widehat\nbeta^y = \emptyset\big\}.
 \end{equation}
 Recall from (\ref{eq:clade_split:distrib}) that $\whN\sim\mClade^+\{\,\cdot\;|\;m^0=b\}$. Thus, \eqref{eq:transn:lifetime} follows from the formula for $\mClade^+\{\life^+ > z\;|\;m^0=b\}$ stated in Proposition \ref{prop:clade:stats} \ref{item:CS:max:mass}.
 
 Assuming $\life^+(\widehat\bN) > y$, time $\ol T^0$ occurs during an anti-clade of $\olN$ below level $y$ at local time $\widehat S^y$. In particular, the subsequent level-$y$ clade, also at local time $\wh S^y$, is cut entirely from $\whN$. Thus, $\widehat\bF^{\geq y} = \restrict{\olF^{\geq y}}{[0,\widehat S^y)}$. That is, $\whF^{\geq y}$ is obtained from $\olF^{\geq y}$ by Poisson thinning. By Proposition \ref{prop:clade:stats} assertions \ref{item:CS:mass}, \ref{item:CS:max}, and \ref{item:CS:mass:max},
 \begin{equation*}
 \begin{split}
  &\mClade\{ m^0\in db;\ \zeta^- < y \}\\ 
    &= \mClade\{ m^0\in db \} - \mClade\{m^0\in db\;|\;\zeta^- > y \} \mClade\{ \zeta^- > y \}\\
  	&= \frac{\alpha}{\Gamma(1-\alpha)}b^{-1-\alpha}db - \frac{1}{2^\alpha}y^{-\alpha}\frac{\alpha 2^\alpha y^{\alpha}}{\Gamma(1-\alpha)}(1-e^{-b/2y})b^{-1-\alpha}db\\
  	&= \frac{\alpha}{\Gamma(1-\alpha)}b^{-1-\alpha}e^{-b/2y}db = e^{-b/2y}\mClade\{ m^0\in db\}.
 \end{split}
 \end{equation*}
 Thus, it follows from Proposition \ref{prop:agg_mass_subord} that the conditional law of $\big(M^y_{\whN}\circ\tau^y_{\whN}(s)-M^y_{\whN}\circ\tau^y_{\whN}(0),\ s\in [0,\wh S^y]\big)$ given $\{\life^+(\widehat\bN) > y\}$ equals the law of $\restrict{R^y}{[0,S^y]}$. Thus, appealing to \eqref{eq:type-1:no_revival}, 
 the conditional distribution of $\wh\nbeta^y$ minus its leftmost block given $\{\wh\nbeta^y\neq\emptyset\}$ is as described in \eqref{eq:transn:PDIP}.
 
 The mass $m^y(\whN)$ of the leftmost block is a function of $\restrict{\whN}{[0,\wh T^y)}$, whereas $\wh\nbeta^y$ minus its leftmost block is a function of $\restrict{\whN}{[\wh T^y,\infty)}$. 
 When the latter is shifted to a measure $\ShiftRestrict{\whN}{[0,\infty)}$ on $[0,\infty)\times\Exc$, these are independent, by the strong Markov property of $\whN$. We conclude by Lemma \ref{lem:LMB}.
\end{proof}


\begin{corollary}[Transition kernel for $\mBxc$-IP-evolutions]\label{cor:type-1:gen_transn}
 Fix $y>0$. Let $(\beta^y_U,U\!\in\!\gamma)$ denote an independent family of partitions, with each $\beta^y_U$ distributed like $\wh\nbeta^y$ in Proposition$\,$\ref{prop:type-1:transn} with $b\!=\!\Leb(U)$. Then $\skewer(y,\bN_{\beta},\bX_{\beta})$ $\stackrel{d}{=}\ConcatIL_{U\in\beta}\beta^y_U$, and this law is supported on $\IPspace$.
%
\end{corollary}

\begin{proof}[Proof of Theorems \ref{thmtype1Hunt} and \ref{thm:diffusion_0}] Note that Corollary \ref{cor:type-1:gen_transn} identifies 
  $(\kappa_y^{(\alpha)},y\!\ge\! 0)$ as the semigroup of $\mBxc$-IP-evolutions, which are 1-self-similar path-continuous Hunt processes by 
  \cite[Theorem 1.4]{IPPA} and continuous in the initial state, by Proposition \cite[Proposition 5.20]{IPPA}. Hence, $(\kappa_y^{(\alpha)},y\ge 0)$ is as required for Theorem \ref{thmtype1Hunt} and $\mBxc$-IP-evolutions are
  type-1 evolutions as claimed in Theorem \ref{thm:diffusion_0}.
\end{proof}

\subsection{Type-0 evolutions: construction and properties}
\label{sec:type-0}
\def\cFI{\cF_{\IPspace}}

We will construct type-0 evolutions from point measures of
spindles and associated scaffolding. Let $\bN$ denote a \PRM[\Leb\otimes\mBxc] on $[0,\infty)\times\Exc$. For $y\in\BR$, let $T^y$ denote the first hitting time of $y$ by the scaffolding $\bX=\xi_\bN$ associated with $\bN$ in \eqref{eq:scaffolding}. We define
\begin{equation*}
 \cev\nbeta^y_j := \skewer\left(y,\restrict{\bN}{[0,T^{-j})},j+\restrict{\bX}{[0,T^{-j})}\right) \qquad \text{for }j\in\BN,\ y\in [0,j].
\end{equation*}
Note that for $k>j$ the pair $\big(\Restrict{\bN}{[T^{j-k},T^{-k})},k+\Restrict{\bX}{[T^{j-k},T^{-k})}\big)$ shifted to a pair 
$\big(\ShiftRestrict{\bN}{[T^{j-k},T^{-k})},k+\ShiftRestrict{\bX}{[T^{j-k},T^{-k})}\big)$ on time interval $[0,T^{-k}-T^{j-k}]$ has the same distribution as $\big(\restrict{\bN}{[0,T^{-j})},j+\restrict{\bX}{[0,T^{-j})}\big)$, and thus
\begin{equation}\label{eq:type-0:consistency}
  (\cev\nbeta^y_k,\ y\in [0,j]) \stackrel{d}{=} (\cev\nbeta^y_j,\ y\in [0,j]).
\end{equation}
Thus, by Kolmogorov's extension theorem and \cite[Lemma II.35.1]{RogersWilliams}, there exists a continuous process $(\cev\nbeta^y,\ y\geq 0)$ such that for every $j\in\BN$ we have $(\cev\nbeta^y,\ y\in [0,j]) \stackrel{d}{=} (\cev\nbeta^y_j,\ y\in [0,j])$.

\begin{definition}[$\BPr^0_{\beta}$, $\BPr^0_{\mu}$]
\label{constr:type_0}
 Let $\beta\in\IPspace$. Let $(\cev\nbeta^y,\ y\geq 0)$ be as above and $(\vecc\nbeta^y,\ y\geq 0)$ an independent type-1 evolution starting from $\beta$. Consider $(\nbeta^y,\,y\geq 0) = (\cev\nbeta^y\concat\vecc\nbeta^y,\ y\geq 0)$. Let $\BPr^0_{\beta}$ its law on $\cCRI$. 
 For probability measures $\mu$ on $\IPspace$ we write $\BPr^0_{\mu}=\int\BPr^0_\beta\mu(d\beta)$.
\end{definition}

We will show that $\BPr^0_\mu$ is the distribution of a type-0 evolution starting from initial distribution $\mu$, by showing that 
$(\BPr^0_\beta,\beta\!\in\!\IPspace)$ is a family of distributions of a Markov process with transition semigroup 
$(\widetilde{\kappa}_y^{(\alpha)},y\!\ge\! 0)$ as defined just above Theorem \ref{thmtype0Hunt}. 
The path-continuity of $(\cev\nbeta^y\!\concat\!\vecc\nbeta^y,y\!\geq\! 0)$ follows from our results for type-1 evolutions.

\begin{figure}
 \centering
 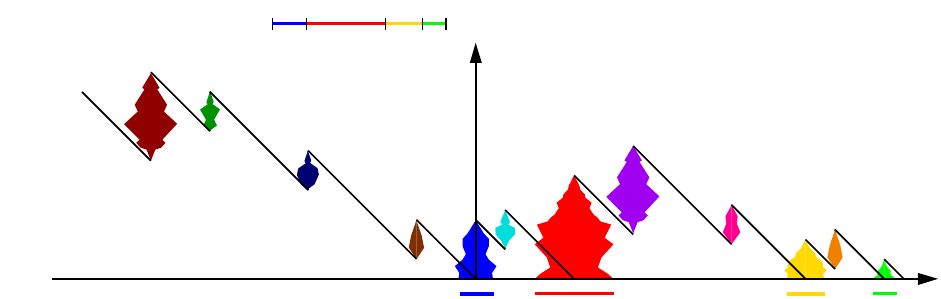
 \caption{To the left of the y-axis, $\big(\cev{\bN},\cev{\bX}\big)$, as in Remark \ref{rmk:type_0_constr}. To the right, $(\bN_\beta,\bX_\beta)$, as below Definition \ref{constr:type-1}.\label{fig:down_from_infty}}
\end{figure}

\begin{remark}\label{rmk:type_0_constr}
 It is possible to construct the type-0 evolution as the skewer of a point process of spindles, rather than via consistency and the extension theorem as we have done above. This would involve setting up a point process of spindles $\cev\bN$ on $(-\infty,0)\times\Exc$ such that, for a suitable extension of the definition \eqref{eq:scaffolding} of $\bX$, the resulting process $\cev\bX$ could be understood as a \StableA\ first-passage descent from $\infty$ down to 0; see Figure \ref{fig:down_from_infty}. Related processes have been studied in the literature. For example, Bertoin \cite[Section VII.2]{BertoinLevy} constructs spectrally negative L\'evy processes that are conditioned to stay positive. Transforming such a process via sign change and an increment reversal  results in a spectrally positive process coming down from $\infty$ to $0$. We find the above consistency  approach to be notationally friendlier.
\end{remark}

Note that $(\cev\nbeta^y,\, y\geq 0)$ itself has distribution $\BPr^0_\emptyset$. We will see that $\emptyset$ is a reflecting boundary for type-0 evolutions, whereas it is absorbing for type 1.

\begin{proposition}[Transition kernel under $\BPr^0_\beta$]\label{prop:type-0:transn}
 Take $\gamma\in\IPspace$ and $y>0$. Let $(\gamma^y_U,\,U\in\gamma)$ denote an independent family of partitions, with each $\gamma^y_U$ distributed 
 as $\wh\nbeta^y$ in Proposition \ref{prop:type-1:transn} with $b = \Leb(U)$. Let $\gamma_0\sim\mu_{0,1/2y}^{(\alpha)}$ be as in 
 \eqref{eq:transn:PDIP}, independent of $(\gamma^y_U,\ U\in\gamma)$. Then under $\BPr^0_{\gamma}$, the interval partition $\nbeta^y$ has the same 
 distribution as
 $\gamma_0\concat \Concat_{U\in\gamma}\gamma^y_U$ in \eqref{eq:intro:transn_0}. 
\end{proposition}

\begin{proof}
 Let $(\cev\nbeta^z)$ and $(\vecc\nbeta^z)$ be as in Definition \ref{constr:type_0} with $\beta = \gamma$. By Corollary \ref{cor:type-1:gen_transn}, $\vecc\nbeta^y \stackrel{d}{=} \ConcatIL_{U\in\gamma}\gamma^y_U$. By construction, this is independent of $\cev\nbeta^y$. It remains only to show that $\cev\nbeta^y$ is distributed like the interval partition arising from the range of $R^y$, up to time $S^y$, in the
 notation of \eqref{eq:transn:PDIP}.
 
 Let $\whN$ have law $\mClade^+\{\,\cdot\,|\,m^0\!=\!1,\,\life^+\!>\!y\}$. Let $\wh T^y$ denote the first hitting time of $y$ in its scaffolding $\whX$. It follows from the description of $\mClade^+\{\,\cdot\,|\,m^0\!=\!1\}$ in (\ref{eq:clade_split:distrib}) and the strong Markov property of $\whN$ applied at time $\wh T^{y}$ that the pair $\big(\restrict{\whN}{[\wh T^y,\infty)},\restrict{\whX}{[\wh T^y,\infty)}\big)$ when shifted by
$\wh T^y$ to $\big(\shiftrestrict{\whN}{[\wh T^y,\infty)},\shiftrestrict{\whX}{[\wh T^y,\infty)}\big)$ has the same distribution as   $\big(\restrict{\bN}{[T^{y-j},T^{-j})}, j+\restrict{\bX}{[T^{y-j},T^{-j})}\big)$ for $j>y$, so
 \begin{equation}\skewer\left(y,\shiftrestrict{\whN}{[\wh T^y,\infty)},\shiftrestrict{\whX}{[\wh T^y,\infty)}\right) \stackrel{d}{=} \cev\nbeta^y.\vspace{-0.1cm} \label{eq:type-0:restricted_skewer}
 \end{equation}
 Note that $\skewer(y,\whN,\whX)$ equals a single leftmost block, corresponding to the first jump of $\whX$ across level $y$, concatenated with the skewer on the left in \eqref{eq:type-0:restricted_skewer}. By \eqref{eq:transn:PDIP}, that term in \eqref{eq:type-0:restricted_skewer} 
 has the desired distribution.
\end{proof}

Let us now state the simple Markov property with respect to the natural filtration $(\cFI^y,\,y\ge 0)$ generated by the canonical process on $\cC([0,\infty),\IPspace)$.

\begin{proposition}[Simple Markov property under $\BPr^0_\mu$]\label{prop:type-0:simple_Markov}
 Let $\mu$ be a probability distribution on $\IPspace$. Fix $y>0$. Take $\eta,f\colon\cCRI\to [0,\infty)$ measurable, with $\eta$ measurable in $\cFI^y$. Let $\theta_y$ denote the shift operator. Then 
  $\BPr^0_\mu\big[\eta\, f\circ\theta_y \big] = \BPr^0_{\mu}\left[\eta\, \BPr^0_{\nbeta^y}[f]\right].$
\end{proposition}

We prove this result in Appendix \ref{appxtype0}. 

\begin{proposition}[Continuity in the initial state]\label{prop:type-1:cts_in_init_state}
 For $f\colon\IPspace\to [0,\infty)$ bounded and continuous and $z>0$, the map $\beta\mapsto\BPr^0_{\beta}[f(\nbeta^z)]$ is continuous on $(\IPspace,\dI)$.
\end{proposition}
\begin{proof} Consider a sequence $(\beta_n,n\ge 0)$ in $\cI$ such that $\dI(\beta_n,\beta_0)\rightarrow 0$ and 
  $\nbeta_n^y:=\cev\nbeta_n^y\concat\vecc\nbeta_n^y$, $y\ge 0$, associated evolutions as in Definition \ref{constr:type_0}, $n\ge 0$. 
  Then we have $\vecc\nbeta_n^z\rightarrow\vecc\nbeta_0^z$ in distribution, by weak continuity of $\kappa^{(\alpha)}_z$ of Theorem 
  \ref{thmtype1Hunt}. By Skorokhod's representation theorem, we may assume a.s.\ convergence, and we may also assume $\cev\nbeta_0^z$ 
  on the same probability space, independent of $(\vecc\nbeta_n^z,n\ge 0)$. Set $\td\nbeta_n^z:=\cev\nbeta_0^z\concat\vecc\nbeta_n^z$, $n\ge 0$.
  Then $\dI(\td\nbeta_n^z,\td\nbeta_0^z)=\dI(\vecc\nbeta_n^z,\vecc\nbeta_0^z)$ and hence \vspace{-0.1cm} 
 $$\BPr^0_{\beta_n}[f(\nbeta^z)]=\EV[f(\nbeta_n^z)]=\EV[f(\td\nbeta_n^z)]\rightarrow\EV[f(\td\nbeta_0^z)]=\BPr^0_{\beta_0}[f(\nbeta^z)].\vspace{-0.7cm}$$
\end{proof}

\begin{corollary}\label{cor:type-1:cts_init_2}
 Take $m\in\BN$, let $f_1,\ldots,f_m\colon\IPspace\rightarrow[0,\infty)$ be bounded and continuous, and take $0\le z_1<\cdots<z_m$. Then $\beta \mapsto \BPr^0_{\beta}\left[\prod_{i=1}^m f_i(\nbeta^{z_i})\right]$ is continuous.
 %
\end{corollary}

See \cite[Corollary 5.21]{IPPA} for a proof in the type-1 case. The type-0 case is analogous, and we also deduce the strong Markov property, in
the natural filtration $(\cFI^y,\,y\ge 0)$ generated by the canonical process on $\cC([0,\infty),\IPspace)$.

\begin{proposition}[Strong Markov property]\label{prop:strong_Markov}
 Let $\mu$ be a probability distribution on $\IPspace$. Let $Y$ be an a.s.\ finite stopping time in $(\cFI^y,\,y\ge 0)$. Take $\eta,f\colon\cCRI\to [0,\infty)$ measurable, with $\eta$ measurable with respect to $\cFI^Y$. Let $\theta_y$ denote the shift operator. Then 
  $\BPr^0_\mu\big[\eta\, f\circ\theta_Y \big] = \BPr^0_{\mu}\left[\eta\, \BPr^0_{\nbeta^Y}[f]\right]$.
\end{proposition}

\begin{proof}[Proof of Theorem \ref{thmtype0Hunt}] Taking Sharpe's definition of a Hunt process, e.g.\ \cite[Definition A.18]{Li11}, we must check four 
 properties for the semigroup $(\widetilde{\kappa}_y^{(\alpha)},y\ge 0)$ on $(\IPspace,\dI)$. By Proposition \ref{prop:type-0:transn}, the 
 distributions $\BPr^0_\mu$ are distributions of Markov processes with semigroup $(\widetilde{\kappa}_y^{(\alpha)},y\ge 0)$. 
 
 (i) By Theorem \ref{thm:Lusin}, $(\IPspace,\dI)$ is Lusin. 
  
 (ii) From Proposition \ref{prop:type-1:cts_in_init_state}, $(\widetilde{\kappa}_y^{(\alpha)},y\ge 0)$ is continuous in the initial state. 
 
 (iii) By construction before Definition \ref{constr:type_0}, sample paths under $\BPr^0_\mu$ are continuous.
 
 (iv) Proposition \ref{prop:strong_Markov}, the strong Markov property holds under $\BPr^0_\mu$.
 
 We prove 1-self-similarity in Lemma \ref{lem:type-1:scaling}.
\end{proof}


\subsection{Total mass processes}
\label{sec:type-1_gen:cts}


Recall that Theorem \ref{thm:BESQ_total_mass} claims \BESQ[0] and \BESQ[2\alpha] total masses, respectively, for all type-1 and type-0 evolutions, 
regardless of their initial state $\beta\!\in\!\IPspace$. By the path-continuity in Theorems \ref{thmtype1Hunt} and \ref{thmtype0Hunt}, the total mass processes are path-continuous, so we only need to check finite-dimensional marginal distributions.

\begin{proof}[Proof of the type-1 assertion of Theorem \ref{thm:BESQ_total_mass}]
 Let $(\wh\nbeta^y,\,y\geq 0)$ be as in Proposition \ref{prop:type-1:transn}. We proceed by establishing: (i) the desired 1-dimensional marginals; (ii) finite-dimensional marginals. 
  For each of these, we show the property first for $(\big\|\wh\nbeta^y\big\|,y\ge 0)$, then for $(\IPmag{\nbeta^y},y\ge 0)$.
 
 (i) By \cite[p.\ 441]{RevuzYor}, the Laplace transform of the marginal distribution at time $y > 0$ of a \BESQ[0] process $(Z(u),\,u\geq 0)$ starting from $b$ is
 \begin{equation*}
  \EV\left[e^{-\lambda Z(y)}\right] = \exp\left(-\frac{\lambda b}{2y\lambda+1}\right).
 \end{equation*}
 
 We wish to compare this to the Laplace transform of $\big\|\wh\nbeta^y\big\|$. In the notation of Proposition \ref{prop:type-1:transn}, 
 Lemma \ref{lem:LMB}, and Remark \ref{rmk:transn:PDIP}, given that it is not zero, $\big\|\wh\nbeta^y\big\| \stackrel{d}{=} R^y(S^y)+L^y$. As noted in 
 Remark \ref{rmk:transn:PDIP},  
 $R^y(S^y) \sim \GammaDist[\alpha,1/2y]$, with Laplace transform $(2y\lambda\!+\!1)^{-\alpha}$. 
 As for $L^y$, we note that, by
  Lemma \ref{lem:LMB},
 \begin{align}
  \EV\left[e^{-\lambda L^y}\right]
  	&= \frac{2^\alpha y^\alpha}{e^{b/2y}-1}\sum_{n=1}^\infty\frac{1}{n!\Gamma(n-\alpha)}\left(\frac{b}{4y^2}\right)^n\int_0^\infty c^{n-1-\alpha}e^{-(\lambda+1/2y)c}dc\nonumber\\
  	&= \frac{(2y\lambda+1)^\alpha}{e^{b/2y}-1}\sum_{n=1}^\infty\frac{1}{n!}\left(\frac{b}{2y(2y\lambda+1)}\right)^n.\label{eq:clade:LMB_Laplace_1}
 \end{align}
 From \eqref{eq:transn:lifetime}, $\Pr\big\{\wh\nbeta^y=\emptyset\big\} = e^{-b/2y}$. 

Now, to prove $\EV \left[\exp\left(-\lambda\big\|\wh\nbeta\big\|\right)\right] = \EV\left[\exp\left(-\lambda Z(y)\right)\right]$ it suffices to show
   $$\EV\left[e^{-\lambda Z(y)}\right] = e^{-b/2y}+(1-e^{-b/2y})\EV\left[e^{-\lambda R^y(S^y)}\right]\EV\left[e^{-\lambda L^y}\right];$$
  i.e.
  \begin{align}  
  \EV\left[e^{-\lambda L^y}\right] 
    &= \frac{\EV[e^{-\lambda Z(y)}]-e^{-b/2y}}{(1-e^{-b/2y})\EV[e^{-\lambda R^y(S^y)}]}\notag \\
    &= \frac{(2y\lambda+1)^\alpha}{1-e^{-b/2y}}\left(e^{-\lambda b/(2y\lambda+1)}-e^{-b/2y}\right)\label{eq:clade:LMB_Laplace_2}\\
   	&= \frac{(2y\lambda+1)^\alpha}{e^{b/2y}-1}\sum_{n=1}^\infty\frac{1}{n!}\left(\frac{b}{2y(2y\lambda+1)}\right)^n,\notag
  \end{align}
which is the expression in \eqref{eq:clade:LMB_Laplace_1}. Hence, $\big\|\wh\nbeta^y\big\|$ is distributed like $Z(y)$ for fixed $y$. This result extends to general initial states $\beta\in\IPspace$ by way of the independence of the point measures $(\bN_U,\,U\in\beta)$ in Definition \ref{constr:type-1}, and by \cite[Theorem 4.1 (iv)]{PitmYor82}, which states that an arbitrary sum of independent \BESQ[0] processes with summable initial values is a \BESQ[0].
 
 (ii) We now prove equality of finite-dimensional marginal distributions by an induction based on the Markov properties of type-1 evolutions and of \BESQ[0]. For $1$-dimensional marginals, we have proved the result in part (i). We now assume the result holds for all $n$-dimensional marginal distributions starting from any initial distribution. We write $\bQ_a$ to denote the law of a \BESQ[0] process $(Z(y),y\!\ge\!0)$ starting from $a\!>\!0$. For all $0\!\leq\! y_1\!<\!\cdots\!<\!y_n\!<\!y_{n+1}$ and $\lambda_j\!\in\![0,\infty)$, $j\!\in\![n+1]$, we have, by the simple Markov property,
 \begin{align*}
  &\bP^1_{\{(0,a)\}}\!\left[\prod_{j=1}^{n+1}e^{-\lambda_j\IPmag{\nbeta^{y_j}}}\right]
    = \bP^1_{\{(0,a)\}}\!\left[e^{-\lambda_1\IPmag{\nbeta^{y_1}}}\bP^1_{\nbeta^{y_1}}\!\left[\prod_{k=1}^{n}e^{-\lambda_{k+1}\IPmag{\nbeta^{y_{k+1}-y_1}}}\right]\right]\\
    &= \bP^1_{\{(0,a)\}}\!\left[e^{-\lambda_1\IPmag{\nbeta^{y_1}}}\bQ_{\IPmag{\nbeta^{y_1}}}\!\left[\prod_{k=1}^{n}e^{-\lambda_{k+1}Z(y_{k+1}-y_1)}\right]\right]\\
    &= \bQ_a\!\left[e^{-\lambda_1Z(y_1)}\bQ_{Z(y_1)}\!\left[\prod_{k=1}^{n}e^{-\lambda_{k+1}Z(y_{k+1}-y_1)}\right]\right]\!
    = \bQ_a\!\!\left[\prod_{j=1}^{n+1}e^{-\lambda_jZ(y_j)}\right]\!.
 \end{align*}
 Again, this extends to general initial distributions by \cite[Theorem 4.1 (iv)]{PitmYor82} and the independence of the $(\bN_U,\,U\in\beta)$ in Definition \ref{constr:type-1}. This completes the induction step and establishes the equality of finite-dimensional distributions, hence the equality of distributions of the processes.
\end{proof}

We now show that for any $\beta\in\IPspace$, under $\BPr^0_{\beta}$ we have 
$(\IPmag{\nbeta^y},y\geq0)\sim\BESQ[2\alpha]$.

\begin{proof}[Proof of the type-0 assertion of Theorem \ref{thm:BESQ_total_mass}]
 By definition, a type-0 evolution is continuous, so it suffices to show that the total mass process is a Markov process with the same transition kernel as \BESQ[2\alpha]. First assume $\beta = \emptyset$. The marginal distribution of \BESQ[2\alpha] is given in \cite[(50)]{GoinYor03} as
 $$q_y^{(2\alpha)}(0,b)db=\frac{1}{2^\alpha y^\alpha}\frac{1}{\Gamma(\alpha)}b^{\alpha-1}e^{-b/2y}db,$$
 which is the \GammaDist[\alpha,1/2y] distribution. Note that there is no point mass at $b=0$, as 0 is reflecting for \BESQ[2\alpha]. As noted in \eqref{eq:transn:PDIP}, $R^y(S^y)\sim\GammaDist[\alpha,1/2y]$ as well. The extension to finite-dimensional marginals follows as in the proof of the type-1 assertion of Theorem \ref{thm:BESQ_total_mass} above. This completes the proof when $\beta = \emptyset$. Now, by Definition \ref{constr:type_0}, the total mass process of a type-0 evolution from a general initial state is a \BESQ[2\alpha] added to the total mass process of an independent type-1 evolution, which by the type-1 assertion is a \BESQ[0]. Thus, the theorem follows from the well-known additivity property of \BESQ-processes; see e.g.\ \cite[Theorem XI.(1.2)]{RevuzYor}.
\end{proof}

As noted in the introduction, type-1 evolutions can be viewed as branching processes, with each interval giving rise to an independently evolving component. We have shown that type-1 evolutions have \BESQ[0] total mass, which is itself a continuous-state branching process. We have now also shown that type-0 evolutions have \BESQ[2\alpha] total mass; this can be viewed as a continuous-state branching process with immigration \cite{KawaWata71}. Indeed, in the construction in Definition \ref{constr:type_0}, the $\vecc\beta^y$ component can be viewed as all descendants of the initial population $\beta$, while $\cev\beta^y$ includes all descendants of subsequent immigrants. See \cite{Paper0,IPPA} for more discussion of the branching process interpretation of the scaffolding-and-spindles construction.

We note one additional connection between type-0 and type-1 evolutions.

\begin{proposition}\label{prop:remaining_type-0}
 Fix $\beta\in\IPspace$, $a>0$, and let $\gamma := \{(0,a)\}\concat\beta$. Consider an independent pair $(\bff,\,(\nbeta^y,\,y\ge0))$ comprising a \BESQA\ and a type-0 evolution with respective initial states $a$ and $\gamma$. Let $(\td\nbeta^y,\,y\ge0)$ denote a type-1 evolution starting from $\gamma$ and let $Y$ denote the lifetime of the original leftmost block in $(\td\nbeta^y,\,y\ge 0)$. Then $Y$ is an $(\cFI^y)$-stopping time and \vspace{-.2cm}
 $$(\td\nbeta^y,\ y\in [0,Y))  \stackrel{d}{=} \big(\big\{(0,\bff(y))\big\} \concat \nbeta^y,\ y\in[0,\life(\bff)) \big). \vspace{-.2cm}$$
\end{proposition}

\begin{proof}
 We begin with $\beta = \emptyset$. Let $\bN_{\gamma}\sim \mClade^+(\,\cdot\;|\;m^0=a)$. By (\ref{eq:clade_split:distrib}), $\bN_{\gamma}$ is distributed like $\Dirac{0,\bff}+\restrict{\bN}{[0,T^{-\life(\bff)})}$, where $\bff$ is a \BESQA\ starting from $a$, independent of $\bN$. Comparing this to the construction of the type-0 evolution $(\cev\nbeta^y,\,y\ge 0)$ around \eqref{eq:type-0:consistency} proves the claimed identity in this case. For other values of $\beta$, the type-1 and type-0 evolutions with respective laws $\BPr^1_{\gamma}$ and $\BPr^0_{\beta}$ may be constructed by concatenating each of the evolutions in the previous case with an independent type-1 evolution with law $\BPr^1_{\beta}$. 
\end{proof}

We define $L\colon \HIPspace\to [0,\infty)$ to map an interval partition to the mass of its leftmost block, or 0 if none exists. Let $R\colon\HIPspace\to [0,\infty)$ denote the remaining mass, $R(\nbeta) = \IPmag{\nbeta}-L(\nbeta)$. It is not hard to see that these maps are measurable.

\begin{corollary}\label{cor:remaining_BESQ}
 Let $(\nbeta^y,\,y\ge0)$ be a type-1 evolution. Let $Y := \inf\{y > 0\colon L(\nbeta^y-) = 0\}$. Then $(L(\nbeta^y),\,y\in [0,Y))$ and $(R(\nbeta^y),\,y\in [0,Y))$ are jointly distributed as an independent \BESQA\ and \BESQ[2\alpha], stopped when the \BESQA\ hits zero.
\end{corollary}

\begin{proof} This follows from Proposition \ref{prop:remaining_type-0} and Theorem \ref{thm:BESQ_total_mass}.\end{proof}

\def\cvN{\cev\bN}
\def\cvX{\cev\bX}


\section{De-Poissonization and stationary IP-evolutions}\label{sec:dePoissonization}

\subsection{Pseudo-stationarity of type-1 evolutions and type-0 evolutions}\label{sec:pseudostat}

We prove Theorem \ref{thm:pseudostat} in stages over the course of this section by considering different cases for the law of the initial total
mass $Z(0)$. Later, we demonstrate a stronger form of this theorem in Theorem \ref{thm:pseudostat_strong}.

\begin{proposition}\label{prop:pseudostat:exp}
 Suppose that in the setting of the type-1, respectively type-0, assertion of Theorem \ref{thm:pseudostat} we have $Z(0)\sim \ExpDist[\rho]$, respectively $Z(0)\sim\GammaDist[\alpha,\rho]$, for some $\rho\in(0,\infty)$. Then the conditional law of $\nbeta^y$, given $\{\nbeta^y\neq\emptyset\}$, equals the unconditional law of $(2y\rho+1)Z(0)\scaleI\ol\beta$.
\end{proposition}

\begin{proof}
 We begin with the type-1 case. 
 We prove this by separately comparing the Laplace transforms of the leftmost blocks of the two interval partitions, comparing Laplace exponents of the subordinators of remaining block masses, and confirming that in each partition the leftmost block is independent of the remaining blocks. This is done in three steps.
 
 \emph{Step 1}. Following Proposition \ref{prop:PDIP} \ref{item:PDIP:Stable}, we may represent $\ol\beta$ as 
 $\{(0,1-G)\}\concat \big(G\scaleI \ol\gamma\big)$, where $G\sim\BetaDist[\alpha,1-\alpha]$ is independent of $\ol\gamma\sim\PDIP[\alpha,\alpha]$. 
 Let $\beta := Z(0)\scaleI\ol\beta$ and $\gamma := Z(0)G\scaleI\ol\gamma$. We denote the leftmost block of $\beta$ by $U_0 := (0,Z(0)(1-G))$. Since 
 $Z(0) \sim \ExpDist[\rho]$, the masses $\Leb(U_0)$ and $\IPmag{\gamma}$ are independent \GammaDist[1-\alpha,\rho] and \GammaDist[\alpha,\rho] 
 variables. 
 
 By Proposition \ref{prop:PDIP} \ref{item:PDIP:Stable}, the partition $\gamma$ corresponds to the range of a \Stable[\alpha] subordinator stopped
 prior to crossing an independent random level $S\sim \ExpDist[\rho]$. This stopping corresponds to thinning the Poisson random measure of jumps, tilting the L\'evy measure by a factor of $e^{-\rho x}$.  


 By the branching property of the transition kernel $\kappa_y^{(\alpha)}$ of type-1 evolutions, we can write 
 $\nbeta^y=\Concat_{U\in\beta}\gamma_U$ with independent interval partitions $\gamma_U\sim\mu_{\Leb(U),1/2y}^{(\alpha)}$, $U\in\beta$. We also
 write $\gamma^y=\Concat_{U\in\gamma}\gamma_U$ so that $\nbeta^y=\gamma_{U_0}\concat\gamma^y$.

 If we consider each $\gamma_U$, $U\in\gamma$, as a mark of the corresponding jump of the tilted subordinator, these marks form a 
 \PRM[\Leb\otimes\mu^*] on $[0,\infty)\times\IPspace$ with 
 $$\mu^*=\int_{(0,\infty)}\mu_{b,1/2y}^{(\alpha)}\frac{\alpha}{\Gamma(1-\alpha)}b^{-1-\alpha}e^{-\rho b}db,$$ 
 stopped at an independent exponential time $E$ of rate
 \begin{equation*}
  \int_0^\infty (1-e^{-\rho b})\frac{\alpha}{\Gamma(1-\alpha)}b^{-1-\alpha}db = \rho^\alpha.
 \end{equation*}
  Since $\mu_{b,1/2y}^{(\alpha)}\{\emptyset\}\!=\!e^{-b/2y}$, the rate of non-empty interval partitions in a \PRM[\Leb\otimes\mu^*] is
  $$
   \int_0^\infty\!(1\!-\!e^{-b/2y})\frac{\alpha}{\Gamma(1\!-\!\alpha)}e^{-\rho b}b^{-1-\alpha}db
  	= \!\left(\!\rho\!+\!\frac{1}{2y}\right)^\alpha\!\!-\rho^\alpha = \!\left(\!\left(\frac{2y\rho\!+\!1}{2y\rho}\right)^\alpha\!\!-\!1\!\right)\rho^\alpha.
  $$
 By competing exponential clocks, the probability of seeing no non-empty interval partition before the independent 
 $\ExpDist\big(\rho^\alpha\big)$ time is
 \begin{equation}\label{eq:p_s_E:competing_clocks}
  \Pr\{\gamma^y=\emptyset\} = \frac{\rho^\alpha}{\left(((2y\rho+1)/(2y\rho))^\alpha-1\right)\rho^\alpha+\rho^\alpha} = \left(\frac{2y\rho}{2y\rho+1}\right)^\alpha.
 \end{equation}
  Since $\Leb(U_0)\sim\GammaDist[1-\alpha ,\rho]$, we also have
 \begin{equation}\label{eq:p_s_E:LMB0_survival}
  \Pr\{\gamma_{U_0}\!=\!\emptyset\} = \int_0^\infty\! e^{-b/2y} \frac{1}{\Gamma\left(1\!-\!\alpha \right)}\rho^{1-\alpha}b^{-\alpha}e^{-\rho b}db = \left(\frac{2y\rho}{2y\rho\!+\!1}\right)^{1-\alpha}.
 \end{equation}
 Thus, by the independence of $\gamma_{U_0}$ and $\gamma^y$,
 \begin{equation}\label{eq:p_s_E:survival}
  \Pr\{\nbeta^y\neq\emptyset\} = 1-\left(\frac{2y\rho}{2y\rho+1}\right)^{\alpha+1-\alpha} = \frac{1}{2y\rho+1}.\pagebreak[2]
 \end{equation}
 
 \emph{Step 2}. First, we compute the Laplace transform of the leftmost block mass $L(\nbeta^y)$ on the event that it arises from $\gamma_{U_0}$; then, we compute it on the event that it arises from one of the interval partitions that make up $\gamma^y$. 
 The Laplace transform $\mu_{b,1/2y}^{(\alpha)}[e^{-\lambda L}]$, may be read from \eqref{LMBintro}. We multiply \eqref{LMBintro} by 
 the probability $1-e^{-b/2y}$ of being non-empty and integrate against the \GammaDist[1-\alpha ,\rho] law of $\Leb(U_0)$:
 \begin{equation*}
 \begin{split}
  &\EV\left[ e^{-\lambda L(\gamma_{U_0})}\cf\{\gamma_{U_0}\neq\emptyset\}\right] = \EV\left[\mu_{\Leb(U_0),1/2y}^{(\alpha)}\big[ e^{-\lambda L}\cf\{L\neq 0\}\big]\! \right]\\
  &= \int_0^\infty\frac{(2y\lambda+1)^\alpha}{1-e^{-b/2y}}\left(e^{-\lambda b/(2y\lambda+1)}-e^{-b/2y}\right)\left(1-e^{-b/2y}\right)\frac{\rho^{1-\alpha}b^{-\alpha}}{\Gamma(1-\alpha)}e^{-\rho b}db\\ 
  &= \left(2y\lambda+1\right)^\alpha\left(\left(\frac{\rho(2y\lambda+1)}{\rho(2y\lambda+1)+\lambda}\right)^{1-\alpha}\!-\left(\frac{2y\rho}{2y\rho+1}\right)^{1-\alpha}\right).
 \end{split}
 \end{equation*}
 The non-empty partitions in $\Concat_{U\in\gamma}\gamma_U$ form an i.i.d.\ sequence with law $\mu^*(\,\cdot\,|\,L\neq 0)$. Then the contribution to $\EV[\exp(-\lambda L(\beta^y))]$ follows similarly:
 \begin{align*}
  &\EV\left[ e^{-\lambda L(\gamma^y)}\cf\{L(\gamma^y)> 0=L(\gamma_{U_0})\} \right]\notag\\
  & = \frac{\Pr\{L(\gamma^y)\!>\! 0 \!=\! L(\gamma_{U_0})\}}{\mu^*\{ L\neq 0\}} \mu^*\left[ e^{-\lambda L}\cf\{L\neq 0\}\right]\notag\\
  &=\! \frac{\left(\frac{2y\rho}{2y\rho+1}\right)^{1-\alpha}\left(\!1\!-\!\left(\frac{2y\rho}{2y\rho+1}\right)^\alpha\right)}{\left(\left(\frac{2y\rho+1}{2y\rho}\right)^\alpha-1\right)\rho^\alpha}\\
     &\qquad\qquad\qquad\times\int_0^\infty\!\!\!(2y\lambda\!+\!1)^\alpha\left(e^{-\lambda b/(2y\lambda+1)}\!-\!e^{-b/2y}\right)\frac{\alpha b^{-1-\alpha}e^{-\rho b}}{\Gamma(1-\alpha)}db\notag\\[-1.3cm]
  \end{align*}
  \begin{equation}= \left(2y\lambda+1\right)^\alpha\left(\left(\frac{2y\rho}{2y\rho+1}\right)^{1-\alpha}-\left(\frac{2y\rho}{2y\rho+1}\right)\left(\frac{\lambda + \rho(2y\lambda+1)}{\rho(2y\lambda+1)}\right)^\alpha\right).\label{eq:p_s_E:LMB_Laplace_2}
 \end{equation}
 Adding these terms and dividing by the formula for $\Pr\{\nbeta^y\neq\emptyset\}$ in \eqref{eq:p_s_E:survival}, we get\vspace{-0.1cm}
 \begin{equation*}
 \begin{split}
  &\EV\left[e^{-\lambda L(\nbeta^y)}\middle|\nbeta^y\neq\emptyset\right]\\
  &\!=\left(2y\lambda\!+\!1\right)^\alpha \left(\!  (2y\rho\!+\!1)\left(\frac{\rho(2y\lambda+1)}{\rho(2y\lambda\!+\!1)+\lambda}\right)^{1-\alpha}\! - 2y\rho\left(\frac{\rho(2y\lambda\!+\!1)+\lambda}{\rho(2y\lambda+1)}\right)^\alpha  \right)\\
  &\!= \left(\frac{\rho}{\rho(2y\lambda\!+\!1)+\lambda}\right)^{1-\alpha} = \left(\frac{\rho/(2y\rho+1)}{(\rho/(2y\rho\!+\!1))+\lambda}\right)^{1-\alpha}\,.
 \end{split}
 \end{equation*}
 This is the Laplace transform of $(2y\rho\!+\!1)\Leb(U_0)\!\sim\!\GammaDist[1\!-\!\alpha,\rho/(2y\rho\!+\!1)]$.\pagebreak[2]
 
 \emph{Step 3}. As noted after Theorem \ref{thmtype1Hunt}, 
   we have $\gamma_U\neq\emptyset$ for at most finitely many $U\in\beta$, and 
   $\nbeta^y$ equals the concatenation of $\gamma_{U_0}$ with those finitely many $\gamma_U$. Recall from Step 1 that we may view the
   $\gamma_U$, $U\in\gamma$, as points of a \PRM[\Leb\otimes\mu^*] stopped at an independent $\ExpDist\big(\rho^\alpha\big)$ time. If we condition 
   on $\{\nbeta^y\neq\emptyset\}$ then, following the competing exponential clocks argument around \eqref{eq:p_s_E:competing_clocks}, we may view 
   the non-empty $\gamma_U$ beyond the leftmost non-empty interval partition as coming from an infinite sequence of independent interval partitions 
   with distribution $\mu^*(\,\cdot\,|\,L\neq 0)$ stopped after an independent number $D\sim\GeomDist\big((2y\rho/(2y\rho+1))^\alpha\big)$ of 
   clades, where $D$ can equal 0.
 
   In the notation of Remark \ref{rmk:transn:PDIP}, the $i^{\text{th}}$ interval partition with distribution $\mu^*(\,\cdot\,|\,L\neq 0)$ contributes 
   its own leftmost block $(0,L^y_i)$, followed by masses from an independent subordinator $R^y_i$ stopped at an independent 
   \ExpDist[(2y)^{-\alpha}] time $S^y_i$, for all $i\ge 1$. The leftmost non-empty interval partition contributes a special leftmost block studied 
   in Step 2, and independent masses from $(R_0^y,S_0^y)$, as for $i\ge 1$. We call the masses from $(R^y_i,S^y_i)$, $i\ge 0$, the ``remaining 
   masses.'' 
   So, we may view the masses in $\nbeta^y$ beyond the far leftmost as arising from an alternating sequence of remaining masses of interval 
   partitions $i=0,\ldots,D$ and the leftmost blocks for $i=1,\ldots,D$.
 
   The stopped $R_i^y$ from all $i\ge 0$ can be combined to capture all remaining masses in a single unstopped subordinator $R_{\rm rem}^y$ with 
   Laplace exponent $\Phi_{1/2y}$ as in Remark \ref{rmk:transn:PDIP}, independent of 
   $(S_i^y,i\ge 0)$ and $D$, and hence of $\widetilde{S}^y:=S_0^y+\cdots+S_D^y$, which is exponential with parameter 
     $$(2y)^{-\alpha}(2y\rho/(2y\rho+1))^\alpha = (\rho/(2y\rho + 1))^\alpha.$$
   This is the time that corresponds to stopping $R_{\rm rem}^y$ after the $D^{\text{th}}$ non-empty interval partition. This independence also 
   yields the independence of $R_{\rm rem}^y$ from the subordinator that has jumps of sizes $L_i^y$ at times $S_0+\cdots+S_{i-1}^y$, $i\ge 1$, with 
   Laplace exponent $(2y)^{-\alpha}\mu^*[ 1-e^{-\lambda L}\,|\,L\neq 0]$. Note
 \begin{equation*}
 \begin{split}
  \mu^*\left[ 1\!-\!e^{-\lambda L}\middle|L\!\neq\!0\right] &= 1 - \frac{\mu^*\left[e^{-\lambda L}\middle|L\!\neq 0\right]}{\mu^*\{L\neq 0\}}\\
  	&= 1-(2y\lambda+1)^\alpha\frac{\left(\frac{2y\rho+1}{2y\rho}\right)^\alpha-\left(\frac{\rho+(2y\rho+1)\lambda}{\rho(1+2\lambda y)}\right)^\alpha}{\left(\frac{2y\rho+1}{2y\rho}\right)^\alpha-1},
 \end{split}
 \end{equation*}
where the expression is a multiple of that in \eqref{eq:p_s_E:LMB_Laplace_2}. By an elementary thinning argument, this subordinator stopped after the $D^{\text{th}}$ jump can be viewed as a subordinator with Laplace exponent
 \begin{equation*}
 \begin{split}
  &\Phi^y_{\textnormal{LMB}}(\lambda) = \frac{1}{(2y)^\alpha}\left(1-\left(\frac{2y\rho}{2y\rho+1}\right)^\alpha\right)\mu^*\left[ 1-e^{-\lambda L}\middle|L\neq 0\right]\\
    &\!= \frac{1}{(2y)^\alpha}\left(\!1\!-\! \left(\frac{2y\rho}{2y\rho\!+\!1}\right)^{\!\alpha}\! - (2y\lambda \!+\! 1)^\alpha\left(\!1 \!-\! \left(\frac{2y\rho}{2y\rho \!+\! 1}\right)^{\!\alpha}\!\left(\frac{\rho \!+\! (2y\rho \!+\! 1)\lambda}{\rho(2y\lambda\! +\! 1)}\right)^{\!\alpha}\right)\!\right)\!.
 \end{split}
 \end{equation*}
stopped at the independent time $\widetilde{S}^y\sim\ExpDist\big((\rho/(2y\rho + 1))^\alpha\big)$.

 Putting these pieces together, $\Phi_{1/2y}(\lambda)+\Phi^y_{\textnormal{LMB}}(\lambda)$ is given by
 \begin{equation*}
 \begin{split}
  &\frac{1}{(2y)^\alpha}\!\left(\!\!(2\lambda y\!+\!1)^\alpha\!\!-\!1\!+\!1\!-\!\left(\!\frac{2y\rho}{2y\rho\!+\!1}\!\right)^{\!\!\alpha}\!\!\!-\!(2y\lambda\!+\!1)^\alpha\!+\!\left(\!\frac{2y}{2y\rho\!+\!1}\!\right)^{\!\!\alpha}\!\!\!(\rho\!+\!(2y\rho\!+\!1)\lambda)^\alpha\!\!\right)\\
  &\!=\! \left(\!\lambda\!+\!\frac{\rho}{2y\rho+1}\!\right)^\alpha\!\!\!-\!\left(\!\frac{\rho}{2y\rho+1}\!\right)^\alpha\!\! =\! \int_0^\infty\!\!(1\!-\!e^{-\lambda x})\frac{\alpha}{\Gamma(1\!-\!\alpha)}e^{-x\rho/(2y\rho+1)}x^{-1-\alpha}dx.
 \end{split}
 \end{equation*}
 The last expression above is the Laplace transform of a subordinator that, when stopped at an independent $\ExpDist\big((\rho/(2y\rho + 1))^\alpha\big)$ time, corresponds as in \eqref{eq:transn:PDIP} 
  to a \PDIP[\alpha,\alpha] scaled by a \GammaDist[\alpha,\rho/(2y\rho+1)] variable that is independent of the \PDIP[\alpha,\alpha]. Putting this together with the result of Step 2 and the independence, in both $\nbeta^y$ and $\beta$, of the leftmost block from the rest, we conclude that $\nbeta^y$ is distributed like a \PDIP[\alpha,0] scaled by an independent \ExpDist[\rho/(2y\rho+1)] variable, as desired.
 
 Looking at the semigroup $\widetilde{\kappa}_y^{(\alpha)}$ defined in Introduction, Step 3 above also proves the claim for the type-0 evolution.
\end{proof}

\begin{lemma}[Scaling invariance of type-1 and type-0 evolutions]\label{lem:type-1:scaling}
 Fix $c>0$. If $(\nbeta^y,\,y\geq 0)$ is a type-1 (respectively type-0) evolution then so is $(c\scaleI \nbeta^{y/c},\,y\geq 0)$.
\end{lemma}

\begin{proof}
 As scaling clearly preserves path properties and Markovianity, we compare the semigroups of $(\nbeta^y,\,y\geq 0)$ and 
 $(c\scaleI \nbeta^{y/c},\,y\geq 0)$. It suffices to show that $cL^{(\alpha)}_{b/c,rc}\ed L^{(\alpha)}_{b,r}$ and $cY^{(\alpha)}_{rc}(s/c^\alpha)\ed Y_r^{(\alpha)}(s)$:
 \begin{align*}
   &\EV\!\left[\!\exp\!\left(\!-\!\lambda cL^{(\alpha)}_{b/c,rc}\right)\!\right]=\left(\!\frac{rc\!+\!\lambda c}{rc}\!\right)^{\!\alpha}\frac{e^{(b/c)(rc)^2/(rc+\lambda c)}\!-\!1}{e^{(b/c)rc}-1}=\EV\!\left[\!\exp\!\left(\!-\!\lambda L_{b,r}^{(\alpha)}\!\right)\!\right]\!,\\
   &\EV\left[\exp\left(-\lambda cY^{(\alpha)}_{rc}(s/c^\alpha)\right)\right]
   =e^{-(s/c^\alpha)(rc+\lambda c)^\alpha)}=\EV\left[\exp\left(-\lambda Y_r^{(\alpha)}(s)\right)\right].\\[-1.2cm]
 \end{align*}
\end{proof}

We can now invert Laplace transforms to deduce the following.

\begin{proposition}\label{prop:pseudostat:fixed}
 Both assertions of Theorem \ref{thm:pseudostat} hold if $Z(0) = b \geq 0$ is fixed.
\end{proposition}

\begin{proof}
 \emph{Type-1 case}. The case $b=0$ is trivial. The transition density of \BESQ[0] can be read from \cite[equation (51)]{GoinYor03}. For $Z(0) = b > 0$ we get $\Pr\{Z(y)=0\} = e^{-b/2y}$ and on $(0,\infty)$
 \begin{equation}\label{eq:BESQ_0_transn}
  \Pr\{Z(y)\in dc\} = \frac{1}{2y}\sqrt{\frac{b}{c}}\exp\left(-\frac{b+c}{2y}\right)I_1\left(\frac{\sqrt{bc}}{y}\right)db,
 \end{equation}
 where $I_1$ is the Bessel function. Let $(\nbeta^y_1,\,y\geq 0)$ denote a type-1 evolution with initial state $\ol\beta\sim\PDIP[\alpha,0]$. For $c>0$ and $y\geq 0$ let $\nbeta^y_{c} := c \scaleI \nbeta^{y/c}_1$; by Lemma \ref{lem:type-1:scaling} this is a type-1 evolution. For $\rho>0$, let $Z_{\rho}\sim \ExpDist[\rho]$ be independent of $(\nbeta^y_1)$. By Proposition \ref{prop:pseudostat:exp}, for all $\rho>0$ and all bounded continuous $f\colon\IPspace\rightarrow[0,\infty)$ with $f(\emptyset)=0$ we have
 \begin{align}
  \int_0^\infty e^{-\rho b}\EV[f(\nbeta^y_b)]db &= \frac{1}{\rho}\EV\left[f\left(\nbeta^y_{Z_\rho}\right)\right]\notag\\
     & = \frac{1}{2y\rho+1}\int_0^\infty e^{-\rho b}\EV[f((2y\rho+1)b\scaleI\ol\beta)]db\notag\\
     &= \frac{1}{(2y\rho+1)^2}\int_0^\infty e^{-\rho c/(2y\rho+1)}\EV[f(c\scaleI\ol\beta)]dc.\label{lt1}
 \end{align}
 We want to identify this Laplace transform as the claimed
 \begin{align}
 &\int_0^\infty e^{-\rho b}\int_0^\infty\frac{1}{2y}\sqrt{\frac{b}{c}}\exp\left(-\frac{b+c}{2y}\right)I_1\left(\frac{\sqrt{bc}}{y}\right)\EV[f(c\scaleI\ol\beta)]dcdb\nonumber\\
    &= \int_0^\infty\frac{1}{2y}\frac{1}{\sqrt{c}}e^{-c/2y}\EV[f(c\scaleI\ol\beta)]\int_0^\infty \sqrt{b}e^{-(\rho+(1/2y))b}I_1\left(\frac{\sqrt{bc}}{y}\right)dbdc\nonumber\\
    &= \!\!\int_0^\infty\!\!\!\frac{1}{2y}\frac{1}{\sqrt{c}}e^{-c/2y}\EV[f(c\scaleI\ol\beta)]\sqrt{\!\frac{c}{y^2}}\frac{1}{2(\rho\!+\!(1/2y))^2}\exp\!\left(\!\frac{c}{4y^2(\rho\!+\!(1/2y))}\!\right)\!dc,\label{lt2}
 \end{align}
 where we use well-known formulas for integrals involving the Bessel function $I_1$: specifically, the normalization of \eqref{eq:BESQ_0_transn} and differentiation $d/dx$ under the integral sign give rise to
 \begin{align*}&\int_0^\infty \frac{1}{\sqrt{u}}e^{-xu}I_1(\sqrt{uv})du = \frac{2}{\sqrt{v}}\left(e^{v/4x}-1\right)\\ \text{and}\quad&
 \int_0^\infty \sqrt{u}e^{-xu}I_1(\sqrt{uv})du = \frac{\sqrt{v}}{2x^2}e^{v/4x}
 \end{align*}
 for all $x,v\in(0,\infty)$. As desired, \eqref{lt1} and \eqref{lt2} can easily be seen to be equal. By continuity in the initial condition for
 type-1 evolutions (proved in \cite[Proposition 5.20]{IPPA}), the map $b\mapsto\EV[f(\nbeta^y_b)]$ is continuous, so for all $b,y\in(0,\infty)$, 
 we conclude that $\EV\left[f(\nbeta^y_b)\right]$ equals
  $$\int_0^\infty\frac{1}{2y}\sqrt{\frac{b}{c}}e^{-(b+c)/2y}I_1\left(\frac{\sqrt{bc}}{y}\right)\EV[f(c\scaleI\ol\beta)]dc = \EV\left[f(Z(y)\scaleI\ol\beta)\right].$$
 Equality in distribution follows since, as noted in Theorem \ref{thm:Lusin}, $(\IPspace,\dI)$ is Lusin, so bounded continuous functions separate points in $\IPspace$.
 
 \emph{Type-0 case}. We begin with a similar argument, making the obvious adjustments of letting $(\nbeta^y_b,\,y\geq 0)$ denote a type-0 evolution for $b>0$, taking $\ol\beta\sim\PDIP[\alpha,\alpha]$, and setting $Z_{\rho}\sim\GammaDist[\alpha,\rho]$. Then Proposition \ref{prop:pseudostat:exp} gives
  \begin{equation*}
  \begin{split}
  &\int_0^\infty\frac{1}{\Gamma(\alpha)}\rho^\alpha b^{\alpha-1}e^{-\rho b}\EV[f(\nbeta^y_b)]db\\ &= \EV\left[f(\nbeta^{y}_{Z_\rho})\right]
    = \int_0^\infty\frac{1}{\Gamma(\alpha)}\rho^\alpha b^{\alpha-1}e^{-\rho b}\EV[f((2\rho y+1)b\scaleI\ol\beta)]db\\
    &= \int_0^\infty\frac{1}{\Gamma(\alpha)}c^{\alpha-1}\left(\frac{\rho}{2\rho y+1}\right)^\alpha e^{-\rho c/(2\rho y+1)}\EV[f(c\scaleI\ol\beta)]dc.
  \end{split}
  \end{equation*}
  Since the total mass evolution is \BESQ[2\alpha], 
  considering $f$ of the form $g\left(\IPmag{\,\cdot\,}\right)$ gives
  $$\frac{1}{\Gamma(\alpha)}c^{\alpha-1}\left(\frac{\rho}{2\rho y+1}\right)^\alpha e^{-\rho c/(2\rho y+1)} = \int_0^\infty\frac{1}{\Gamma(\alpha)}\rho^\alpha b^{\alpha-1}e^{-\rho b}q_y^{(2\alpha)}(b,c)db,$$
  where $q^{(2\alpha)}_y$ is the time-$y$ transition density of \BESQ[2\alpha]. Hence, after the cancellation of $\rho^\alpha/\Gamma(\alpha)$,
  $$\int_0^\infty b^{\alpha-1}e^{-\rho b}\EV[f(\nbeta^y_b)]db = \int_0^\infty b^{\alpha-1}e^{-\rho b}\int_0^\infty q_y^{(2\alpha)}(b,c)\EV[f(c\scaleI\ol\beta)]dcdb.$$
  Since this holds for all $\rho>0$, we conclude by uniqueness of Laplace transforms that\vspace{-0.3cm}
  $$b^{\alpha-1}\EV[f(\nbeta^y_b)] = b^{\alpha-1}\int_0^\infty q_y^{(2\alpha)}(b,c)\EV[f(c\scaleI\ol\beta)]dc,\vspace{-0.1cm}$$
  first for Lebesgue-a.e.\ $b>0$, then for every $b>0$ by continuity. Again, this gives equality in distribution, since $(\IPspace,\dI)$ is Lusin.
\end{proof}

\begin{proof}[Proof of Theorem \ref{thm:pseudostat}]
 The arguments for types 0 and 1 are identical. We showed in the proof of Proposition \ref{prop:pseudostat:fixed} that $\nbeta^y_b$ has the same distribution as $Z(y)\scaleI\ol\beta$ for all $Z(0)= b\ge 0$. Now consider any random $Z(0)$ independent of $(\nbeta^y_1,\,y\ge 0)$.\vspace{-0.1cm}
 \begin{equation*}
 \begin{split}
  \EV\left[f\left(\nbeta^y_{Z(0)}\right)\right] &= \int_0^\infty\EV[f(\nbeta^y_b)]\Pr\{Z(0)\in db\}\\
  	&= \int_0^{\infty}\EV[f(Z(y)\scaleI\ol\beta)\;|\;Z(0)=b]\Pr\{Z(0)\in db\}\\ &= \EV[f(Z(y)\scaleI\ol\beta)].
 \end{split}
 \end{equation*}\vspace{-20pt}
 
\end{proof}

\subsection{$\!\!$Diffusions on $(\IPspace,\dI)$, stationary with \PDIP $\left(\alpha,0\right)$ and \PDIP $\left(\alpha,\alpha\right)$ laws}

\def\bsa{\boldsymbol{\nbeta}}
\def\Absa{\rho_{\bsa}}
\def\dPF{\ol\cF_{\IPspace}}

Throughout this section we write $\IPspace_1 := \{\gamma\in\IPspace\colon\IPmag{\gamma}=1\}$. Recall the de-Poissonization transformation of Theorem \ref{thm:stationary}. In this section, we prove that theorem. We slightly update our earlier notation.

\begin{definition}[De-Poissonization]\label{def:dePois}
 For $\bsa = (\nbeta^y,y\!\geq\! 0) \in \cC([0,\infty),\IPspace)$ with $\nbeta^0\neq\emptyset$, we set for all $u\ge 0$\vspace{-0.1cm}
 $$\ol{\nbeta}^u := \IPmag{\nbeta^{\Absa(u)}}^{-1}\scaleI\nbeta^{\Absa(u)}, 
  \text{ where }  \Absa(u) = \inf\left\{  y\!\ge\! 0\colon\!\! \int_0^y\! \IPmag{\nbeta^z}^{-1} dz \!>\! u  \right\}.\vspace{-0.1cm}$$
 We call the map $D$ sending $(\nbeta^y,\,y\geq 0) \mapsto (\ol\nbeta^u,\,u\geq 0)$ the \emph{de-Poissonization map} and we call $(\ol\nbeta^u,\,u\geq 0)$ the de-Poissonized process.
\end{definition}

\begin{proposition}\label{prop:dePois:time_change}
 For $\bsa = (\nbeta^y,\,y\geq 0)$ a type-0 or type-1 evolution with initial state $\beta\neq\emptyset$, the time-change $\Absa$ is continuous and strictly increasing, and $\lim_{u\upto\infty}\Absa(u) = \inf\{y > 0\colon \nbeta^y = \emptyset\}$.
\end{proposition}

This is really an assertion about integrals of inverses of the \BESQ\ total mass processes of Theorem \ref{thm:BESQ_total_mass}, and in that setting it is common knowledge. It can be read, for example, from \cite[p.\ 314-5]{GoinYor03}. 
The a.s.\ path-continuity claimed in Theorem \ref{thm:stationary} follows from Proposition \ref{prop:dePois:time_change} and the path-continuity of the type-1 and type-0 evolutions. It remains to prove the claimed Markov property and stationary distributions.



Take $\bsa = (\nbeta^y,\,y\geq 0)\in\cCRI$ with $\nbeta^0\neq\emptyset$. By changes of variables we see that\vspace{-0.3cm}
$$D(\bsa) = D(c\scaleI\nbeta^{y/c},\,y\ge 0)\qquad\mbox{for all }c>0.$$
Consequently, a type-1 (respectively, type-0) evolution starting from $c\scaleI\beta$ has the same de-Poissonized process as a type-1 (resp.\ type-0) evolution starting from $\beta$. Thus, for laws $\mu$ on $\IPspace\setminus\{\emptyset\}$ we can denote by $\ol\BPr^1_{\ol{\mu}}$ (resp.\ $\ol\BPr^0_{\ol{\mu}}$) the distribution of a de-Poissonized type-1 (resp.\ type-0) evolution starting from the initial distribution $\ol{\mu}$ of $\IPmag{\beta}^{-1}\scaleI\beta$, where $\beta\sim\mu$.

Recall the natural filtration $(\cFI^y,\,y\ge0)$ on $\cC([0,\infty),\IPspace)$ used in the Markov properties of type-0 and type-1 evolutions, such 
as Propositions \ref{prop:type-0:simple_Markov} and \ref{prop:strong_Markov}. Since $(\Absa(u),\,u\ge 0)$ is an increasing family of $(\cFI^y)$-stopping times, we can introduce the time-changed filtration $\dPF^u = \cFI^{\Absa(u)}$, $u\ge 0$. 

\begin{proposition}[Strong Markov property of de-Poissonized evolutions]\label{prop:dePois:Markov}
 Let $\ol\mu$ be a probability distribution on $\IPspace_1$. Let $U$ be an a.s.\ finite $(\dPF^u)$-stopping time. Let $\ol\eta$ and $\ol f$ be non-negative, measurable functions on $\cCRIi$, with $\ol\eta$ being $\dPF^U$-measurable. Then
 \begin{equation*}
  \ol\BPr^1_{\ol\mu}\big[\ol\eta\,\ol f\circ\theta_U\big] = \ol\BPr^1_{\ol\mu}\left[\ol\eta\, \ol\BPr^1_{\ol\nbeta^U}\left[\ol f\,\right]\right]\!,
  \quad \text{and} \quad \ol\BPr^0_{\ol\mu}\big[\ol\eta\,\ol f\circ\theta_U\big] = \ol\BPr^0_{\ol\mu}\left[\ol\eta\, \ol\BPr^0_{\ol\nbeta^U}\left[\ol f\,\right]\right]\!.
 \end{equation*}
\end{proposition}

\begin{proof}
 We begin by proving the type-0 assertion. In fact, we prove a stronger statement. Consider the canonical process $\bsa = (\nbeta^y,\,y\ge0)$ under $\BPr^0_{\ol\mu}$, so $D(\bsa)$ is a de-Poissonized type-0 evolution with law $\ol\BPr^0_{\ol\mu}$. We show the strong Markov property of $D(\bsa)$ with respect to $(\dPF^u,u\!\ge\!0)$.
 
 Let $V$ be an a.s.\ finite $(\dPF^u)$-stopping time. Consider $\eta\colon\cCRI\to [0,\infty)$ measurable in $\dPF^{V}$ and set $f := \ol f\circ D$, where $\ol f$ is as in the statement above. 
 Let $Y := \Absa(V)$. Since $\Absa$ is $(\dPF^u)$-adapted, continuous and strictly increasing, \cite[Proposition 7.9]{Kallenberg} yields that $Y$ is an $(\cFI^y)$-stopping time and $\cFI^Y = \dPF^{V}$. Now, let $\theta$ denote the shift operator. For $u\geq 0$,
 $$\ol{\nbeta}^{V+u} = \IPmag{\nbeta^{\Absa(V+u)}}^{-1}\scaleI\nbeta^{\Absa(V+u)}
    = \IPmag{\nbeta^{Y+h(u)}}^{-1}\scaleI\nbeta^{Y+h(u)},$$
    where $h(u) := \rho_{\theta_Y\bsa}(u)$. 
 Thus, $\theta_{V}\circ D = D\circ\theta_Y$. 
 Then
 \begin{align*}
  \BPr^0_{\ol\mu}\left[\eta\,\ol{f}\circ\theta_V\circ D\right]
  &= \BPr^0_{\ol\mu}\left[\eta\,f\circ\theta_Y\right]\\
  &= \BPr^0_{\ol\mu}\left[\eta\,\BPr^0_{\nbeta^Y}[f]\right]
  = \BPr^0_{\ol\mu}\left[\eta\,\BPr^0_{\nbeta^Y}\left[\ol{f}\circ D\right]\right]
  = \BPr^0_{\ol\mu}\left[\eta\,\ol\BPr^0_{\ol{\bsa}^V}\left[\ol f\,\right]\right],
 \end{align*}
 by the strong Markov property of the type-0 evolution, Proposition \ref{prop:strong_Markov}. The same argument works for the de-Poissonized type-1 evolution and the laws $\BPr^1_{\ol\mu}$.
\end{proof}

\begin{proof}[Proof of the Hunt assertion of Theorem \ref{thm:stationary}]
 As in the proof of Theorem \ref{thmtype0Hunt}, we must check four properties.
 
 (i) By Theorem \ref{thm:Lusin}, $(\IPspace,\dI)$ is Lusin. Since the mass map $\IPmag{\,\cdot\,}$ is continuous, the set $\IPspace_1$ is a Borel subset of this space, and is thus Lusin as well.
  
 (ii) From Proposition \ref{prop:type-1:cts_in_init_state} (and Theorem \ref{thmtype1Hunt}), the semi-group for the type-0 (resp.\ type-1) evolution is continuous in the initial state. Helland \cite[Theorem 2.6]{Helland78} shows that time-change operations of the sort considered here are continuous maps from Skorokhod space to itself. Thus, the semi-group for the de-Poissonized type-0 (resp.\ type-1) evolution is also continuous.
 
 (iii) Sample paths are continuous, as noted after the statement of Proposition \ref{prop:dePois:time_change}.
 
 (iv) Proposition \ref{prop:dePois:Markov} gives the required strong Markov property.
\end{proof}

To prove stationarity, we progressively strengthen the pseudo-stationarity results of Theorem \ref{thm:pseudostat}. Denote by $(\cF_{\rm mass}^y,y\ge 0)$ the right-continuous filtration on $\cCRI$ generated by $(\IPmag{\nbeta^y},\,y\ge 0)$.\pagebreak


\begin{lemma}\label{strongstat}
 Let $\mu$ denote the law of $B\scaleI\ol\beta$, where $B$ is some non-negative random variable independent of $\ol\beta\sim\PDIP[\alpha,0]$. Then for all $y\ge 0$, all $\cF_{\rm mass}^y$-measurable $\eta\colon \cCRI\to [0,\infty)$, and all measurable $h\colon \IPspace_1\to [0,\infty)$, we have
 \begin{equation*}
  \BPr^1_{\mu}\left[\eta\cf\{\nbeta^y\neq\emptyset\} h\left(\IPmag{\nbeta^y}^{-1}\scaleI\nbeta^y\right)\right] = \BPr^1_{\mu}\left[\eta \cf\{\nbeta^y\neq\emptyset\}\right]\EV\left[h\left(\ol\beta\right)\right].
 \end{equation*}
 The same assertion holds if we replace superscript `1's with `0's and take $\ol\beta\sim\PDIP[\alpha,\alpha]$.
\end{lemma}

\begin{proof}
 We begin with the type-1 assertion. Let $(\gamma^y,\,y\ge0)$ denote a type-1 evolution with $\gamma^0 = \ol\beta\sim\PDIP[\alpha,0]$, and suppose this is independent of $B$, with both defined on $(\Omega,\cA,\Pr)$. Then $(B\scaleI\gamma^{y/B},\,y\ge0)$ has law $\BPr^1_{\mu}$. By Theorem \ref{thm:pseudostat}, for $f_0$, $f_1\colon [0,\infty)\to [0,\infty)$ measurable,
  \begin{equation*}
  \begin{split}
   &\BPr^1_{\mu}\left[f_0(\IPmag{\nbeta^0})f_1\left(\IPmag{\nbeta^y}\right)\cf\{\nbeta^y\neq\emptyset\}h\left(\IPmag{\nbeta^y}^{-1}\scaleI\nbeta^y\right)\right]\\
   &= \EV\left[f_0(B)f_1\left(B\IPmag{\gamma^{y/B}}\right)\cf\{\gamma^{y/B}\neq\emptyset\}h\left(\IPmag{\gamma^{y/B}}^{-1}\scaleI\gamma^{y/B}\right)\right]\\
   &=\!\!\int_0^{\infty}\!\!\!f_0(m)\EV\!\left[\!f_1\!\left(\!m\IPmag{\gamma^{y/m}}\right)\!\cf\{\gamma^{y/m}\!\neq\!\emptyset\}h\!\left(\IPmag{\gamma^{y/m}}^{-1}\!\!\!\scaleI\!\gamma^{y/m}\!\right)\!\right]\!\Pr\{B\!\in\! dm\}\\
   &= \int_0^{\infty}f_0(m)\EV\left[f_1\left(m\IPmag{\gamma^{y/m}}\right)\cf\{\gamma^{y/m}\neq\emptyset\}\right]\EV[h(\ol\beta)]\Pr\{B\in dm\}\\
   &= \BPr^1_{\mu}\left[f_0\left(\IPmag{\nbeta^0}\right)f_1\left(\IPmag{\nbeta^y}\right)\cf\{\nbeta^y\neq\emptyset\}\right]\EV[h(\ol\beta)].
  \end{split}
  \end{equation*} 
  An inductive argument based on the Markov property of the type-1 evolution then says that 
  for $0<y_1<\cdots<y_{n}$ (writing $\ol{y}_j=y_j-y_1$, $j\in[n]$) and $f_0,\ldots,f_{n}\colon [0,\infty) \to [0,\infty)$ measurable,\vspace{-0.1cm} 
  \begin{equation*}
  \begin{split}
   &\BPr^1_\mu\Bigg[\prod_{j=0}^{n}f_j\left(\IPmag{\nbeta^{y_j}}\right)\cf\{\nbeta^{y_{n}}\neq\emptyset\}h\left(\IPmag{\nbeta^{y_{n}}}^{-1}\scaleI\nbeta^{y_n}\right)\Bigg]\\
    &=\BPr^1_\mu\Bigg[f_0\left(\IPmag{\nbeta^0}\right)\BPr^1_{\nbeta^{y_1}}\Bigg[\prod_{j=1}^{n}f_{j}\left(\IPmag{\nbeta^{\widebar{y}_{j}}}\right)\cf\{\nbeta^{\widebar{y}_{n}}\!\neq\!\emptyset\}h\left(\IPmag{\nbeta^{\widebar{y}_{n}}}^{-1}\!\scaleI\nbeta^{\widebar{y}_{n}}\right)\!\Bigg]\Bigg]\\
    &=\BPr^1_\mu\Bigg[f_0\left(\IPmag{\nbeta^0}\right)\BPr^1_{\nbeta^{y_1}}\Bigg[\prod_{j=1}^{n}f_{j}\left(\IPmag{\nbeta^{\widebar{y}_{j}}}\right)\cf\{\nbeta^{\widebar{y}_{n}}\neq\emptyset\}\Bigg]\EV[h(\ol\beta)]\Bigg]\\
    &= \BPr^1_\mu\Bigg[\prod_{j=0}^{n}f_j\left(\IPmag{\nbeta^{y_j}}\right)\cf\{\nbeta^{y_{n}}\neq\emptyset\}\Bigg]\EV[h(\ol\beta)].\vspace{-0.1cm}
  \end{split}
  \end{equation*}
  A monotone class theorem completes the proof. The same argument works for type 0.
\end{proof}

To do de-Poissonization, we will replace $y$ by a stopping time in the filtration $(\cF_{\rm mass}^y,\,y\ge 0)$, specifically the time-change stopping times $Y = \Absa(u)$.

\begin{theorem}[Strong pseudo-stationarity]\label{thm:pseudostat_strong}
 Let $\mu$ denote the law of $B\scaleI\ol\beta$, where $B$ is some non-negative random variable independent of $\ol\beta\sim\PDIP[\alpha,0]$. 
 Let $Y$ be an $(\cF_{\rm mass}^y,\,y\ge0)$-stopping time. Then for all $\cF_{\rm mass}^Y$-measurable $\eta\colon \cCRI\to [0,\infty)$ and all measurable $h\colon \IPspace_1\to [0,\infty)$, 
 \begin{equation*}
  \BPr^1_{\mu}\left[\eta\cf\{\nbeta^Y\neq\emptyset\} h\left(\IPmag{\nbeta^Y}^{-1}\scaleI\nbeta^Y\right)\right] = \BPr^1_{\mu}\left[\eta \cf\{\nbeta^Y\neq\emptyset\}\right]\EV\left[h\left(\ol\beta\right)\right].
 \end{equation*}
 The same assertion holds if we replace superscript `1's with `0's and take $\ol\beta\sim\PDIP[\alpha,\alpha]$.
\end{theorem}

\begin{proof}
 We begin with the type-1 assertion. We use the standard dyadic approximation of $Y$ by $Y_n=  2^{-n}\lfloor 2^nY+1\rfloor \wedge 2^n$ which eventually tends to $Y$ from above. Since $Y$ and $Y_n$ are $(\cF_{\rm mass}^y)$-stopping times, the random variable $\eta_k = \eta\cf\{Y_n = k2^{-n}\}$ is $\cF_{\rm mass}^{k2^{-n}}$-measurable for $k\in[2^{2n}-1]$. By Lemma \ref{strongstat},
 \begin{equation*}
 \begin{split}
  &\BPr^1_{\mu}\left[\eta\cf\left\{\nbeta^{Y_n}\neq\emptyset;\,Y_n=k2^{-n}\right\} h\left(\IPmag{\nbeta^{Y_n}}^{-1}\scaleI\nbeta^{Y_n}\right)\right]\\
    &\quad= \BPr^1_{\mu}\left[\eta_k\cf\left\{\nbeta^{k2^{-n}}\neq\emptyset\right\} h\left(\IPmag{\nbeta^{k2^{-n}}}^{-1}\scaleI\nbeta^{k2^{-n}}\right)\right]\\
    &\quad= \BPr^1_{\mu}\left[\eta_k\cf\left\{\nbeta^{k2^{-n}}\neq\emptyset\right\}\right]\EV[h(\ol\beta)]\\
    &\quad= \BPr^1_{\mu}\left[\eta \cf\left\{\nbeta^{Y_n}\neq\emptyset;\,Y_n=k2^{-n}\right\}\right]\EV[h(\ol\beta)].
 \end{split}
 \end{equation*}
 Summing over $k\in[2^{2n}-1]$ and letting $n\rightarrow\infty$, the continuity of $(\nbeta^y)$ and the observation that
 $\bigcup_{k\in [2^{2n}-1]}\left\{\nbeta^{Y_n}\neq\emptyset;\,Y_n=k2^{-n}\right\}$ increases to $\left\{\nbeta^{Y_n}\neq\emptyset\right\}$ complete the proof for type 1, first for continuous $h$, but then for measurable $h$ via the monotone class theorem. The type-0 argument is identical.
\end{proof}

\begin{proof}[Proof of the stationarity assertions of Theorem \ref{thm:stationary}]$\;$\\
 We apply Theorem \ref{thm:pseudostat_strong} to $\eta=1$ and the stopping times $Y=\Absa(u)$, which satisfy $\nbeta^Y\neq\emptyset$ a.s.. In the notation of Proposition \ref{prop:dePois:Markov},\vspace{-0.1cm}
 $$\ol\BPr^1_{\ol\mu}[h(\ol{\nbeta}^u)] = \BPr^1_{\mu}\left[\cf\left\{\nbeta^{\Absa(u)}\neq\emptyset\right\} h\left(\IPmag{\nbeta^{\Absa(u)}}^{-1}\scaleI\nbeta^{\Absa(u)}\right)\right] = \EV[h(\ol\beta)]\vspace{-0.1cm}$$
 for each $u>0$, as required. The same argument applies to type 0.
\end{proof}

\appendix

\section{Statistics of clades and excursions}
\label{sec:clade_stats}

In this section we prove Proposition \ref{prop:clade:stats}. 
More results in the vein of Proposition \ref{prop:clade:stats2} may be derived from these in a similar manner. Several of the following may be construed as descriptions of the It\^o excursion measure $\mSxc$ associated with $\bX$. We use notation $J:=J^++J^-$.

\begin{proposition}\label{prop:clade:stats2}
 \begin{enumerate}[label=(\roman*), ref=(\roman*)]
  \item $\displaystyle\! \mClade\big\{\len\!>\! x\big\}\!=\! \frac{(1+\alpha)x^{-\alpha/(1+\alpha)}}{\!(2^{\alpha}\Gamma(1\!+\!\alpha))^{1/(1\!+\!\alpha)}\Gamma(1/(1\!+\!\alpha))\!}.$\label{item:CS2:len}\vspace{-.1cm}
  \item $\displaystyle \mClade\big\{ J > z\big\} = \frac{1+\alpha}{\Gamma(1+\alpha)\Gamma(1-\alpha)2^{\alpha}}z^{-\alpha}.$\vspace{-0.0cm}\label{item:CS2:crossing}
  \item $\displaystyle \mClade\big\{ J^+ \in dy\;\big|\;m^0 = b\big\} = \frac{(b/2)^{1+\alpha}}{\Gamma(1+\alpha)}y^{-\alpha-2}e^{-b/2y}dy.$\label{item:CS2:over:mass}
  \item $\displaystyle \mClade\big\{ m^0 > b \big\} = \frac{1}{\Gamma(1-\alpha)}b^{-\alpha}.$\label{item:CS2:mass}
  \item $\displaystyle \mClade\big\{ J^+ > y \big\} = \frac{1}{\Gamma(1+\alpha)\Gamma(1-\alpha)2^{\alpha}}y^{-\alpha}.$\label{item:CS2:over}
  \item $\displaystyle \mClade\{ m^0 \leq b\;|\;J^+ = y\} = 1 - e^{-b/2y}.$\label{item:CS2:mass:over}
  \item $\displaystyle \mClade\big\{ \life^+ \leq z\;\big|\;J^+ = y\big\} = \cf\{z\geq y\}\left(\frac{z-y}{z}\right)^\alpha.$\label{item:CS2:max:over}
  \item $\displaystyle \mClade\big\{ \life^+ \leq z\;\big|\;m^0 = b\big\} = e^{-b/2z}.$\label{item:CS2:max:mass}
  \item $\displaystyle \mClade\big\{ \life^+ > z \big\} = \frac{1}{2^\alpha}z^{-\alpha}.$\label{item:CS2:max}\vspace{-.1cm}
  \item $\displaystyle \mClade\{ m^0 \in db\;|\;\life^+ \geq z\} = \frac{\alpha(2z)^{\alpha}}{\Gamma(1-\alpha)}b^{-\alpha-1}(1-e^{-b/2z})db.$\vspace{0.1cm}\label{item:CS2:mass:max}
  %
\end{enumerate}
\noindent Each of these identities also holds if we replace all superscripts `$+$' with `$-$'.
\end{proposition}

The equivalence when replacing `+'s with `$-$'s follows from the reversal property stated in (\ref{eq:clade_split:distrib}). Before proving these identities we note a pair of relevant properties of $\bX$. Recall that $(T^y,\,y\in\BR)$ denotes the first hitting times for $\bX$.

\begin{proposition}[Theorem VII.1 of \cite{BertoinLevy}]\label{prop:hitting_time:subord}
 The process $(T^{-y},\,y\geq 0)$ of hitting times is \Stable[1/(1+\alpha)] subordinator, and its Laplace exponent is the inverse $\psi^{-1}$ of the Laplace exponent of $\bX$:
 \begin{equation}
  \EV\left[ e^{-\theta T^{-y}} \right] = e^{-y\psi^{-1}(\theta)}, \quad \text{where} \quad \psi^{-1}(\theta) = \left(2^\alpha\Gamma(1+\alpha)\right)^{1/(1+\alpha}\theta^{1/(1+\alpha)}.\label{eq:hitting_time:Laplace}
 \end{equation}
\end{proposition}


\begin{proposition}\label{prop:inv_LT:subord}
 For each $y\in\BR$, the shifted inverse local time process $(\tau^y(s)-\tau^y(0),\,s\ge 0)$ is a \Stable[\alpha/(1+\alpha)] subordinator with Laplace exponent $\Phi(\theta) = (1+\alpha)\left(2^\alpha\Gamma(1+\alpha)\right)^{-1/(1+\alpha)}\theta^{\alpha/(1+\alpha)}$.
\end{proposition}

\begin{proof}
 It is straightforward to check that this is a \Stable[\alpha/(1+\alpha)] subordinator. For $f\colon \BR\to\BR$ bounded and measurable,
 \begin{equation*}
 \begin{split}
  &\int_{-\infty}^\infty f(y)\ell^y(t)dy = \int_0^t f(\bX(s))ds \\
    &\stackrel{d}{=} \int_0^t f\left( c^{1/(1+\alpha)}\bX\left(\frac{s}{c}\right) \right)ds = \int_0^{t/c} f\left(c^{1/(1+\alpha)}\bX (r) \right)cdr\\
  	&=\! \int_{-\infty}^\infty \!cf\left(c^{1/(1+\alpha)}y\right)\ell^y\left(\frac{t}{c}\right)dy =\! \int_{-\infty}^\infty\! f(z)c^{\alpha/(1+\alpha)}\ell^{c^{-1/(1+\alpha)}z}\left(\frac{t}{c}\right)dz.
 \end{split}
 \end{equation*}
 Hence $(c^{\alpha/(1+\alpha)}\ell^{c^{-1/(1+\alpha)}y}(t/c);\ t\ge 0,y\in\BR) \stackrel{d}{=}(\ell^y(t);\ t\ge 0,y\in\BR)$ and so
 \begin{equation*}
  \tau^0(s) \stackrel{d}{=} \inf\left\{t\ge 0\ :\ c^{\alpha/(1+\alpha)}\ell^0(t/c) > s\right\} = c\tau^0(s/c^{\alpha/(1+\alpha)})
 \end{equation*}
 satisfies \Stable[\alpha/(1+\alpha)] self-similarity. Thus, $\EV[e^{-\theta\tau^0(s)}] = e^{-sb\theta^{\alpha/(1+\alpha)}}$ for some $b\in(0,\infty)$. To identify $b$, we use the property that $\Pr\{\bX(t)\le 0\} = 1/(1+\alpha)$ for all $t>0$. This follows from an identity in \cite[p.\ 218]{BertoinLevy}. 
 Specifically, let $S_{\theta}$ be an $\ExpDist[\theta]$ random variable independent of $\bX$ and define $K_{\theta} := \EV[\ell^0(S_{\theta})]$. Then on the one hand,
 \begin{align*}
  K_{\theta} &= \int_0^\infty \Pr\{\ell^0(S_{\theta}) > s\}ds\\
    &   =   \int_0^\infty \Pr\{S_{\theta} > \tau^{0}(s)\}ds   =   \int_0^\infty \EV \left[e^{-\theta \tau^0(s)}\right]ds = \frac{1}{b}\theta^{-\alpha/(1+\alpha)}.
 \end{align*}
 On the other hand, by the strong Markov property of $\bX$ at the hitting time $T^y$, spatial homogeneity, and Proposition \ref{prop:hitting_time:subord},
 \begin{equation*}
  \EV[\ell^y(S_{\theta})] = \Pr\{T^y < S_{\theta}\}\EV[\ell^0(S_{\theta})] =e^{y\psi^{-1}(\theta)} K_{\theta} \qquad \text{for }y \leq 0.
 \end{equation*}
 By Fubini's theorem and the occupation density formula for local times
 \begin{equation*}
 \begin{split}
  \frac{1}{1+\alpha} &= \Pr\{\bX(S_{\theta}) \leq 0\} = \EV\left[ \int_0^\infty\theta e^{-\theta s}\cf\{\bX(s)\le 0\} ds \right]\\
  	&= \EV\left[ \int_0^\infty\theta^2e^{-\theta t}\int_0^t\cf\{\bX(s)\le 0\}dsdt \right]\\
 	&= \EV\left[ \int_0^\infty\theta^2e^{-\theta t}\int_{-\infty}^0\ell^y(t)dydt \right]
 	= \theta\int_{-\infty}^0\EV \left[ \ell^y(S_{\theta})\right]dy\\
    &= \theta K_\theta\int_{-\infty}^0 e^{y\psi^{-1}(\theta)}dy = \theta K_\theta\left(2^\alpha\Gamma(1+\alpha)\right)^{-1/(1+\alpha)}\theta^{-1/(1+\alpha)}.
  \end{split}
  \end{equation*}
  Substituting in for $K_{\theta}$, we get $1/(1+\alpha) = (1/b)\left(2^\alpha\Gamma(1+\alpha)\right)^{-1/(1+\alpha)}$; isolating $b$ gives the desired value.
\end{proof}

Recall that $\Pr\{\bX(t)\le 0\} = 1/(1+\alpha)$. On the other hand, the It\^o excursion measure $\mSxc$ of $\bX$, which we can obtain as push-forward of 
$\mBxc$ under the scaffolding construction \eqref{eq:scaffolding}, is invariant under increment reversal ($180^\circ$ rotation around the unique jump 
across 0), by (\ref{eq:clade:invariance}). This means that typically, the process has spent half its time positive up to the last zero but is 
likely to be found in the first half of a much longer excursion. 
%
%
We now derive the results in Proposition \ref{prop:clade:stats2}.

\begin{proof}[Proof of Proposition \ref{prop:clade:stats2}]
 \ref{item:CS2:len}. For convenience, we quote \eqref{eq:clade:invariance} here:
 \begin{equation}
  \mClade(\reverseH(N)\in\cdot\,) = \mClade \quad \text{and} \quad \mClade(c\scaleH N\in\cdot\,) = c^{-\alpha}\mClade.
   \label{clade:invariance}
 \end{equation}
 The latter of these formulas entails that $\mClade\{\len > x\} = Cx^{-\alpha/(1+\alpha)}$, for some constant $C$. As noted in Proposition \ref{prop:inv_LT:subord}, the inverse local time process $(\tau^0(s),\,s\ge 0)$ is a subordinator. Its L\'evy measure $\Pi$ equals $\mClade\{\len\in\cdot\,\}$. Then, recalling the identity $\Phi(\theta) = \int_0^\infty (1-e^{-\theta x})d\Pi(x)$, which may be read from \cite[Chapter 3]{BertoinLevy}, we obtain \ref{item:CS2:len} by solving for $C$ in
 $$(1\!+\!\alpha)\left(2^\alpha\Gamma(1\!+\!\alpha)\right)^{-1/(1+\alpha)}\theta^{\alpha/(1+\alpha)}
   =\!\int_0^\infty\!(1\!-\!e^{-\theta x})C\frac{\alpha}{1\!+\!\alpha}x^{-1-\alpha/(1+\alpha)}dx.$$
 
 
 \ref{item:CS2:crossing}. The length of a bi-clade $N$ equals the time until the first crossing of zero, plus the subsequent time until its scaffolding $X$ hits zero. Suppose $N\sim\mClade\{\,\cdot\;|\;J^+ = y\}$. Decomposing $N=N^-\concat N^+$ with the convention of a split central spindle, the scaffolding $X^+$ of the positive part is a \StableA\ first-passage path from $y$ down to zero independent of the negative part $X^-$, by the strong Markov property under $\mSxc$ at the crossing time $T_0^+$. Thus, by (\ref{clade:invariance}), if $N\sim\mClade\{\,\cdot\;|\;J^- = x\}$ then $X^-$ is the increment reversal of a \StableA\ first-passage path from $x$ down to zero, again independent of $X^+$. 
Appealing to the subordinator property noted in Proposition \ref{prop:inv_LT:subord}, under $\mClade\{\,\cdot\;|\;(J^-,J^+) = (x,y)\}$, the length $\len$ is distributed as the hitting time $T^{-x-y}$. Thus, $\mClade\{\len\in\cdot\;|\;J = z\}$ equals the law of $T^{-z}$. 
 
 
 It follows from the right-continuity of $\bX$ that $\mClade\{ J > z\}$ is finite for all $z>0$. By the scaling property \eqref{clade:invariance}, this equals $Cz^{-\alpha}$ for some constant $C$. It remains to determine the value of $C$. By Proposition \ref{prop:inv_LT:subord}, our argument for \ref{item:CS2:len} above, and Proposition \ref{prop:hitting_time:subord},
 \begin{align*}
  &\frac{1+\alpha}{2^{\alpha/(1+\alpha)}(\Gamma(1+\alpha))^{1/(1+\alpha)}}\theta^{\alpha/(1+\alpha)} = \mClade\left[1-e^{-\theta\len(N)}\right]\\ &\!=\!\!\int_0^\infty\!\!\left(\!1\!-\!\EV\left[e^{-\theta T^{-z}}\right]\!\right)\!\alpha Cz^{-1-\alpha}dz
   \!=\! C\Gamma(1\!-\!\alpha)\left(2^\alpha\Gamma(1\!+\!\alpha)\right)^{\alpha/(1+\alpha)}\theta^{\alpha/(1+\alpha)}.
 \end{align*}
 Solving for $C$ gives the desired result.
 
 \ref{item:CS2:over:mass}. Let $N$ have law $\mClade\{\,\cdot\;|\;m^0 = b\}$. Let $\hat f$ denote the leftmost spindle in $N^+$, i.e.\ the top part of the middle spindle of $N$. By (\ref{eq:clade_split:distrib}), $\hat f$ is a \BESQA\ started from $b$ and killed at zero. Then $J^+(N) = \life(\hat f)$; the law of the latter is specified in Lemma \ref{lem:BESQ:length}, which quotes \cite{GoinYor03}. In particular, this has distribution \InvGammaDist[1+\alpha,b/2].
 
 \ref{item:CS2:mass}. We know this formula up to a constant from \eqref{clade:invariance} and the $m^0$ entry in Table \ref{tbl:clade_scaling} on page \pageref{tbl:clade_scaling}. To obtain the constant, we appeal to \ref{item:CS2:crossing} and \ref{item:CS2:over:mass}. In particular, it follows from (\ref{eq:clade_split:indep})--(\ref{eq:clade_split:distrib}) that for $N$ with law $\mClade\{\,\cdot\;|\;m^0=b\}$, the over- and undershoot are i.i.d.\ with law \InvGammaDist[1+\alpha,\frac{b}{2}], as in \ref{item:CS2:over:mass} above. This allows us to express $\mClade\{ J\in dy\;|\;m^0 = b\}$ as
 \begin{equation*}
  \int_0^y \frac{b^{2+2\alpha}}{2^{2+2\alpha}(\Gamma(1+\alpha)^2)}\frac{1}{(zy-z^2)^{2+\alpha}}\exp\left( -\frac{b}{2z}-\frac{b}{2(y-z)} \right)dz.
 \end{equation*}
 Integrating this against the law $\mClade\{m^0\in db\} = Cb^{-1-\alpha}db$, we get
 \begin{equation*}
 \begin{split}
  &\mClade\{ J\in dy\}\\
  &= dy\!\int_0^\infty\!\! Cb^{-1-\alpha}\!\int_0^y \!\frac{b^{2+2\alpha}(zy\!-\!z^2)^{-2-\alpha}}{2^{2+2\alpha}(\Gamma(1\!+\!\alpha))^2}\exp\left( -\frac{b}{2z}-\frac{b}{2(y\!-\!z)} \right)\!dzdb\\
  	&= \frac{Cdy}{2^{2+2\alpha}(\Gamma(1+\alpha))^2} \int_0^y (zy-z^2)^{-2-\alpha} \Gamma\left(2+\alpha\right)\left(\frac{y}{2(zy-z^2)}\right)^{-2-\alpha}dz\\
    & = \frac{(1+\alpha)C}{2^\alpha\Gamma(1+\alpha)}y^{-1-\alpha}dy.
 \end{split}
 \end{equation*}
 Setting this equal to \ref{item:CS2:crossing} gives $C = \alpha/\Gamma(1-\alpha)$, as desired.
 
 \ref{item:CS2:over} and \ref{item:CS2:mass:over}. The former arises from integrating the product of formula \ref{item:CS2:over:mass} with the derivative of \ref{item:CS2:mass}. The latter is then computed by Bayes' rule.

 
 \ref{item:CS2:max:over}. 
 By the strong Markov property under $\mSxc$ at the crossing time $T_0^+$, this equals the probability that a \StableA\ process started from $y$ exits the interval $[0,z]$ out of the lower boundary first. 
 This is a standard calculation via scale functions \cite[Theorem VII.8]{BertoinLevy}, carried out for a spectrally negative stable process in 
\cite{Bertoin96}, from which the claimed result can be obtained by a sign change.
 
 \ref{item:CS2:max:mass}. This is computed by integrating the product of formulas \ref{item:CS2:over:mass} and \ref{item:CS2:max:over}, which can be reduced to a Gamma integral.
 
 \ref{item:CS2:max} and \ref{item:CS2:mass:max}. The former is computed by integrating the product of the derivative of formula \ref{item:CS2:mass} with \ref{item:CS2:max:mass}. The latter follows via Bayes' Rule.
\end{proof}

The remaining results in this section go towards proving Lemma \ref{lem:LMB} and thereby completing the proof of Proposition \ref{prop:type-1:transn}.
\begin{figure}
 \centering
 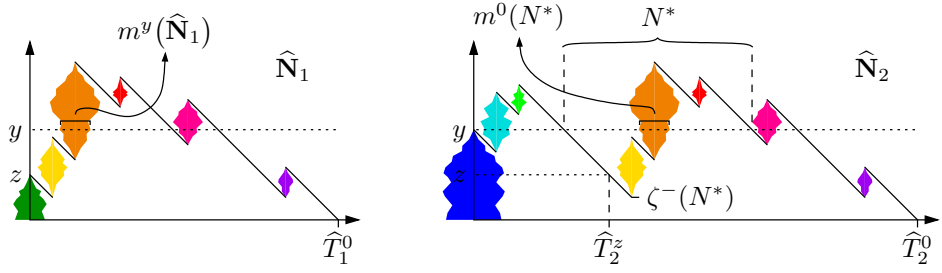
 \caption{Illustration of the coupling in the proof of Lemma \ref{lem:LMB_reversal}.\label{fig:LMB_reversal}}
\end{figure}
\begin{lemma}\label{lem:LMB_reversal}
 Recall \eqref{eq:LMB_def} defining $m^y(N)$ as the mass of the leftmost spindle at level $y$. Then for all $0<z<y$
 \begin{equation}
  \mClade^+\{m^y\!\in\! dc\,|\,J^+ \!= \!z,\,\life^+\! >\! y\} \!=\! \frac{\alpha 2^\alpha c^{-1-\alpha}}{\Gamma(1-\alpha)}\,\frac{e^{-c/2y}\! -\! e^{-c/2(y-z)}}{(y\!-\!z)^{-\alpha} \!- y^{-\alpha}} db.\label{eq:LMB_reversal}
 \end{equation}
\end{lemma}

\begin{proof}
 We prove this by showing that $\mClade^+\{m^y\in db\;|\;J^+ = z,\,\life^+ > y\}$ equals
 \begin{equation*}
 \begin{split}
  &\mClade^+\{m^0\in db\;|\;\life^+ \in (y-z,y)\}\\
  &\!\!=\!\frac{\mClade^+\{m^0\!\in\! db|\life^+\! >\! y\!-\!z\}\mClade^+\{\life^+\!>\!y\!-\!z\} \!-\! \mClade^+\{m^0\!\in\! db|\life^+ \!>\! y\}\mClade^+\{\life^+\!>\!y\}}{\mClade^+\{\life^+\!>\!y\!-\!z\}\!-\!\mClade^+\{\life^+\!>\!y\}}.
 \end{split}
 \end{equation*}
 The latter equals the right hand side of \eqref{eq:LMB_reversal} by Proposition \ref{prop:clade:stats2} \ref{item:CS2:max} and \ref{item:CS2:mass:max}. We prove this by a coupling construction, illustrated in Figure \ref{fig:LMB_reversal}.
 
 Fix $y>z>0$. Let $\whN_1\sim\mClade^+\{\,\cdot\;|\;J^+ = z\}$. As in Corollary \ref{cor:clade_law_given_over}, this may be expressed as $\whN_1 = \Dirac{0,\bff_1} + \Restrict{\bN_1}{[0,T^{-z}_1]}$, where $\zeta(\bff_1) = z$, $\bN_1$ is a \PRM[\Leb\otimes\mBxc], and $T^{-z}_1 = \wh T^0_1$ is the hitting time of $-z$ by the scaffolding $\bX_1$ associated with $\bN_1$ as in \eqref{eq:scaffolding}, or that of 0 by $\whX_1$. Correspondingly, let $\whN_2 = \Dirac{0,\bff_2} + \Restrict{\bN_2}{[0,T^{-y}_2]}$ have distribution $\mClade^+(\,\cdot\;|\;J^+ = y)$. Let $\wh T^z_2$ denote the time at which $\whX_2$ first hits $z$ and $T^{-y}_2 = \wh T^0_2$ the time at which it hits zero. Then 
$(\Restrict{\whN_2}{(\wh T^z_2,\wh T^0_2)},\Restrict{\whX_2}{[\wh T^z_2,\wh T^0_2]})$ shifted to turn into measure/scaffolding 
on $[0,\wh T^0_2-\wh T^z_2]$, denoted by $(\ShiftRestrict{\whN_2}{(\wh T^z_2,\wh T^0_2)},\ShiftRestrict{\whX_2}{[\wh T^z_2,\wh T^0_2]})$, will 
satisfy 
 \begin{equation}\label{eq:LMB_reversal_dist_eq}
  \left(\ShiftRestrict{\whN_2}{(\wh T^z_2,\wh T^0_2)},\ShiftRestrict{\whX_2}{[\wh T^z_2,\wh T^0_2]}\right)  \stackrel{d}{=} \left(\Restrict{\whN_1}{(0,\wh T^0_1)},\Restrict{\whX_1}{(0,\wh T^0_1)}\right), 
 \end{equation}
 which is a \StableA\ first passage from $z$ down to zero.
 
 The time $\wh T^z_2 = T^{z-y}_2$ occurs during the first bi-clade $N^*$ of $\bN_2$ about level $0$ that has $\life^-(N^*) \geq y-z$. Now, consider the event $A_2$ that $\whX_2$ returns up to level $y$ during the time interval $[\wh T^z_2,\wh T^0_2]$. Then $A_2 = \{\life^-(N^*) < y\}$. Thus, conditionally given $A_2$, the mass $m^0(N^*)$ is distributed according to $\mClade^+\{m^0\in \cdot\;|\;\life^- \in (y-z,y)\}$. 
 This is equal, via the time-reversal invariance of \eqref{eq:clade_split:distrib}, to the distribution $\mClade^+\{m^0\in \cdot\;|\;\life^+ \in (y-z,y)\}$.
 
 The quantity $m^0(N^*)$ and the event $A_2$ correspond, via \eqref{eq:LMB_reversal_dist_eq}, to the quantity $m^y(\whN_1)$ and the event $A_1$ that $\whX_1$ reaches level $y$ before reaching zero. Conditionally given $A_1$, the mass $m^y(\whN_1)$ is distributed according to $\mClade^+\{m^y\in \cdot\;|\;J^+ = z,\,\life^+ > y\}$. Thus, the two laws are equal, as desired.
\end{proof}

\begin{proof}[Proof of Lemma \ref{lem:LMB}]
 Fix $y>0$. We decompose the event $\{\life^+ > y\}$ into two components, based on whether $J^+ > y$:
 \begin{equation}
  \mClade^+\{m^y \!\in\! dc\,|\,m^0 \!=\! b,\life^+ \!>\! y\}
  	= \dfrac{ \left[\! \begin{array}{l}
  			\mClade\{m^y \!\in\! dc,y \!\in\! [J^+,\life^+)\,|\,m^0 \!=\! b\}\\
  			\ +\ \mClade\{m^y \!\in\! dc,y \!<\! J^+\,|\,m^0 \!=\! b\}
  		\end{array} \!\right] }
  		{ \mClade\{\life^+ > y\;|\;m^0 = b\} }.\label{eq:LMB_2_cases}
 \end{equation}
 The second summand in the above numerator describes the case in which the initial leftmost spindle of the clade survives to level $y$. Thus, this summand equals the density of the time $y$ distribution of a \BESQA\ started from $b$. We denote this by $q^{(-2\alpha)}_y(b,c)dc$. From \cite[Proposition 3; Equation (49)]{GoinYor03},
 \begin{equation*}
  q^{(-2\alpha)}_y(b,c) = q^{(4+2\alpha)}_y(c,b) = \frac{1}{2y}\left(\frac bc\right)^{(1+\alpha)/2}\!\!\!e^{-(b+c)/2y}I_{1+\alpha}\left(\frac{\sqrt{bc}}{y}\right) \ \text{for }c\!>\!0.
 \end{equation*}
 Hence, 
 \begin{equation}\label{eq:BESQ_transn_2}
  \mClade\{m^y \!\in\! dc,y \!<\! J^+|m^0 \!=\! b\} = \frac{1}{2y}\left(\frac bc\right)^{(1+\alpha)/2}\!\!\!e^{-(b+c)/2y}I_{1+\alpha}\left(\!\frac{\sqrt{bc}}{y}\right).
 \end{equation}
 
 It remains to evaluate the first summand in the numerator in \eqref{eq:LMB_2_cases}. Via Corollary \ref{cor:clade_law_given_over}, under the law $\mClade^+\{\,\cdot\;|\;J^+=z\}$, the variables $\life^+$ and $m^y$ for $y>z$ are independent of $m^0$. Thus,
 \begin{equation*}
 \begin{split}
  &\mClade^+\{m^y \in db,\,y \in [J^+,\life^+)\;|\;m^0 = a\}\\
  	&= \int_{z=0}^y \!\mClade^+\{m^y\!\in\! db|J^+ \!=\! z,\life^+ \!>\! y\} \mClade^+\{\life^+ \!>\! y|J^+ \!=\! z\} \mClade^+\{J^+\!\in\! dz|m^0 \!=\! a\}.\notag
 \end{split}
 \end{equation*}
 We have formulas for these three conditional laws in Lemma \ref{lem:LMB_reversal} and Proposition \ref{prop:clade:stats2} \ref{item:CS2:over:mass} and  \ref{item:CS2:max:over}. Plugging in, the above expression equals
 $$\left[\!\int_{z=0}^y\!\!\!\frac{\alpha 2^\alpha c^{-1-\alpha}}{\Gamma(1\!-\!\alpha)}\,\frac{e^{-c/2y} \!-\! e^{-c/2(y-z)}}{(y-z)^{-\alpha} - y^{-\alpha}} \left(\! 1 \!-\! \left(\frac{y\!-\!z}{y}\right)^{\!\alpha}\right) \frac{b^{1+\alpha}e^{-b/2z}}{2^{1+\alpha}\Gamma(1\!+\!\alpha)z^{2+\alpha}}dz\right]\! dc.$$
 Set $u= z/y$ and then $v = (1-u)/u$. Note that $1/(1-u) = 1+(1/v)$. Our integral becomes
 \begin{equation*}
  \frac{\alpha(b/c)^{1+\alpha}}{2y\Gamma(1\!-\!\alpha)\Gamma(1\!+\!\alpha)} \exp\!\left(\!-\frac{b\!+\!c}{2y}\right)\! \int_{0}^{\infty}\! \left(\!1\! - \!\exp\left(\!-\frac{c}{2y}\frac{1}{v}\right)\!\right)\exp\left(\!-\frac{b}{2y}v\right)v^\alpha dv.
 \end{equation*}
 We distribute the difference and compute the two resulting integrals separately:
 \begin{equation*}
  \int_{0}^{\infty} \exp\left(-\frac{b}{2y}v\right)v^\alpha dv = \Gamma(1+\alpha)\left( \frac{2y}{b} \right) ^{1+\alpha} 
 \end{equation*}
 and, via \cite[Exercise 34.13]{Sato}, 
 \begin{align*}
  &\int_{0}^{\infty} \exp\left(-\frac{c}{2y}\frac{1}{v} - \frac{b}{2y}v \right)v^\alpha dv\\ 
  &= \left(\frac{c}{b}\right)^{(1+\alpha)/2}\frac{\Gamma(1-\alpha)\Gamma(1+\alpha)}{\alpha}
		\left(I_{1+\alpha}\left(\frac{\sqrt{bc}}{y}\right)-I_{-1-\alpha}\left(\frac{\sqrt{bc}}{y}\right)\right).  
 \end{align*}
  Subtracting the second component from the first and multiplying in all constants, we find that $\mClade\{m^y \!\in\! dc,y\! \in\! [J^+,\life^+)|m^0 \!=\! b\}$ equals
 \begin{equation}\label{eq:mass_given_clade_but_not_spindle}
  \frac{1}{2y}\!\left(\frac{b}{c}\right)^{\!\!(1+\alpha)/2}\!\!\!\!\!e^{-(b+c)/2y}\!
	\left(\!\frac{\alpha}{\Gamma(1\!-\!\alpha)}\!\left(\!\frac{4y^2}{bc}\!\right)^{\!(1+\alpha)/2}\!\!\!\!\!-\!I_{1+\alpha}\!\left(\!\!\frac{\sqrt{bc}}{y}\right)\!-\!I_{-1-\alpha}\!\left(\!\!\frac{\sqrt{bc}}{y}\right)\!\right).
 \end{equation}
 

 Via Proposition \ref{prop:clade:stats2} \ref{item:CS2:max:mass}, the denominator in \eqref{eq:LMB_2_cases} is $1-e^{-b/2y}$. Adding \eqref{eq:BESQ_transn_2} to \eqref{eq:mass_given_clade_but_not_spindle} and dividing by $1-e^{-b/2y}$, the expression in \eqref{eq:LMB_2_cases} equals
 $$\frac{1}{2y}\left(\frac{b}{c}\right)^{(1+\alpha)/2}e^{-(b+c)/2y}\left(\frac{bc}{4y^2}\right)^{-(1+\alpha)/2}
		\sum_{n=1}^\infty\frac{1}{n!\Gamma(n-\alpha)}\left(\frac{bc}{4y^2}\right)^n,$$
 since the $n=0$ term in the $I_{-1-\alpha}$-series is 
 $$\left(\frac{bc}{4y^2}\right)^{-(1+\alpha)/2}\frac{1}{\Gamma(-\alpha)}=\left(\frac{bc}{4y^2}\right)^{-(1+\alpha)/2}\frac{-\alpha}{\Gamma(1-\alpha)}.\vspace{-0.8cm}$$
\end{proof}

This lemma completes the proof of Proposition \ref{prop:type-1:transn}.

\section{Markov property of type-0 evolutions}\label{appxtype0}

Let us recall some more terminology and results from \cite{IPPA}, as needed in the proof. It will be useful to concatenate the anti-clades in the 
point measure $\bF^{\le y}$ defined just before Proposition \ref{prop:bi-clade_PRM} into a point measure of spindles. Recall that an anti-clade includes a broken spindle, which has been cut off at the level corresponding to the upward passage of level $y\in\BR$ by the 
associated scaffolding. Let us use notation $f^{\le y}(z)=f(z)\cf\{z\in[0,y]\}$ and $f^{\ge y}(z)=f(y+z)\cf\{z\in[0,\infty)\}$ for the lower and
upper parts of a spindle $f$ broken at level $y$. More precisely, for $y>0$, the point measure $\bF^{\le y}$ excludes the incomplete anti-clade 
before the scaffolding $\bX$ first exceeds level $y$ at time $T^{\ge y}=\inf\{t\ge 0\colon\bX(t)\ge y\}$, via a jump marked by a spindle 
$\ff_{T^{\ge y}}$. We add this incomplete anti-clade as an additional point, for $y>0$,
$$\bF_0^{\le y}:=\delta\left(0,\fN|_{[0,T^{\ge y})}+\delta\left(T^{\ge y},\ff_{T^{\ge y}}^{\le y-\bX(T^{\ge y}-)}\right)\right)+\bF^{\le y}.$$ 
Also set $\bF_0^{\le y}:=\bF^{\le y}$ for $y\le 0$. Then we define the concatenation
$$\cutoffL{y}{\bN} := \Concat_{\text{points }(s,N^-_s)\text{ of }\bF_0^{\leq y}}N^-_s.$$
Let $y>0$. For a clade $\bN_U=\delta(0,\ff)+\bN|_{[0,T^{-\life(\ff)}]}$ and general $\bN_\beta=\ConcatIL_{U\in\beta}\bN_U$, $\beta\in\IPspace$,  
obtained by concatenating independent clades $\bN_U$, $U\in\beta$, we generalize this definition by restriction and concatenation: 
$$\cutoffL{y}{\bN_U}:=\delta\left(0,\ff^{\le y}\right)
					 +\Concat_{\text{points }(s,N^-_s)\text{ of }\bF_0^{\leq y-\life(\ff)}\colon s<\ell^{y-\life(\ff)}(T^{-\life(\ff)})}N_s^-$$
and $\cutoffL{y}{\bN_\beta}:=\Concat_{U\in\beta}\cutoffL{y}{\bN_U}$. See \cite[Lemma 3.41]{IPPA}. We similarly define $\cutoffG{y}{\bN_\beta}$. 
Let $\Pr^1_\beta$ denote the distribution of $\bN_\beta$; this was denoted by $\Pr^{(\alpha)}_\beta$ in \cite[Definition 5.2]{IPPA}, but here we prefer the superscript `1' as a reference to type-1 evolutions and we suppress the $\alpha$. 
Denote by $\cF^y_\fN$ the $\sigma$-algebra generated by $\cutoffL{y}{\bN}$ and by $\cF^y_{\fN_\beta}$ the $\sigma$-algebra generated by 
$\cutoffL{y}{\bN_\beta}$. These $\sigma$-algebras form filtrations as $y$ varies, and the Markov property of type-1 evolutions can be expressed in 
terms of cut-off processes and these filtrations:

\begin{proposition}[Proposition 4.24 of \cite{IPPA}]\label{prop:PRM:Fy-_Fy+} Let $\bN\sim\PRM(\Leb\otimes\mBxc)$. Let $T$ be a stopping time in the
  natural time filtration $(\cF_\bN(t),t\ge 0)$, with $\cF_\bN(t)$ generated by $\bN|_{[0,t]}$ such that $S^0 := \ell^0(T)$ is measurable in 
  the level-0 $\sigma$-algeba $\cF^0_\bN$, and such that $\bX < 0$ on the time interval $(\tau^0(S^0-),T)$. Let $\td\bN:=\bN|_{[0,T)}$ and $\td\nbeta^y=\skewer(y,\td\bN,\td\bX)$,
  $y\ge 0$. Then for each $y\geq 0$, the point measure $\cutoffG{y}{\bN|_{[0,T)}}$ is 
  conditionally independent of $\cF^{y}_{\bN}$ given $\td\nbeta^y$, with the regular conditional distribution (r.c.d.) $\Pr^1_{\td\nbeta^y}$.
\end{proposition}

\begin{proposition}[Proposition 5.6 of \cite{IPPA}]\label{prop:type-1:Fy-_Fy+} 
 Let $\bN_\beta\sim\Pr^1_\beta$ and $\nbeta^y=\skewer(y,\bN_\beta,\fX_\beta)$, $y\ge 0$. For $y>0$, the point process $\cutoffG{y}{\bN_\beta}$ is conditionally independent of $\cF^y_{\bN_\beta}$ given $\nbeta^y$, with r.c.d.\ $\Pr^1_{\nbeta^y}$.
\end{proposition}

We also note a natural property of the skewer map to be unaffected by the cutoffs.

\begin{lemma}[Lemma 4.23(ii) of \cite{IPPA}]\label{lem:cutoff_skewer}
For $\bN$ a \PRM[\Leb\otimes\mBxc], it is a.s.\ the case that for every $t\ge 0$,
  	\begin{equation*}
  	 \skewer\left(y,\Restrict{\bN}{[0,t]}\right) = \left\{\begin{array}{ll}
  		\skewer\Big(y,\cutoffL{z}{\restrict{\bN}{[0,t]}}\Big)		& \text{if }y \leq z,\\[6pt]
  		\skewer\left(y-z,\cutoffG{z}{\restrict{\bN}{[0,t]}}\right)	& \text{if }y \geq z.
  	 \end{array}\right.
  	\end{equation*}
  	The same holds for $\bN_{\beta}$, for any $\beta\in\HIPspace$, with $z>0$.
\end{lemma}

\begin{proof}[Proof of Proposition \ref{prop:type-0:simple_Markov}]
 Take $\beta\in\IPspace$, $0\leq u_1 < \cdots < u_n \leq y$, and $0\leq v_1 < \cdots < v_m$. Suppose $\eta(\nbeta^z,\,z\ge0) = \eta'(\nbeta^{u_j},\,j\in[n])$ and $f(\nbeta^z,\,z\ge0) = f'(\nbeta^{v_j},\,j\in[m])$. We will show that in this case,
 \begin{align}
  &\BPr^0_{\beta}\big[\eta'(\nbeta^{u_j},\,j\in[n])f'(\nbeta^{y+v_j},\,j\in[m])\big]\nonumber\\ 
  &= \BPr^0_{\beta}\big[\eta'(\nbeta^{u_j},\,j\in[n])\BPr^0_{\nbeta^y}\big[f'(\nbeta^{v_j},\,j\in[m])\big]\big].
  \label{eq:type-0:simple_Markov_1}
 \end{align}
 Indeed, this will suffice to prove the proposition: we can extend to general $\eta$ and $f$ by a monotone class theorem, and we generalize the equation from $\BPr^0_{\beta}$ to $\BPr^0_{\mu}$ by mixing.
 
 For $x>0$, set $\cev\bN_x := \restrict{\bN}{[0,T^{-x})}$, similar to the point processes discussed in \eqref{eq:type-0:consistency}. Let $\bN_{\beta}\sim\Pr^1_{\beta}$, independent of $\cev\bN_x$. For the purpose of this argument, we define $\Pr^0_{x,\beta}$ to be the distribution of $\bN_{x,\beta} := \cev\bN_x\concat\bN_{\beta}$. We work towards a type-0 version of Proposition \ref{prop:type-1:Fy-_Fy+}.
 
 Take $z>u_n$. Set $\vecc\nbeta^y := \skewer(y,\bN_\beta)$, $\cev\nbeta^y_{z+y} := \skewer\big(-z,\cev\bN_{z+y}\big)$, and $\nbeta^y_{z+y} := \cev\nbeta^y_{z+y}\concat\vecc\nbeta^y$. Let $\cev\bN_{z} := \restrict{\cev\bN_{z+y}}{[0,T^{-z})}$ and $\cev\bN^z_y := \restrict{\cev\bN_{z+y}}{[T^{-z},T^{-z-y})}$, where $T^{-z}$ is the hitting time of $-z$ in $\cev\bX_{z+y}$. By the strong Markov property of $\bN$, these components are independent. Recall the cutoff processes above. In our setting,
 \begin{equation}\label{eq:type-0:Fy-+_decomp}
 \begin{array}{r@{\ =\ }r@{\;\concat\;}l}
  \cutoffG{-z}{\bN_{z+y,\beta}} &\displaystyle\cev\bN_{z} \concat \cutoffGB{0}{\cev\bN^z_y} & \cutoffG{y}{\bN_{\beta}},\\[.2cm]
  \text{and} \quad
  \cutoffL{-z}{\bN_{z+y,\beta}} &\displaystyle					  \cutoffLB{0}{\cev\bN^z_y} & \cutoffL{y}{\bN_{\beta}}.
 \end{array}
 \end{equation}
 By Proposition \ref{prop:PRM:Fy-_Fy+}, $\cutoffG{0}{\cev\bN^z_y}$ is conditionally independent of $\cutoffLB{0}{\cev\bN^z_y}$ given $\cev\nbeta^y_{z+y}$, with r.c.d.\ $\Pr^1_{\cev\nbeta^y_{z+y}}$. 
 Analogously, substituting Proposition \ref{prop:type-1:Fy-_Fy+} 
 for Proposition \ref{prop:PRM:Fy-_Fy+}
, we see that $\cutoffG{y}{\bN_{\beta}}$ is conditionally independent of $\cutoffL{y}{\bN_{\beta}}$ given $\vecc\nbeta^y$, with r.c.d.\ $\Pr^1_{\vecc\nbeta^y}$. Thus, since $\big(\cev\bN_{z},\cev\bN^z_y,\bN_{\beta}\big)$ is an independent triple, $\cutoffG{-z}{\bN_{z+y,\beta}}$ is therefore conditionally independent of $\cutoffL{-z}{\bN_{{z+y},\beta}}$ given $\nbeta^y_{z+y}$, with r.c.d.\ $\Pr^0_{z,\nbeta^y_{z+y}}$. Now, \eqref{eq:type-0:simple_Markov_1} follows by Lemma \ref{lem:cutoff_skewer}.
\end{proof}


\bibliographystyle{abbrv}
\bibliography{AldousDiffusion}
\end{document}